\documentclass[a4paper,12pt,reqno]{amsart}

\usepackage{amsmath,amsthm,amssymb,amstext}
\usepackage{amsfonts}
\usepackage{mathrsfs}
\usepackage{bm}
\usepackage{mathtools}
\usepackage{dsfont}
\usepackage{esint}       
\usepackage{enumitem}
\usepackage{comment}
\usepackage{microtype}
\usepackage{natbib}
\usepackage{xcolor}

\usepackage[
    left=2.3cm,
    right=2.3cm,
    top=2.5cm,
    bottom=2.5cm
]{geometry}

\usepackage{hyperref}

\hypersetup{
    pdfborder={0 0 0},
    colorlinks=true,
    linkcolor=blue,
    citecolor=red,
    urlcolor=magenta
}

\allowdisplaybreaks
\theoremstyle{plain}

\newtheorem{theorem}{Theorem}[section]
\newtheorem{lemma}[theorem]{Lemma}
\newtheorem{proposition}[theorem]{Proposition}
\newtheorem{corollary}[theorem]{Corollary}

\theoremstyle{definition}

\theoremstyle{remark}

\newtheorem{remark}{Remark}

\theoremstyle{plain}

\numberwithin{equation}{section}

\newcommand{\R}{\mathbb{R}}

\newcommand{\rnn}{\mathbb{R}^{N}}

\title[
Critical Sobolev Inequalities for a Partially Fractional Grushin Operator
]{
Critical Sobolev Inequalities and Extremals for a Partially Fractional Grushin Operator
}

\author[S.~Malhotra]{S.~Malhotra}
\address[S.~Malhotra]{Department of Mathematics, Indian Institute of Technology Delhi, 110016, India}
\email{shammi22malhotra@gmail.com}
\author[L.~Sharma]{L.~Sharma}
\address[L.~Sharma]{Department of Mathematics, Indian Institute of Science Education and Research (IISER) Pune, India}
\email{lovelesh.sharma@iiserpune.ac.in}

\author[K.~Sreenadh]{K.~Sreenadh}
\address[K.~Sreenadh]{Department of Mathematics, Indian Institute of Technology Delhi, 110016, India}
\email{sreenadh@maths.iitd.ac.in}



\begin{document}

\begin{abstract}
We study the partially fractional Grushin-type operator
\[
\mathcal L_{\alpha,\beta}u=(-\Delta_x)^\alpha u-(1+\beta)^2|x|^{2\beta}\Delta_y u,
\qquad (x,y)\in\R^n\times\R^m,\quad 0<\alpha<1,\ \beta\ge 0,
\]
which couples fractional diffusion in $x$ with degenerate local diffusion in $y$. Working with the anisotropic dilations $(x,y)\mapsto(\rho x,\rho^{\alpha+\beta}y)$ and an adapted quasi-distance, we prove a uniform fractional Grushin-Poincar\'e inequality on anisotropic cylinders and, via a Riesz-type potential built from a telescoping sum over dyadic cylinders, a global Sobolev inequality at the sharp critical exponent $2^*_{\alpha,\beta}=2Q_{\alpha,\beta}/(Q_{\alpha,\beta}-2\alpha)$, where $Q_{\alpha,\beta}=n+(\alpha+\beta)m$. We then establish a concentration-compactness principle for the associated energy space. For $\beta>0$ the weight $|x|^{2\beta}$ destroys translation invariance in $x$, and we show that concentration is confined to the degenerate hyperplane $ \Sigma= \{(0,y):y\in\R^m\}$; this localization yields the existence of extremals for the sharp constant $S_{\alpha,\beta}(n,m)$.
\end{abstract}
\maketitle
\tableofcontents

\maketitle
\section{Introduction and main results}

Degenerate elliptic operators of Grushin type, $\Delta_x+|x|^{2\beta}\Delta_y$ on $\R^n\times\R^m$, sit at the intersection of sub-Riemannian geometry and degenerate PDE theory. The study of this operator was initiated in the works of Tricomi \cite{tricomi_grushin}, Baouendi \cite{baouendi_phd_thesis} and Grushin \cite{grushin_origin}. Since the foundational work of Franchi and Lanconelli in a series of papers given by \cite{franchi_lanconelli_origin_lambda_operators, franchi_lanconelli_lambda_op_holder_cty, franchi_lanconelli_lambda_op_harnack}, and the systematic development of the associated function spaces, Poincar\'e inequalities, and Sobolev inequalities with the properties of their extremals by Franchi, Lu, and Wheeden \citep{franchi_lu_wheeden}, Franchi, Guti\'errez, and Wheeden \cite{franchi_gutierrez_wheeden}, Garofalo and Nhieu \cite{garofalo_nhieu}, and Monti \cite{monti_sobolev_grushin}, such operators have become a standard model for degenerate diffusions whose underlying geometry is genuinely anisotropic: the vector fields $\partial_{x_i}$ and $|x|^\beta\partial_{y_j}$ generate a Carnot-Carath\'eodory structure in which the homogeneous dimension exceeds the topological dimension $n+m$, when \(\beta>0\). \\

In parallel, fractional operators such as $(-\Delta)^\alpha$ and their associated Caffarelli-Silvestre extensions \cite{caffarelli_silvestre} have produced a rich nonlocal analogue of the classical Sobolev theory. The underlying function spaces go back to Aronszajn \cite{aronszajn_frac}, Gagliardo \cite{gagliardo_frac_space}, and Slobodeckij \cite{slobodeckii1958generalized} and their modern treatment is surveyed by Di Nezza, Palatucci, and Valdinoci \cite{di_nezza_palatucci_valdinoci}. Poincar\'e-type inequalities adapted to the Gagliardo seminorm were established by Bourgain, Brezis, and Mironescu \cite{bourgain_brezis_mironescu} and by Ponce \cite{ponce_poincare}, while the sharp form of the fractional Sobolev inequality, together with the classification of its extremals, follows from the Hardy-Littlewood-Sobolev inequality of Lieb \cite{lieb_sharp} and was extended to all fractional orders by Cotsiolis and Tavoularis \cite{cotsiolis_tavoularis}. The corresponding failure of compactness at the critical exponent is described by the fractional concentration-compactness and profile-decomposition results of Palatucci and Pisante \cite{Pala-Pis-Cal}, and the resulting variational theory has been developed by Servadei and Valdinoci \cite{servadei_valdinoci}, among many others. \\

The present paper is concerned with the operator obtained by combining these two mechanisms: a genuine fractional Laplacian in the $x$-variable and a possibly degenerate, weighted Laplacian in the $y$-variable,
\begin{equation}
\label{eq:operator-intro}
\mathcal L_{\alpha,\beta}u
:=(-\Delta_x)^\alpha u-(1+\beta)^2|x|^{2\beta}\Delta_y u,
\qquad 0<\alpha<1,\ \beta\ge 0.
\end{equation}

\medskip Since only the $x$-diffusion in \eqref{eq:operator-intro} is nonlocal, we refer to $\mathcal L_{\alpha,\beta}$ as a \emph{partially fractional} Grushin operator. Its structure differs both from that of mixed local-nonlocal operators such as $-\Delta+(-\Delta)^s$, in which the two parts act in all variables, and from that of the fractional powers $\big(-\Delta_x-|x|^{2\beta}\Delta_y\big)^s$ of the Grushin operator \cite{balhara2019}. Note that for $\beta=0$ the operator $\mathcal L_{\alpha,0}=(-\Delta_x)^\alpha-\Delta_y$ is no longer degenerate, but it is still partially fractional.

\medskip Partially fractional-classical operators of this type are motivated by anisotropic models arising in degenerate kinetic and dispersive equations, including equations of Boussinesq-Ostrovsky (BO) and Benjamin-Ono-Zakharov-Kuznetsov(BO-ZK) type. They provide a tractable framework in which the effect of fractional diffusion in one set of variables may be isolated from that of anisotropic degeneracy in the complementary variables. We refer to D'Ambrosio \citep{dambrosio}, Loiudice \citep{loiudice}, and Monti and Morbidelli \citep{monti_morbidelli} for related anisotropic and fractional Sobolev theory, and to Bogdan, Sztonyk, and Knopova \citep{bogdan_sztonyk_knopova} and Bogdan and Jakubowski \citep{bogdan_jakubowski} for potential-theoretic aspects of mixed fractional operators. Despite these developments, the critical Sobolev theory for the partially fractional with \(\mathcal{L}_{\alpha,\beta}\) does not seem to be available. The main difficulty is that this energy is neither the full fractional Sobolev energy in \(\mathbb R^{n+m}\) nor the standard local Grushin energy. It contains a fractional seminorm only in the \(x\)-variables and a weighted local gradient only in the \(y\)-variables. Consequently, the classical fractional Sobolev inequality and the standard Grushin-Sobolev inequality cannot be applied directly. This is the main reason for developing a Poincaré inequality, a critical Sobolev inequality, and a concentration-compactness principle adapted to this partially fractional structure.

\medskip Blockwise fractional Grushin-type operators have previously been studied by Jleli and Samet \cite{jleliSamet2017} in connection with Liouville-type and nonexistence results for nonlinear elliptic equations and systems. The present work is different in both its operator and its goal: we study the critical Sobolev theory associated with the mixed energy above, rather than fractional powers of the classical Grushin operator \cite{balhara2019}. In particular, only the \(x\)-component is fractional, while the \(y\)-component remains a weighted local diffusion.  More precisely, we establish a uniform fractional Grushin-Poincaré inequality on anisotropic cylinders and a global Sobolev inequality at the scaling-critical exponent. We also develop a concentration-compactness principle and, for \(\beta>0\), show that the concentration measure is supported on the degenerate hyperplane \(\{(0,y):y\in\mathbb R^m\}\). Finally, we establish the existence of extremals for the corresponding optimal Sobolev constant.\\

We first look for dilations under which $\mathcal L_{\alpha,\beta}$ is homogeneous. Under isotropic dilations of $\R^{n+m}$ the two terms of \eqref{eq:operator-intro} scale differently, so the dilations have to be anisotropic. Accordingly, we seek a dilation of the form $\delta_\lambda(x,y)=(\lambda x,\lambda^\sigma y)$, $\lambda>0$, in which the exponent $\sigma>0$ is chosen so that $(-\Delta_x)^\alpha$ and $|x|^{2\beta}\Delta_y$ have the same homogeneity. Let $u_\lambda(x,y):=u(\lambda x,\lambda^\sigma y)$. Since $(-\Delta_x)^\alpha$ is homogeneous of order $2\alpha$,
\[
(-\Delta_x)^\alpha u_\lambda(x,y)
=\lambda^{2\alpha}\big[(-\Delta_x)^\alpha u\big](\lambda x,\lambda^\sigma y),
\]
whereas $\nabla_y u_\lambda(x,y)=\lambda^\sigma(\nabla_y u)(\lambda x,\lambda^\sigma y)$ gives $\Delta_y u_\lambda(x,y)=\lambda^{2\sigma}(\Delta_y u)(\lambda x,\lambda^\sigma y)$, and hence
\[
|x|^{2\beta}\Delta_y u_\lambda(x,y)
=\lambda^{2\sigma-2\beta}|\lambda x|^{2\beta}(\Delta_y u)(\lambda
x,\lambda^\sigma y),
\]
where we have used $|x|^{2\beta}=\lambda^{-2\beta}|\lambda x|^{2\beta}$. Requiring the two terms to carry the same power of $\lambda$ forces $2\sigma-2\beta=2\alpha$, that is, $\sigma=\alpha+\beta$. The natural dilation is therefore
\begin{equation}
\label{eq:dilation}
\delta_\lambda(x,y)=(\lambda x,\lambda^{\alpha+\beta}y),
\qquad\lambda>0,
\end{equation}
and with this choice
\[
\mathcal L_{\alpha,\beta}u_\lambda(x,y)
=\lambda^{2\alpha}\big[\mathcal L_{\alpha,\beta}u\big](\lambda
x,\lambda^{\alpha+\beta}y),
\]
so that $\mathcal L_{\alpha,\beta}$ is homogeneous of order $2\alpha$ with respect to $\delta_\lambda$.

The dilations \eqref{eq:dilation} induce in turn a homogeneous dimension associated with the operator, namely
\begin{equation}
\label{eq:Q}
Q_{\alpha,\beta}:=n+(\alpha+\beta)m.
\end{equation}
Indeed, the map $x\mapsto\lambda x$ on $\R^n$ transforms $dx$ into $\lambda^n\,dx$, while $y\mapsto\lambda^{\alpha+\beta}y$ on $\R^m$ transforms $dy$ into $\lambda^{(\alpha+\beta)m}\,dy$, whence \( d\big(\delta_\lambda(x,y)\big)=\lambda^{Q_{\alpha,\beta}}\,dx\,dy . \)

\begin{remark}
Formally setting $\alpha=1$ in \eqref{eq:Q} gives $n+(1+\beta)m$, the homogeneous dimension of the Grushin operator $\Delta_x+|x|^{2\beta}\Delta_y$.
\end{remark}

We next use this anisotropic scaling to determine the critical exponent. To this end we consider the energy attached to $\mathcal L_{\alpha,\beta}$,
\begin{equation}
\begin{aligned}
\label{eq:energy}
\mathcal E_{\alpha,\beta}(u)
:&=\frac12\int_{\R^m}\!\int_{\R^n}\!\int_{\R^n}
\frac{\big| u(x,y)-u(\xi,y)\big|^2}{|x-\xi|^{n+2\alpha}}
\,dx\,d\xi\,dy \\
& \qquad +(1+\beta)^2\int_{\R^{n+m}}|x|^{2\beta}|\nabla_y u|^2\,dz =\int_{\R^{n+m}}|\nabla_{\mathcal L}u|^2\,dz,
\end{aligned}
    \end{equation}
where \(\nabla_{\mathcal L}\) is introduced in \eqref{eq:grushin-gradient} below.

Let $u_\lambda(x,y):=\lambda^s u(\lambda x,\lambda^{\alpha+\beta}y)$ with $s\in\R$. By the homogeneity of the fractional seminorm,
\[
[u_\lambda]_{x,\alpha}^2
=\lambda^{2s+2\alpha-Q_{\alpha,\beta}}[u]_{x,\alpha}^2 .
\]
As for the weighted gradient term, we have $\nabla_y u_\lambda(x,y)=\lambda^{s+\alpha+\beta}(\nabla_y u)(\lambda x,\lambda^{\alpha+\beta}y)$, so that, setting $X=\lambda x$ and $Y=\lambda^{\alpha+\beta}y$ (whence $dx\,dy=\lambda^{-Q_{\alpha,\beta}}dX\,dY$ and $|x|^{2\beta}=\lambda^{-2\beta}|X|^{2\beta}$),
\[
\int_{\R^{n+m}}|x|^{2\beta}|\nabla_y u_\lambda(x,y)|^2\,dx\,dy
=\lambda^{2s+2\alpha-Q_{\alpha,\beta}}
\int_{\R^{n+m}}|X|^{2\beta}|\nabla_Y u(X,Y)|^2\,dX\,dY .
\]
Consequently $\mathcal E_{\alpha,\beta}(u_\lambda) =\lambda^{2s+2\alpha-Q_{\alpha,\beta}}\mathcal E_{\alpha,\beta}(u)$, which is scale invariant precisely when $s=(Q_{\alpha,\beta}-2\alpha)/2$; the energy-invariant scaling is therefore $u_\lambda(x,y)=\lambda^{(Q_{\alpha,\beta}-2\alpha)/2}u(\lambda x,\lambda^{\alpha+\beta}y)$.

We are thus led to the inequality
\begin{equation}\label{ineq:start_sobolev_grushin}
    \left( \int_{\R^{n+m}} |u|^p \, dz \right)^{1/p}
    \leq C \left( \int_{\R^{n+m}}|\nabla_{\mathcal L}u|^2\,dz \right)^{1/2},
\end{equation}
for some $p>1$ and every $u \in C_c^{\infty}(\R^{n+m})$. With $u_\lambda$ as above and the same change of variables,
\[
\|u_\lambda\|_{L^p(\R^{n+m})}^p
=\lambda^{\frac{p(Q_{\alpha,\beta}-2\alpha)}2-Q_{\alpha,\beta}}
\|u\|_{L^p(\R^{n+m})}^p .
\]
Since the right-hand side of \eqref{ineq:start_sobolev_grushin} is unchanged under the energy-invariant scaling, the inequality can hold with a constant $C$ independent of $u$ only if the $L^p$-norm is likewise scale invariant, that is, only if $\tfrac{p(Q_{\alpha,\beta}-2\alpha)}2=Q_{\alpha,\beta}$. This forces $p=2Q_{\alpha,\beta}/(Q_{\alpha,\beta}-2\alpha)$ and $Q_{\alpha,\beta}>2\alpha$. The critical Sobolev-Grushin exponent associated with the energy of $\mathcal L_{\alpha,\beta}$ is accordingly
\begin{equation*}
2^*_{\alpha,\beta}=\frac{2Q_{\alpha,\beta}}{Q_{\alpha,\beta}-2\alpha},
\qquad Q_{\alpha,\beta}>2\alpha .
\end{equation*}
The proof of an inequality of the form \eqref{ineq:start_sobolev_grushin} passes through a Poincar\'e inequality on bounded sets, which requires a notion of distance compatible with the anisotropic dilations \eqref{eq:dilation}. To this end we observe that the quasi-norm
\begin{equation}\label{eq:quasi-norm}
\rho_{\alpha,\beta}(x,y):=\big(|x|^{2(\alpha+\beta)}+|y|^2\big)^{\frac{1}{2(\alpha+\beta)}}
\end{equation}
is homogeneous with respect to $\delta_\lambda$, in the sense that $\rho_{\alpha,\beta}(\delta_\lambda(x,y))=\lambda\,\rho_{\alpha,\beta}(x,y)$, and that $\rho_{\alpha,\beta}(x,y)\asymp|x|+|y|^{\frac1{\alpha+\beta}}$. The associated quasi-distance is
\[
d_{\alpha,\beta}\big((x,y),(\xi,\eta)\big)
:=\rho_{\alpha,\beta}(x-\xi,y-\eta)
=\big(|x-\xi|^{2(\alpha+\beta)}+|y-\eta|^2\big)^{\frac1{2(\alpha+\beta)}} .
\]
For $z_0=(x_0,y_0)\in\R^{n+m}$, we define the $d_{\alpha,\beta}$-ball
\[
B_d(z_0,r):=\{z\in\R^{n+m}:d_{\alpha,\beta}(z,z_0)<r\},
\]
which satisfies, by the homogeneity of $\rho_{\alpha,\beta}$,
\[
\delta_\lambda\big(B_d(z_0,r)\big)=B_d\big(\delta_\lambda(z_0),\lambda r\big) .
\]
For $r>0$ we write
\[
B_r^n(x_0):=\{x\in\R^n:|x-x_0|<r\},
\qquad
B_{r^{\alpha+\beta}}^m(y_0):=\{y\in\R^m:|y-y_0|<r^{\alpha+\beta}\},
\]
and we call
\[
\mathcal{C}_r(z_0):=B_r^n(x_0)\times B_{r^{\alpha+\beta}}^m(y_0)
=\{(x,y)\in\R^n\times\R^m:|x-x_0|<r,\ |y-y_0|<r^{\alpha+\beta}\}
\]
the \emph{anisotropic Grushin cylinder} of radius $r$ centred at $z_0$. Since $\rho_{\alpha,\beta}(x,y)\asymp|x|+|y|^{1/(\alpha+\beta)}$, the families $\{B_d(z_0,r)\}_{r>0}$ and $\{\mathcal{C}_r(z_0)\}_{r>0}$ are mutually comparable, up to dimensional constants and a bounded dilation of the radius. When $z_0=0$, we abbreviate $B_r:=B_d(0,r)$ and $\mathcal{C}_r:=\mathcal{C}_r(0)$.

With these basic properties at hand, we begin with the fractional Grushin-Poincar\'e inequality. Our proof of the Poincar\'e inequality estimates the $x$-part and the $y$-part of the energy separately, which is natural on product sets, and the balls $B_d(z_0,r)$ are not products. We therefore work instead on the Grushin cylinders $\mathcal{C}_r(z_0)$, which carry a natural product decomposition, and obtain the following result.

\begin{theorem}[Uniform fractional Grushin-Poincar\'e inequality]
\label{thm:uniform-poincare} Let $0<\alpha<1$, $\beta\ge0$ and $\gamma:=\alpha+\beta$. There exists a constant $C=C(n,m,\alpha,\beta)>0$ such that for every $z_0=(x_0,y_0)\in\R^n\times\R^m$, every $r>0$, and  {every $u\in\mathcal H_{\alpha,\beta}\big(\mathcal{C}_r(z_0)\big)$ (the space is defined in \eqref{eq:H-space} below),}

\begin{equation}
\label{eq:poincare-uniform}
\int_{\mathcal{C}_r(z_0)}|u-u_{\mathcal{C}_r(z_0)}|^2\,dx\,dy
\le C\,r^{2\alpha}\left([u]^2_{x,\alpha;\mathcal{C}_r(z_0)}
+\int_{\mathcal{C}_r(z_0)}|x|^{2\beta}|\nabla_yu|^2\,dx\,dy\right),
\end{equation}
where \(u_{\mathcal{C}_r(z_0)}\)  and \([u]^2_{x,\alpha;\mathcal{C}_r(z_0)}\) are the average value of \(u\) and fractional seminorm on \(\mathcal{C}_r(z_0)\) respectively.
\end{theorem}
 {Inequality \eqref{eq:poincare-uniform} is stable under the action of Lipschitz functions, see Lemma~\ref{lem:lip-stable} below, which is stated once the space $\mathcal H_{\alpha,\beta}$ has been introduced.} Before turning to the critical Sobolev embedding we first prove the Sobolev embedding for all subcritical exponents on anisotropic cylinders. The proof rests on a Marcinkiewicz-type telescoping and dyadic-chaining argument, in the spirit of \cite[Theorem~20.5]{heinonen_nonlinear_potential_book}, together with a Vitali covering lemma for the doubling quasi-metric measure space $(\R^{n+m},d_{\alpha,\beta},\mu)$. The chain starts from a cylinder that lies outside the support of $u$, so that the average of $u$ over it vanishes. For this to be possible from every starting point, the offset $M(\gamma)$ and the radii of the cylinders in the chain must depend on $\gamma=\alpha+\beta$; see Proposition~\ref{prop:M-offset}.

\begin{theorem}[Subcritical Sobolev embedding]
\label{thm:Sobolev_subcritical} Let $0<\alpha<1$, $\beta\ge0$, $\gamma=\alpha+\beta$, and $\varepsilon\in(0,\alpha)$. Then there exists a constant $C>0$, depending only on $\varepsilon,n,m,\alpha,\beta$, such that for every $0<\delta<2^*_{\alpha,\beta,\varepsilon}:=2Q_{\alpha,\beta}/(Q_{\alpha,\beta}-2\alpha+2\varepsilon)$,
\begin{equation}
\label{eq:subcritical-sobolev}
\int_{\mathcal{C}(z_0,r)}|u|^\delta\,dz
\le C\,r^{\varepsilon \delta}\,\mu\big(\mathcal{C}(z_0,r)\big)^{1-\frac{\delta}{2^*_{\alpha,\beta,\varepsilon}}}
\left(\int_{\mathcal{C}(z_0,r)}|\nabla_{\mathcal L}u|^2\,dz\right)^{\delta/2}
\end{equation}
whenever $\mathcal{C}(z_0,r)$ is an anisotropic Grushin cylinder and $u:\mathcal{C}(z_0,r)\to\R$ is a Lipschitz function with compact support in $\mathcal{C}(z_0,r)$. The constant $C$ blows up as $\delta\to2^*_{\alpha,\beta,\varepsilon}$.
\end{theorem}

The critical endpoint $\delta=2^*_{\alpha,\beta}$, which is excluded from Theorem~\ref{thm:Sobolev_subcritical}, is recovered in Section~\ref{sec:riesz} by a different approach, inspired by Section~5 of \cite{hajlasz_koskela}. We construct a Riesz-type potential $\mathcal J^{\mathcal{C}_R}_{\alpha,2}$, defined in \eqref{eq:reisz_potential}, as a telescoping sum over dyadic anisotropic cylinders; we show that it dominates the oscillation of $u$ pointwise (Theorem~\ref{thm:representation}); and we establish for it a sharp weak-type $(2,2^*_{\alpha,\beta})$ estimate by balancing it optimally against the Hardy-Littlewood maximal function (Theorem~\ref{thm:weaktype-riesz}). In order to upgrade this weak-type estimate to a strong $L^{2^*_{\alpha,\beta}}$ bound, the authors of \cite{hajlasz_koskela} exploit the fact that the relevant energy pair satisfies the truncation property for the Poincar\'e inequality. The nonlocal term in our operator obstructs this lifting, and we circumvent the difficulty by estimating the nonlocal contribution directly. Combining the two steps yields the sharp critical exponent on Grushin cylinders, as recorded in Theorem~\ref{thm:sobolev-cylinders}. The inequality is then extended to the whole space $\R^{n+m}$.

\begin{theorem}[Global Sobolev inequality on $\R^{n+m}$]
\label{thm:global} There exists $C=C(\alpha,\beta,n,m)>0$ such that for every $u\in\mathcal D^{\alpha,2}_\beta(\R^{n+m})$
\begin{equation}
\label{eq:global-sobolev}
\left(\int_{\R^{n+m}}|u|^{2^*_{\alpha,\beta}}\,dx\,dy\right)^{1/2^*_{\alpha,\beta}}
\le C \left(\int_{\R^{n+m}}|\nabla_{\mathcal L}u|^2\,dz\right)^{1/2},
\end{equation}
where $\mathcal D^{\alpha,2}_\beta(\R^{n+m})$ denotes the completion of $C_c^\infty(\R^{n+m})$ with respect to the norm $\|\cdot\|_{\mathcal D^{\alpha,2}_\beta}=\mathcal E_{\alpha,\beta}(\cdot)^{1/2}$, as in \eqref{eq:D-space}.
\end{theorem}
As a corollary of Theorem~\ref{thm:global} we obtain a Rellich-Kondrachov type compact embedding for the space $\mathcal D^{\alpha,2}_\beta(\R^{n+m})$.

\begin{theorem}[Rellich-Kondrachov on bounded sets]
\label{lem:rellich} Let $0<\alpha<1$ and $\beta\ge0$, and let $\{u_k\}$ be a bounded sequence in $\mathcal D^{\alpha,2}_\beta(\R^{n+m})$. Then for every bounded measurable set $\mathcal A\subset\R^{n+m}$ and every $1\le\delta<2^*_{\alpha,\beta}$ some subsequence of $\{u_k\}$ converges in $L^\delta(\mathcal A)$; that is,
\[
\mathcal D^{\alpha,2}_\beta(\R^{n+m})\hookrightarrow\hookrightarrow
L^{\delta}_{\mathrm{loc}}(\R^{n+m}),
\qquad 1\le\delta<2^*_{\alpha,\beta}.
\]
\end{theorem}
Finally, the Grushin-Sobolev inequality \eqref{eq:global-sobolev} allows us to define the associated optimal constant,
\[
S_{\alpha,\beta}(n,m):=\inf_{\substack{u\in\mathcal
D^{\alpha,2}_\beta(\R^{n+m})\\u\ne0}}
\frac{\mathcal E_{\alpha,\beta}(u)}
{\big(\int_{\R^{n+m}}|u|^{2^*_{\alpha,\beta}}\,dx\,dy\big)^{2/2^*_{\alpha,\beta}}}
=\inf_{\substack{u\in\mathcal D^{\alpha,2}_\beta(\R^{n+m})\\
\|u\|_{L^{2^*_{\alpha,\beta}}}=1}}\mathcal E_{\alpha,\beta}(u).
\]
The quotient defining $S_{\alpha,\beta}(n,m)$ is invariant under the translation-dilation group
\begin{equation*}
u_{\rho,\eta}(x,y):=\rho^{\frac{Q_{\alpha,\beta}-2\alpha}2}u(\rho
x,\rho^{\alpha+\beta}y+\eta),
\qquad\rho>0,\ \eta\in\R^m,
\end{equation*}
in the sense that $\|u_{\rho,\eta}\|_{L^{2^*_{\alpha,\beta}}(\R^{n+m})} =\|u\|_{L^{2^*_{\alpha,\beta}}(\R^{n+m})}$ and $\mathcal E_{\alpha,\beta}(u_{\rho,\eta})=\mathcal E_{\alpha,\beta}(u)$. We emphasise that the quotient is \emph{not} invariant under translations in the $x$-variable for $\beta > 0$, owing to the presence of the weight $|x|^{2\beta}$.

We shall prove that this infimum is attained for every $n,m\ge1$, by means of the concentration-compactness method, in the following form.
\begin{theorem}
[Concentration-compactness principle] \label{lem:ccp-mixed-fractional} Let  {$0< \alpha < 1$, $\beta > 0$} and $\{u_k\}\subset\mathcal D^{\alpha,2}_\beta(\R^{n+m})$ be a bounded sequence such that $u_k\rightharpoonup u$ weakly in $\mathcal D^{\alpha,2}_\beta(\R^{n+m})$ and $u_k\to u$ a.e. in $\R^{n+m}$. Set $w_k:=u_k-u$, and assume that, up to a subsequence,
\[
|w_k|^{2^*_{\alpha,\beta}}\,dx\,dy\stackrel{*}{\rightharpoonup}\nu
\quad\text{and}\quad
\big[D_{x,\alpha}w_k+(1+\beta)^2|x|^{2\beta}|\nabla_yw_k|^2\big]\,dx\,dy
\stackrel{*}{\rightharpoonup}\mu
\]
weakly-$*$ in $\mathcal M(\R^{n+m})$, for some nonnegative finite measures $\mu$ and $\nu$, where $D_{x,\alpha}$ is the fractional $x$-energy \eqref{eq:density}. Define the concentration masses at \(\infty\) by
\begin{align*}
\mu_\infty&:=\lim_{R\to\infty}\limsup_{k\to\infty}
\int_{\R^{n+m}\setminus B_R}\big[D_{x,\alpha}w_k
+(1+\beta)^2|x|^{2\beta}|\nabla_yw_k|^2\big]\,dx\,dy, \\
\nu_\infty&:=\lim_{R\to\infty}\limsup_{k\to\infty}
\int_{\R^{n+m}\setminus B_R}|w_k|^{2^*_{\alpha,\beta}}\,dx\,dy,
\end{align*}
where $B_R=B_d(0,R)$. Then
\begin{equation}
\label{minq}
\|\nu\|^{2/2^*_{\alpha,\beta}}\le S_{\alpha,\beta}(n,m)^{-1}\|\mu\|,
\qquad
\nu_\infty^{2/2^*_{\alpha,\beta}}\le S_{\alpha,\beta}(n,m)^{-1}\mu_\infty,
\end{equation}
and
\begin{align}
\limsup_{k\to\infty}\int_{\R^{n+m}}|u_k|^{2^*_{\alpha,\beta}}\,dx\,dy
&=\int_{\R^{n+m}}|u|^{2^*_{\alpha,\beta}}\,dx\,dy+\|\nu\|+\nu_\infty,
\label{eq:Lp-balance}\\
\limsup_{k\to\infty}\mathcal E_{\alpha,\beta}(u_k)
&=\mathcal E_{\alpha,\beta}(u)+\|\mu\|+\mu_\infty .
\label{eq:E-balance}
\end{align}
If, in addition, $u=0$, $\nu\ne0$, and equality holds in the first inequality of \eqref{minq}, then $\nu$ is concentrated at a single point: there exist $z_0\in\R^{n+m}$ and $\nu_0>0$ such that $\nu=\nu_0\delta_{z_0}$. Finally, $\nu$ is concentrated on the degenerate hyperplane
\[
\Sigma:=\{(0,y):y\in\R^m\},
\qquad\text{that is,}\qquad
\nu\big(\R^{n+m}\setminus\Sigma\big)=0 ;
\]
in particular, $\nu$ cannot be a Dirac mass at a point $z_0=(x_0,y_0)$ with $x_0\ne0$.
\end{theorem}

Theorem~\ref{lem:ccp-mixed-fractional} is the analogue for $\mathcal E_{\alpha,\beta}$ of the concentration-compactness principle of Lions \cite{lions1985}; see \cite{Pala-Pis-Cal} for the fractional Laplacian. Together with the translation-dilation invariance \eqref{scaling} of the Sobolev quotient, it supplies one of the key compactness tools needed to prove that $S_{\alpha,\beta}(n,m)$ is attained. The decisive point is that concentration of $\nu$ can occur only on the degenerate hyperplane $\Sigma=\{(0,y):y\in\R^m\}$. This leads to the following existence result.
\begin{theorem}[Existence of minimizers for $S_{\alpha,\beta}(n,m)$]
\label{thm:existence-minimizer-alphabeta} Let  {$0< \alpha < 1$, $\beta > 0$} and $\{v_k\}\subset\mathcal D^{\alpha,2}_\beta(\R^{n+m})$ be a minimizing sequence for $S_{\alpha,\beta}(n,m)$, that is,
\begin{equation*}
\|v_k\|_{L^{2^*_{\alpha,\beta}}} = 1,
\qquad
\mathcal E_{\alpha,\beta}(v_k)\longrightarrow S_{\alpha,\beta}(n,m).
\end{equation*}
Then there exists a sequence $\{(e_k,\rho_k)\}\subset\R^m\times(0,\infty)$ such that the rescaled sequence
\begin{equation}
\label{eq:rescaled-sequence-alphabeta}
v_k^{e_k,\rho_k}(x,y):=\rho_k^{\frac{Q_{\alpha,\beta}-2\alpha}{2}}\,
v_k\big(\rho_k x,\,\rho_k^{\alpha+\beta}y+e_k\big)
\end{equation}
admits a subsequence converging strongly in $\mathcal D^{\alpha,2}_\beta(\R^{n+m})$ to some $v\in\mathcal D^{\alpha,2}_\beta(\R^{n+m})$ with
\[
\|v\|_{L^{2^*_{\alpha,\beta}}} = 1,
\qquad
\mathcal E_{\alpha,\beta}(v) = S_{\alpha,\beta}(n,m).
\]
In particular, a minimizer for $S_{\alpha,\beta}(n,m)$ exists.
\end{theorem}
As a direct consequence of Theorem \ref{thm:existence-minimizer-alphabeta}, we obtain the Euler-Lagrange equation satisfied by extremals.
\begin{corollary}
Let \(0<\alpha<1\), \(\beta>0\), and let \(v\in \mathcal{D}_{\beta}^{\alpha,2}(\mathbb R^{n+m})\) be a minimizer for \(S_{\alpha,\beta}(n,m)\), normalized by \( \|v\|_{L^{2^*_{\alpha,\beta}}(\mathbb R^{n+m})}=1. \) Then \(v\) satisfies
\[
\mathcal B_{\alpha,\beta}(v,\varphi)
=
S_{\alpha,\beta}(n,m)
\int_{\mathbb R^{n+m}}
|v|^{2^*_{\alpha,\beta}-2}v\varphi\,dx\,dy,~~~\quad \text{for every}~~ \varphi\in C_c^\infty(\mathbb R^{n+m}).
\]
\end{corollary}
\begin{remark}
The preceding corollary gives the Euler-Lagrange equation satisfied by extremals of the sharp Sobolev quotient. However, it does not provide an explicit closed form for such extremals. The classification and possible explicit representation of extremals for \(S_{\alpha,\beta}(n,m)\) remain open questions.
\end{remark}
\begin{remark}
The rescaling \eqref{eq:rescaled-sequence-alphabeta} involves translations in the $y$-variable only, in accordance with the invariance group \eqref{scaling}. In the non-degenerate case $\beta=0$ the weight $|x|^{2\beta}$ disappears and $\mathcal L_{\alpha,0}=(-\Delta_x)^\alpha-\Delta_y$; the energy $\mathcal E_{\alpha,0}$ is then invariant under translations in both variables, and \eqref{eq:rescaled-sequence-alphabeta} may be replaced by the larger family $\rho_k^{(Q_{\alpha,0}-2\alpha)/2}v_k(\rho_k x+f_k,\rho_k^{\alpha}y+e_k)$ with $(f_k,e_k)\in\R^n\times\R^m$. For $\beta>0$ this extension is not feasible, and it is precisely the loss of $x$-translation invariance that confines the concentration set to the hyperplane $\Sigma$.
\end{remark}
\begin{remark}

The coefficient \((1+\beta)^2\) in \eqref{eq:operator-intro} is only a normalization constant. More generally, one may replace it by any fixed constant \(\Theta>0\). All functional inequalities and compactness results obtained in this paper remain valid for such an operator, with constants depending also on \(\Theta\). The choice \(\Theta=(1+\beta)^2\) is made only for consistency with a standard normalization in the Grushin literature. Since the present paper does not use an explicit fundamental solution, this coefficient plays no essential role in the proofs.
\end{remark}

\noindent\textbf{Plan of the paper.} In Section~\ref{sec:setting} we introduce the energy spaces associated with $\mathcal L_{\alpha,\beta}$ and prove some of its associated properties. Section~\ref{sec:poincare} is devoted to the uniform fractional Grushin-Poincar\'e inequality and to the resulting subcritical Sobolev embeddings. In Section~\ref{sec:riesz} we establish the critical Sobolev inequality by means of a dyadic Riesz-type potential, and derive the local compactness properties required in the sequel. Finally, Section~\ref{sec:optimal} develops the concentration-compactness principle, localizes concentration on the degenerate hyperplane, and proves the existence of extremals for the optimal Sobolev constant.

\vspace{0.5cm}

\textbf{Notations.} We end this section by collecting the notations used throughout the paper.

\begin{itemize}[leftmargin=2.9cm,itemsep=1pt]
\item[$z=(x,y)$] a point of $\R^{n}\times\R^{m}$; $N:=n+m$ is the Euclidean (topological) dimension, and $dz=dx\,dy$ denotes Lebesgue measure on $\R^{n+m}$, also written $\mu$.

\item[$\rho_{\alpha,\beta}$, $d_{\alpha,\beta}$] the homogeneous quasi-norm $\rho_{\alpha,\beta}(x,y)=\big(|x|^{2\gamma}+|y|^{2}\big)^{1/2\gamma}$ of~\eqref{eq:quasi-norm} and the induced quasi-distance $d_{\alpha,\beta}(z,z')=\rho_{\alpha,\beta}(z-z')$; $B_d(z_0,r)$ is the corresponding ball, written $B_r$ when $z_0=0$.
\item[$\mathcal C(z_0,r)$] the anisotropic Grushin cylinder $B^{n}_{r}(x_0)\times B^{m}_{r^{\gamma}}(y_0)$, written interchangeably as $\mathcal C_{r}(z_0)$, and as $\mathcal C_{r}$ when $z_0=0$.
\item[$V_R(z_0)$] the collar of $\mathcal C_R(z_0)$, defined in~\eqref{eq:VR}.
\item[{$[u]_{x,\alpha;\Omega}$}] the fractional $x$-seminorm \eqref{eq:seminorm}, abbreviated $[u]_{x,\alpha}$ when $\Omega=\R^{n+m}$, and $D_{x,\alpha}u$ the corresponding pointwise energy density \eqref{eq:density}.
\item[$\nabla_{\mathcal L}u$] the Grushin gradient \eqref{eq:grushin-gradient}, and $\mathcal E_{\alpha,\beta}(u)$ the associated energy \eqref{eq:energy}, with polarization $\mathcal B_{\alpha,\beta}$ \eqref{eq:bilinear}.
\item[$\mathcal D^{\alpha,2}_\beta$] the homogeneous energy space \eqref{eq:D-space}, and $S_{\alpha,\beta}(n,m)$ its sharp Sobolev constant (Section~\ref{sec:optimal}).
\item[$\mathcal J^{\mathcal C_R}_{\alpha,2}$] the dyadic Riesz-type potential \eqref{eq:reisz_potential}, and $\mathcal M_{V_R}$ the associated restricted Hardy-Littlewood maximal operator.
\item[$\mathcal M(\Omega)$] the space of finite signed Radon measures on $\Omega$; $\stackrel{*}{\rightharpoonup}$ denotes weak-$*$ convergence in $\mathcal M(\Omega)$, and $\delta_{z_0}$ the Dirac mass at $z_0$.
\item[$\fint_E f$] the average integral $\fint_E f:=|E|^{-1}\int_E f$ of an integrable (or nonnegative measurable) function $f$ over a measurable set $E$ of finite positive measure, where $E$ may lie in $\R^{n}$, $\R^{m}$ or $\R^{n+m}$ and $|E|$ is its Lebesgue measure in the corresponding dimension; in particular $u_\Omega=\fint_\Omega u\,dz$.
\item[$C,C_1,C_2,\dots$] positive constants depending only on $n,m,\alpha,\beta$, unless further dependence is indicated, and allowed to change from line to line.
\end{itemize}

\section{Functional Framework and Anisotropic geometry}
\label{sec:setting}
In this section we introduce the functional framework associated with \(\mathcal L_{\alpha,\beta}\). We first define the energy spaces and establish the density of smooth compactly supported functions.

We define the partially fractional Grushin-type operator
\[
\mathcal L_{\alpha,\beta}u
=(-\Delta_x)^\alpha u-(1+\beta)^2|x|^{2\beta}\Delta_y u
\quad\text{in }\R^n\times\R^m,
\]
where $0<\alpha<1$, $\beta\ge 0$,  and $(-\Delta_x)^\alpha$ denotes the fractional Laplacian acting in the $x$-variable.

We introduce the fractional seminorm in the $x$-variable. For $u$ measurable and $\Omega\subseteq\R^{n+m}$ measurable we set
\begin{equation}
\label{eq:seminorm}
[u]_{x,\alpha;\Omega}^2
:=\frac12\int_{\R^m}\int_{\R^n}\int_{\R^n}
\frac{|u(x,y)-u(\xi,y)|^2}{|x-\xi|^{n+2\alpha}}\,
\chi_\Omega(x,y)\,\chi_\Omega(\xi,y)\,dx\,d\xi\,dy ,
\end{equation}
the integrand being understood to vanish whenever either point lies outside $\Omega$, so that $[u]_{x,\alpha;\Omega}$ depends only on $u|_\Omega$; we abbreviate $[u]_{x,\alpha}:=[u]_{x,\alpha;\R^{n+m}}$. Correspondingly, for $u$ measurable on all of $\R^{n+m}$, the pointwise energy density in the $x$-variable is
\begin{equation}
\label{eq:density}
D_{x,\alpha}u(x,y):=\frac12\int_{\R^n}
\frac{|u(x,y)-u(\xi,y)|^2}{|x-\xi|^{n+2\alpha}}\,d\xi ,
\qquad\text{so that}\qquad
[u]_{x,\alpha}^2=\int_{\R^{n+m}}D_{x,\alpha}u(x,y)\,dx\,dy ,
\end{equation}
and, throughout the paper, we write with a slight abuse of notation
\begin{equation}
\label{eq:grushin-gradient}
|\nabla_{\mathcal L}u(x,y)|^2
:=D_{x,\alpha}u(x,y)+(1+\beta)^2|x|^{2\beta}|\nabla_y u(x,y)|^2 ,
\end{equation}
calling $\nabla_{\mathcal L}u$ the \emph{Grushin gradient} of $u$; the symbol $\nabla_{\mathcal L}u$ is used only through the quantity $|\nabla_{\mathcal L}u|^2$, which is the energy density of $\mathcal L_{\alpha,\beta}$ and which integrates to $\mathcal E_{\alpha,\beta}(u)$ as in \eqref{eq:energy}.

Consequently, for every measurable $\Omega\subseteq\R^{n+m}$ and every $u$ defined on all of $\R^{n+m}$,
\begin{equation}
\label{eq:gradient-on-omega}
\big\|\nabla_{\mathcal L}u\big\|_{L^2(\Omega)}^2
=\int_\Omega D_{x,\alpha}u\,dz
+(1+\beta)^2\int_\Omega|x|^{2\beta}|\nabla_yu|^2\,dz
\ \ge\
[u]_{x,\alpha;\Omega}^2
+(1+\beta)^2\int_\Omega|x|^{2\beta}|\nabla_yu|^2\,dz ,
\end{equation}
and in general the inequality is strict. The difference between the two sides is computed in the next lemma.

\begin{lemma}
\label{lem:energy-comparison} Let $u$ be measurable on $\R^{n+m}$ and let $\Omega\subseteq\R^{n+m}$ be measurable. Writing $\Omega_y:=\{x\in\R^n:(x,y)\in\Omega\}$, one has
\[
\int_\Omega D_{x,\alpha}u\,dz
=[u]^2_{x,\alpha;\Omega}
+\frac12\int_{\R^m}\int_{\Omega_y}\int_{\R^n\setminus\Omega_y}
\frac{|u(x,y)-u(\xi,y)|^2}{|x-\xi|^{n+2\alpha}}\,d\xi\,dx\,dy .
\]
In particular
\[
[u]^2_{x,\alpha;\Omega}\ \le\ \int_\Omega D_{x,\alpha}u\,dz\ \le\ [u]^2_{x,\alpha},
\]
and both inequalities are strict in general.
\end{lemma}
\begin{proof}
The proof follows easily from Tonelli's theorem by writing \(\int_\Omega D_{x,\alpha}u\,dz\) slice-wise in \(y\) and splitting \(\mathbb R^n=\Omega_y\cup(\mathbb R^n\setminus\Omega_y)\).
\end{proof}

\begin{remark}
The quantity $\int_\Omega D_{x,\alpha}u\,dz$ depends on the values of $u$ outside $\Omega$ and is therefore not intrinsic to $\Omega$, whereas $[u]_{x,\alpha;\Omega}$ is. For this reason the local spaces below are defined through $[\,\cdot\,]_{x,\alpha;\Omega}$, while $\nabla_{\mathcal L}u$ and the quantity $\|\nabla_{\mathcal L}u\|_{L^2(\Omega)}$ are used only for functions defined on all of $\R^{n+m}$; Lemma~\ref{lem:energy-comparison} relates the two. When $u$ has compact support inside $\Omega$, however, a bound in the opposite direction is available, as the following corollary shows.
\end{remark}

\begin{corollary}
\label{cor:support-comparison} Let $\Omega\subseteq\R^{n+m}$ be open and let $u$ be measurable on $\R^{n+m}$ with $\operatorname{supp}u$ a compact subset of $\Omega$. Then
\[
[u]^2_{x,\alpha}\ \le\ 2\int_\Omega D_{x,\alpha}u\,dz,
\qquad\text{and consequently}\qquad
\int_{\R^{n+m}}|\nabla_{\mathcal L}u|^2\,dz\ \le\ 2\int_\Omega|\nabla_{\mathcal L}u|^2\,dz .
\]
\end{corollary}
\begin{proof}
Since \(u=0\) on \(\mathbb R^{n+m}\setminus\Omega\), fractional seminorm is controlled by twice the energy density integrated over \(\Omega\), by Lemma~\ref{lem:energy-comparison}; the \(y\)-gradient term is supported in \(\Omega\), which gives the second estimate.
\end{proof}

\smallskip We can now introduce the natural energy space associated with $\mathcal L_{\alpha,\beta}$, directly on an arbitrary domain. For $\Omega\subseteq\R^n\times\R^m$ we write $\Omega^{x}:=\{y\in\R^m:(x,y)\in\Omega\}$ for the slice at fixed $x\in\R^n$, and $\Omega_{y}:=\{x\in\R^n:(x,y)\in\Omega\}$ for the slice at fixed $y\in\R^m$, and we set
\begin{equation}
\label{eq:H-space}
\begin{aligned}
\mathcal H_{\alpha,\beta}(\Omega):=\Big\{u\in L^2(\Omega):\ \
&[u]_{x,\alpha;\Omega}<+\infty,\quad u(x,\cdot)\in W^{1,1}_{\mathrm{loc}}(\Omega^{x})\ \text{for a.e.\ }x\in\R^n,\\
& \text{and } |x|^\beta\nabla_yu\in L^2(\Omega;\R^m)\Big\},
\end{aligned}
\end{equation}
where $\nabla_yu$ denotes the slice-wise weak gradient, and we abbreviate $\mathcal H_{\alpha,\beta}:=\mathcal H_{\alpha,\beta}(\R^{n+m})$.

\begin{remark}
If $2\beta<n$, then $|x|^{-\beta}\in L^2_{\mathrm{loc}}(\R^n)$, so that for every compact $K\subset\Omega$
\[
\int_K|\nabla_yu|\,dz
\le\big\||x|^\beta\nabla_yu\big\|_{L^2(\Omega)}\,
\big\||x|^{-\beta}\big\|_{L^2(K)}<+\infty ;
\]
hence $\nabla_yu\in L^1_{\mathrm{loc}}(\Omega)$ and the slice-wise weak gradient coincides a.e.\ with the distributional $y$-gradient on $\Omega$. For $2\beta\ge n$ this need not hold across the degeneracy set $\Sigma:=\{(x,y):x=0\}$, and the two notions still agree a.e.\ on $\Omega\setminus\Sigma$. In what follows $\nabla_yu$ always denotes the slice-wise gradient, which is defined for a.e.\ $x$ and needs no integrability across $\Sigma$. Local integrability of $\nabla_yu$ in $(x,y)$ jointly is used only in the proof of Lemma~\ref{lem:poincare-cell}, on a set where $|x|$ is bounded below, and in the density argument below, where it follows from \eqref{eq:dens-gradbdd}.
\end{remark}

The space $\mathcal H_{\alpha,\beta}(\Omega)$ is endowed with the norm
\[
\|u\|_{\mathcal H_{\alpha,\beta}(\Omega)}^2
:=[u]_{x,\alpha;\Omega}^2
+(1+\beta)^2\big\||x|^\beta\nabla_y u\big\|_{L^2(\Omega)}^2
+\|u\|_{L^2(\Omega)}^2 .
\]
For $u,v\in\mathcal H_{\alpha,\beta}(\Omega)$ we set
\[
\ll u,v\gg_{x,\alpha;\Omega}
:=\! \frac12\int_{\R^m}\!\int_{\R^n}\!\int_{\R^n} \! \!
\frac{\big(u(x,y)-u(\xi,y)\big)\big(v(x,y)-v(\xi,y)\big)}{|x-\xi|^{n+2\alpha}}
\,\chi_\Omega(x,y)\,\chi_\Omega(\xi,y)dx\,d\xi\,dy,
\]
and define
\[
(u,v)_{\mathcal H_{\alpha,\beta}(\Omega)}
:=\ll u,v\gg_{x,\alpha;\Omega}
+(1+\beta)^2\int_{\Omega}|x|^{2\beta}\nabla_y u\cdot\nabla_y v\,dx\,dy
+\int_{\Omega}uv\,dx\,dy ,
\]
writing $\ll u,v\gg_{x,\alpha}:=\ll u,v\gg_{x,\alpha;\R^{n+m}}$ and $(\cdot,\cdot)_{\mathcal H_{\alpha,\beta}}:=(\cdot,\cdot)_{\mathcal H_{\alpha,\beta}(\R^{n+m})}$. It is easy to verify that $\big(\mathcal H_{\alpha,\beta}(\Omega),(\cdot,\cdot)_{\mathcal H_{\alpha,\beta}(\Omega)}\big)$ is a Hilbert space whenever \(\Omega\) is open, or more generally whenever the slices \(\Omega_x\) are open for a.e. \(x\).

Finally, the energy $\mathcal E_{\alpha,\beta}$ is a quadratic form, whose corresponding bilinear form is given by
\begin{equation} \label{eq:bilinear}
\begin{aligned}
\mathcal B_{\alpha,\beta}(u,v)
:=&\frac12\int_{\R^m}\int_{\R^n}\int_{\R^n}
\frac{\big(u(x,y)-u(\xi,y)\big)\big(v(x,y)-v(\xi,y)\big)}{|x-\xi|^{n+2\alpha}}
\,dx\,d\xi\,dy \\
&+(1+\beta)^2\int_{\R^{n+m}}|x|^{2\beta}\,\nabla_y u\cdot\nabla_y v\,dx\,dy ,
\end{aligned}
\end{equation}
so that $\mathcal B_{\alpha,\beta}(u,u)=\mathcal E_{\alpha,\beta}(u)$ for every $u$ for which the right-hand side of \eqref{eq:energy} is finite, and the homogeneous energy space associated with $\mathcal L_{\alpha,\beta}$ is
\begin{equation}
\label{eq:D-space}
\mathcal D^{\alpha,2}_\beta(\R^{n+m})
:=\overline{C_c^\infty(\R^{n+m})}^{\;\|\cdot\|_{\mathcal D^{\alpha,2}_\beta}},
\end{equation}
that is, the completion of $C_c^\infty(\R^{n+m})$ with respect to the energy norm given by
\begin{equation*}
    \|u\|_{\mathcal D^{\alpha,2}_\beta}
:=\mathcal E_{\alpha,\beta}(u)^{1/2}
=\Big([u]_{x,\alpha}^2+(1+\beta)^2\int_{\R^{n+m}}|x|^{2\beta}|\nabla_yu|^2\,dz\Big)^{1/2}.
\end{equation*}

\begingroup  

Two stability properties of $\mathcal H_{\alpha,\beta}$ will be used repeatedly, together with the local version of the space that the results of Sections~\ref{sec:riesz} and~\ref{sec:optimal} require.

\begin{remark}
Let $\Omega\subseteq\Omega'\subseteq\R^{n+m}$ be measurable and let $u\in\mathcal H_{\alpha,\beta}(\Omega')$. Then all three terms of the norm are monotone in the domain,
\[
\|u\|_{L^2(\Omega)}\le\|u\|_{L^2(\Omega')},
\qquad
\big\||x|^\beta\nabla_yu\big\|_{L^2(\Omega)}\le\big\||x|^\beta\nabla_yu\big\|_{L^2(\Omega')},
\qquad
[u]_{x,\alpha;\Omega}\le[u]_{x,\alpha;\Omega'},
\]
the last because $\chi_\Omega(x,y)\chi_\Omega(\xi,y)\le\chi_{\Omega'}(x,y)\chi_{\Omega'}(\xi,y)$ pointwise in \eqref{eq:seminorm}. If moreover the slices $\Omega^{x}$ are open for a.e.\ $x\in\R^n$, then $u|_\Omega\in\mathcal H_{\alpha,\beta}(\Omega)$ with $\|u\|_{\mathcal H_{\alpha,\beta}(\Omega)}\le\|u\|_{\mathcal H_{\alpha,\beta}(\Omega')}$, the slice-wise condition of \eqref{eq:H-space} being inherited because a compact subset of $\Omega^{x}$ is a compact subset of $\Omega'^{\,x}$.
\end{remark}

The space $\mathcal H_{\alpha,\beta}(\Omega)$ of \eqref{eq:H-space} is intrinsic to $\Omega$: it is built from the restricted seminorm $[\,\cdot\,]_{x,\alpha;\Omega}$, and it is the right class for the local inequalities of Section~\ref{sec:poincare}, whose right-hand sides involve only the values of $u$ on $\Omega$. It is, however, not suited to Sections~\ref{sec:riesz} and~\ref{sec:optimal}. On $\R^{n+m}$ its norm contains $\|u\|_{L^2(\R^{n+m})}$, whereas functions in the homogeneous space $\mathcal D^{\alpha,2}_\beta(\R^{n+m})$ of \eqref{eq:D-space} lie in $L^{2^*_{\alpha,\beta}}(\R^{n+m})$ but in general not in $L^2(\R^{n+m})$, so that $\mathcal D^{\alpha,2}_\beta(\R^{n+m})\not\subset\mathcal H_{\alpha,\beta}(\R^{n+m})$. Moreover, those sections work with the Grushin gradient \eqref{eq:grushin-gradient}, whose nonlocal part is the unrestricted density $D_{x,\alpha}u$, and by Lemma~\ref{lem:energy-comparison} a bound on $[u]_{x,\alpha;\Omega}$ for bounded $\Omega$ gives no bound on $\int_\Omega D_{x,\alpha}u\,dz$. Accordingly we define the local energy class of $\mathcal L_{\alpha,\beta}$ directly through \eqref{eq:grushin-gradient},
\begin{equation}
\label{eq:H-loc}
\begin{aligned}
\mathcal H^{\mathrm{loc}}_{\alpha,\beta}(\R^{n+m})
:=\Big\{\,u\in L^2_{\mathrm{loc}}(\R^{n+m}): &
u(x,\cdot)\in W^{1,1}_{\mathrm{loc}}(\R^{m})\\ & \text{for a.e. }x\in\R^n, \ \nabla_{\mathcal L}u\in L^2_{\mathrm{loc}}(\R^{n+m})\Big\}.
\end{aligned}
\end{equation}

\begin{remark}
\label{rem:Hloc-restricts} If $u\in\mathcal H^{\mathrm{loc}}_{\alpha,\beta}(\R^{n+m})$ then $u|_\Omega\in\mathcal H_{\alpha,\beta}(\Omega)$ for every bounded open $\Omega\subset\R^{n+m}$, so the inequalities of Section~\ref{sec:poincare} apply to $u$ on every anisotropic cylinder. Indeed $u\in L^2(\Omega)$ and the slice-wise condition hold by hypothesis, while \eqref{eq:grushin-gradient} and Lemma~\ref{lem:energy-comparison} give
\[
[u]^2_{x,\alpha;\Omega}\ \le\ \int_\Omega D_{x,\alpha}u\,dz
\ \le\ \int_\Omega|\nabla_{\mathcal L}u|^2\,dz<+\infty,
\qquad
(1+\beta)^2\!\int_\Omega|x|^{2\beta}|\nabla_yu|^2dz\le\int_\Omega|\nabla_{\mathcal L}u|^2dz .
\]
In particular, since every set of type $V_R$ defined in \eqref{eq:VR} is bounded, $\nabla_{\mathcal L}u\in L^2(V_R)$ for every $z_0$ and every $R>0$.
\end{remark}

Clearly $C_c^\infty(\R^{n+m})\subset\mathcal H^{\mathrm{loc}}_{\alpha,\beta}(\R^{n+m})$.

\begin{lemma}[Lipschitz stability]
\label{lem:lip-stable} Let $\Omega\subseteq\R^{n+m}$ be measurable with open slices, let $F:\R\to\R$ be $L$-Lipschitz with $F(0)=0$, and let $u\in\mathcal H_{\alpha,\beta}(\Omega)$. Then $F\circ u\in\mathcal H_{\alpha,\beta}(\Omega)$ and
\begin{equation}
\label{eq:lip-stable}
|F\circ u|\le L\,|u|\ \text{ a.e.},
\qquad
[F\circ u]_{x,\alpha;\Omega}\le L\,[u]_{x,\alpha;\Omega},
\qquad
\big|\nabla_y(F\circ u)\big|\le L\,\big|\nabla_yu\big|\ \text{ a.e.\ in }\Omega .
\end{equation}
In particular, for every $0\le t_1<t_2$ the truncation $T^{t_2}_{t_1}[u]:=\min\{\max\{0,u-t_1\},\,t_2-t_1\}$ and $u_\pm:=\max\{\pm u,0\}\in \mathcal H_{\alpha,\beta}(\Omega)$ and satisfy \eqref{eq:lip-stable} with $L=1$; moreover
\begin{equation*}
\nabla_y T^{t_2}_{t_1}[u]=\nabla_yu\,\chi_{\{t_1<u<t_2\}}
\qquad\text{a.e.\ in }\Omega .
\end{equation*}
All of this holds verbatim for $u\in\mathcal H^{\mathrm{loc}}_{\alpha,\beta}(\R^{n+m})$, applied on each bounded open $\Omega$ via Remark~\ref{rem:Hloc-restricts}. Finally, if $u$ is defined on all of $\R^{n+m}$ then the \emph{unrestricted} density \eqref{eq:density} and the Grushin gradient \eqref{eq:grushin-gradient} obey the same bound pointwise,
\begin{equation}
\label{eq:lip-stable-density}
D_{x,\alpha}(F\circ u)\le L^2\,D_{x,\alpha}u
\qquad\text{and}\qquad
\big|\nabla_{\mathcal L}(F\circ u)\big|\le L\,\big|\nabla_{\mathcal L}u\big|
\qquad\text{a.e.\ in }\R^{n+m},
\end{equation}
and both are unchanged by the addition of a constant to $u$.
\end{lemma}

\begin{proof}
Since $F(0)=0$ we have $|F(s)|=|F(s)-F(0)|\le L|s|$ for every $s\in\R$, which gives the first assertion of \eqref{eq:lip-stable} and, in particular, $F\circ u\in L^2(\Omega)$.

For the second, the Lipschitz bound applied to the pair $\big(u(x,y),u(\xi,y)\big)$ gives
\[
\big|F(u(x,y))-F(u(\xi,y))\big|^2\le L^2\,\big|u(x,y)-u(\xi,y)\big|^2
\]
pointwise; multiplying by $\tfrac12|x-\xi|^{-(n+2\alpha)}\chi_\Omega(x,y)\chi_\Omega(\xi,y)$ and integrating in $(x,\xi,y)$ yields $[F\circ u]^2_{x,\alpha;\Omega}\le L^2[u]^2_{x,\alpha;\Omega}$ by \eqref{eq:seminorm}.

For the third, fix $x\in\R^n$ with $u(x,\cdot)\in W^{1,1}_{\mathrm{loc}}(\Omega^{x})$, which is the case for a.e.\ $x$. The composition of an $L$-Lipschitz function with a $W^{1,1}_{\mathrm{loc}}$ function is again $W^{1,1}_{\mathrm{loc}}$, with $|\nabla(F\circ w)|\le L|\nabla w|$ a.e.; see Theorem $1$ of \cite{marcus_mizel}. Applying this on the slice $\Omega^{x}$ to $w=u(x,\cdot)$ and then using Fubini's theorem gives the stated a.e.\ bound on $\Omega$; multiplying by $|x|^{2\beta}$ and integrating shows $|x|^\beta\nabla_y(F\circ u)\in L^2(\Omega;\R^m)$, so $F\circ u\in\mathcal H_{\alpha,\beta}(\Omega)$.
\end{proof}
\endgroup

We start with a proposition that gives density of $C_c^\infty(\R^{n+m})$ in $\mathcal H_{\alpha,\beta}(\R^{n+m})$.
\begin{proposition}
Let $0<\alpha<1$ and $\beta\ge0$. Then $C_c^\infty(\R^{n+m})$ is dense in $\big(\mathcal H_{\alpha,\beta}(\R^{n+m}),\|\cdot\|_{\mathcal H_{\alpha,\beta}}\big)$, that is, \( \mathcal H_{\alpha,\beta}(\R^{n+m}) =\overline{C_c^\infty(\R^{n+m})}^{\;\|\cdot\|_{\mathcal H_{\alpha,\beta}}}. \)
\end{proposition}
\begin{proof}
Since $C_c^\infty(\R^{n+m})\subset\mathcal H_{\alpha,\beta}(\R^{n+m})$ and the latter is complete, it suffices to approximate an arbitrary $u\in\mathcal H_{\alpha,\beta}(\R^{n+m})$ in the energy norm by test functions.

Now, for $T>0$, put $F_T(t):=\max\{-T,\min\{t,T\}\}$ and $u^T:=F_T\circ u$. Then $t\mapsto t-F_T(t)$ is $1$-Lipschitz and vanishes at $0$, so by  {Lemma~\ref{lem:lip-stable}} applied to it,
\[
\big|(u-u^T)(x,y)-(u-u^T)(\xi,y)\big|\le|u(x,y)-u(\xi,y)|,
\qquad
\nabla_y(u-u^T)=\nabla_yu\,\mathbf 1_{\{|u|>T\}}\ \text{ a.e.}
\]
Both right-hand sides are square integrable against the respective measures, and both left-hand sides tend to $0$ a.e.\ as $T\to\infty$ because $|u|<\infty$ a.e. Dominated convergence therefore gives
\[
\|u^T-u\|_{\mathcal H_{\alpha,\beta}}\longrightarrow0
\qquad\text{as }T\to\infty .
\]
Hence we can assume from now on that $u\in L^\infty(\R^{n+m})$.

Next, fix $\eta\in C_c^\infty(\R^{n+m})$ with $0\le\eta\le1$, $\eta\equiv1$ on $B_1^n\times B_1^m$ and $\eta\equiv0$ outside $B_2^n\times B_2^m$, and for $R\ge1$ set
\[
\eta_R(x,y):=\eta\big(\delta_{1/R}(x,y)\big)=\eta\Big(\frac xR,\frac y{R^{\alpha+\beta}}\Big),
\qquad u_R:=\eta_Ru .
\]
Then $u_R\in\mathcal H_{\alpha,\beta}$, $u_R$ is bounded with compact support, and we claim $\|u_R-u\|_{\mathcal H_{\alpha,\beta}}\to0$.

Since $0\le\eta_R\le1$ and $\eta_R\to1$ pointwise, dominated convergence gives $\|u_R-u\|_{L^2}\to0$. For the weighted gradient, write
\[
|x|^\beta\nabla_y(u_R-u)=(\eta_R-1)\,|x|^\beta\nabla_yu+|x|^\beta u\,\nabla_y\eta_R .
\]
The first term tends to $0$ in $L^2$ by dominated convergence. For the second, $|\nabla_y\eta_R|\le\|\nabla_y\eta\|_\infty R^{-(\alpha+\beta)}$, and $\nabla_y\eta_R$ is supported in $\{|x|\le2R\}$, whence
\begin{equation*}
|x|^\beta|\nabla_y\eta_R|\le\|\nabla_y\eta\|_\infty(2R)^\beta R^{-(\alpha+\beta)}=CR^{-\alpha},
\end{equation*}
so that $\||x|^\beta u\nabla_y\eta_R\|_{L^2}\le CR^{-\alpha}\|u\|_{L^2}\to0$.

For the fractional seminorm set $w_R:=u_R-u=(\eta_R-1)u$ and decompose
\[
w_R(x,y)-w_R(\xi,y)
=\big(\eta_R(x,y)-1\big)\big(u(x,y)-u(\xi,y)\big)
+u(\xi,y)\big(\eta_R(x,y)-\eta_R(\xi,y)\big),
\]
so that $[w_R]_{x,\alpha}^2\le I_R+J_R$ with
\[
I_R:=\int_{\R^m}\!\!\iint_{\R^n\times\R^n}
\frac{|\eta_R(x,y)-1|^2\,|u(x,y)-u(\xi,y)|^2}{|x-\xi|^{n+2\alpha}}\,dx\,d\xi\,dy,
\]
and 
\[
J_R:=\int_{\R^{n+m}}|u(\xi,y)|^2\,K_R(\xi,y)\,d\xi\,dy,
\]
where $K_R(\xi,y):=\int_{\R^n}|x-\xi|^{-n-2\alpha}|\eta_R(x,y)-\eta_R(\xi,y)|^2\,dx$. The integrand of $I_R$ tends to $0$ pointwise and is dominated by $4|u(x,y)-u(\xi,y)|^2|x-\xi|^{-n-2\alpha}\in L^1$, so $I_R\to0$. For $J_R$ we use $|\eta_R(x,y)-\eta_R(\xi,y)|\le\min\{1,\|\nabla_x\eta\|_\infty|x-\xi|/R\}$ to get, uniformly in $(\xi,y)$,
\begin{equation*}
K_R(\xi,y)\!\le \! C\!\int_{\R^n}\frac{\min\{1,|h|^2/R^2\}}{|h|^{n+2\alpha}}\,dh
\!=\!C\Big(R^{-2}\!\!\int_{|h|<R}\!\!|h|^{2-n-2\alpha}dh+\!\!\int_{|h|\ge R}\!\!|h|^{-n-2\alpha}dh\Big)
\!=\!\frac{C(n,\alpha)}{R^{2\alpha}},
\end{equation*}
both integrals being finite precisely because $0<\alpha<1$. Hence $J_R\le C R^{-2\alpha}\|u\|_{L^2}^2\to0$, and the claim follows.

Now, let $v\in\mathcal H_{\alpha,\beta}\cap L^\infty$ have compact support, let $\varsigma\in C_c^\infty(\R^m)$ be a standard mollifier and put $v^\sigma:=v*_y\varsigma_\sigma$, the convolution acting in the $y$-variable only. Then $v^\sigma$ has compact support and

\emph{(a)} $\|v^\sigma-v\|_{L^2}\to0$;

\emph{(b)} since the weight $|x|^{2\beta}$ does not depend on $y$, it commutes with $*_y$, so $|x|^\beta\nabla_yv^\sigma=\big(|x|^\beta\nabla_yv\big)*_y\varsigma_\sigma \to|x|^\beta\nabla_yv$ in $L^2$;

\emph{(c)} writing $\tau_\zeta v(x,y):=v(x,y-\zeta)$ and using Minkowski's integral inequality,
\begin{equation}\label{eq:nonlocal_y_density}
[v^\sigma-v]_{x,\alpha}\le\int_{\R^m}\varsigma_\sigma(\zeta)\,[\tau_\zeta v-v]_{x,\alpha}\,d\zeta
\le\sup_{|\zeta|\le\sigma}[\tau_\zeta v-v]_{x,\alpha}\longrightarrow0,
\end{equation}
because $[\,\cdot\,]_{x,\alpha}$ is (half) the $L^2\big(\R^n\times\R^n\times\R^m\big)$ norm of $(x,\xi,y)\mapsto|x-\xi|^{-(n+2\alpha)/2}\big(v(x,y)-v(\xi,y)\big)$, on which $\tau_\zeta$ acts as a translation, and translations are continuous on $L^2$.

Thus $\|v^\sigma-v\|_{\mathcal H_{\alpha,\beta}}\to0$. Since $v\in L^\infty$,
\begin{equation}
\label{eq:dens-gradbdd}
\|\nabla_yv^\sigma\|_{L^\infty}=\|v*_y\nabla\varsigma_\sigma\|_{L^\infty}
\le\|\nabla\varsigma_\sigma\|_{L^1(\R^m)}\|v\|_{L^\infty}<\infty ,
\end{equation}
so that $\nabla_yv^\sigma\in L^\infty$ with compact support, hence $\nabla_yv^\sigma\in L^2(\R^{n+m})$ even without the weight.

Finally, let $g:=v^\sigma$ be as in defined above, let $\varrho\in C_c^\infty(\R^n)$ be a standard mollifier and set $g^\varepsilon:=g*_x\varrho_\varepsilon$. Then $g^\varepsilon\in C_c^\infty(\R^{n+m})$. Indeed, it is compactly supported, smooth in $x$ by convolution, smooth in $y$ because $\partial_y^\theta$ commutes with $*_x$ and $g$ is smooth in $y$, and all mixed derivatives are bounded and continuous. Moreover $\|g^\varepsilon-g\|_{L^2}\to0$, and $[g^\varepsilon-g]_{x,\alpha}\to0$ exactly as in \eqref{eq:nonlocal_y_density}, translations now acting in $x$. Finally, let $\mathcal K$ be a compact set containing the supports of all $g^\varepsilon$, $\varepsilon\le1$, and put $\Lambda:=\sup_{(x,y)\in\mathcal K}|x|^{2\beta}<\infty$. Since $\nabla_yg^\varepsilon=(\nabla_yg)*_x\varrho_\varepsilon$ and, by \eqref{eq:dens-gradbdd}, $\nabla_yg\in L^2(\R^{n+m})$, we obtain
\[
\big\||x|^\beta\nabla_y(g^\varepsilon-g)\big\|_{L^2}
\le\Lambda^{1/2}\,\big\|(\nabla_yg)*_x\varrho_\varepsilon-\nabla_yg\big\|_{L^2}
\longrightarrow0 .
\]
Hence $\|g^\varepsilon-g\|_{\mathcal H_{\alpha,\beta}}\to0$.
\end{proof}

The next proposition shows that $\rho_{\alpha,\beta}$ introduced in \eqref{eq:quasi-norm} satisfies a quasi-triangle inequality with an explicit sharp constant, and identifies exactly when $d_{\alpha,\beta}$ is a genuine metric.
\begin{proposition}[Quasi-triangle inequality]
For all $v,v'\in\R^n\times\R^m$,
\begin{equation*}
\rho_{\alpha,\beta}(v+v')\;\le\;C_{\triangle}(\gamma)\,\big(\rho_{\alpha,\beta}(v)+\rho_{\alpha,\beta}(v')\big),
\end{equation*}
where
\begin{equation}
\label{eq:C0-def}
C_{\triangle}(\gamma)\;:=\;\max\big\{1,\,2^{1/\gamma-1}\big\}
\;=\;
\begin{cases}
1, & \gamma\ge1,\\[2pt]
2^{1/\gamma-1}, & 0<\gamma<1.
\end{cases}
\end{equation}
The constant $C_{\triangle}(\gamma)$ is sharp for every $\gamma>0$. Consequently, $d_{\alpha,\beta}$ is a metric on $\R^n\times\R^m$ if and only if $\gamma=\alpha+\beta\ge1$.
\end{proposition}

\begin{proof}
Write $v=(x,y)$, $v'=(x',y')$ and $p:=2\gamma$. Then
\[
\rho_{\alpha,\beta}(x,y)=\big(|x|^{p}+\big(|y|^{1/\gamma}\big)^{p}\big)^{1/p}
=\big\|\big(|x|,|y|^{1/\gamma}\big)\big\|_{\ell^p(\R^2)},
\]
where for $0<p<1$ the right-hand side is the $\ell^p$ quasi-norm. Since this expression is nondecreasing in each of its two nonnegative entries, the triangle inequalities $|x+x'|\le|x|+|x'|$ in $\R^n$ and $|y+y'|\le|y|+|y'|$ in $\R^m$ give
\begin{equation}
\label{eq:qt-first}
\rho_{\alpha,\beta}(v+v')\;\le\;\big\|\big(|x|+|x'|,\,(|y|+|y'|)^{1/\gamma}\big)\big\|_{\ell^p(\R^2)}.
\end{equation}

Suppose first that $\gamma\ge1$. Then $t\mapsto t^{1/\gamma}$ is concave on $[0,\infty)$ and vanishes at $0$, hence subadditive, so $(|y|+|y'|)^{1/\gamma}\le|y|^{1/\gamma}+|y'|^{1/\gamma}$. Since $p\ge2$, Minkowski's inequality in $\ell^p(\R^2)$ applied to the right-hand side of~\eqref{eq:qt-first} yields $\rho_{\alpha,\beta}(v+v')\le\rho_{\alpha,\beta}(v)+\rho_{\alpha,\beta}(v')$.

Suppose next that $\tfrac12\le\gamma<1$. By convexity of $t\mapsto t^{1/\gamma}$,
\[
(s+s')^{1/\gamma}\le2^{1/\gamma-1}\big(s^{1/\gamma}+s'^{1/\gamma}\big),\qquad s,s'\ge0,
\]
and trivially $|x|+|x'|\le2^{1/\gamma-1}(|x|+|x'|)$. Inserting both bounds into~\eqref{eq:qt-first} and using the homogeneity of the norm together with Minkowski's inequality, which is available because $p\ge1$, we obtain
\[
\rho_{\alpha,\beta}(v+v')
\le2^{1/\gamma-1}\big\|\big(|x|,|y|^{1/\gamma}\big)+\big(|x'|,|y'|^{1/\gamma}\big)\big\|_{\ell^p(\R^2)}
\le2^{1/\gamma-1}\big(\rho_{\alpha,\beta}(v)+\rho_{\alpha,\beta}(v')\big).
\]

Suppose finally that $0<\gamma<\tfrac12$, so that $p<1$. Set $a:=|x|^{p}$, $a':=|x'|^{p}$, $b:=|y|^{2}$, $b':=|y'|^{2}$, so that $r:=\rho_{\alpha,\beta}(v)^{p}=a+b$ and $r':=\rho_{\alpha,\beta}(v')^{p}=a'+b'$. Raising \eqref{eq:qt-first} to the power $p$ gives
\[
\rho_{\alpha,\beta}(v+v')^{p}\le(|x|+|x'|)^{p}+(|y|+|y'|)^{2}
=\big(a^{1/p}+a'^{1/p}\big)^{p}+\big(b^{1/2}+b'^{1/2}\big)^{2}.
\]
As $1/p>1$, the $\ell^{1/p}$-norm on $\R^2$ is dominated by the $\ell^1$-norm, so $\big(a^{1/p}+a'^{1/p}\big)^{p}\le a+a'$; moreover $\big(b^{1/2}+b'^{1/2}\big)^{2}=b+b'+2\sqrt{bb'}\le2(b+b')$. Hence $\rho_{\alpha,\beta}(v+v')^{p}\le2(r+r')$. On the other hand, convexity of $t\mapsto t^{1/p}$ gives $r^{1/p}+r'^{1/p}\ge2^{1-1/p}(r+r')^{1/p}$, that is, $\big(\rho_{\alpha,\beta}(v)+\rho_{\alpha,\beta}(v')\big)^{p}\ge2^{p-1}(r+r')$. Combining the two estimates,
\[
\rho_{\alpha,\beta}(v+v')^{p}\le2^{2-p}\big(\rho_{\alpha,\beta}(v)+\rho_{\alpha,\beta}(v')\big)^{p},
\]
and taking $p$-th roots yields the claim, since $2^{(2-p)/p}=2^{1/\gamma-1}$. This proves~\eqref{eq:C0-def}.

For sharpness, the choice $v'=0$ shows that no constant smaller than $1$ is admissible. If $0<\gamma<1$, take $v=v'=(0,y)$ with $y\neq0$; then
\[
\rho_{\alpha,\beta}(v+v')=|2y|^{1/\gamma}=2^{1/\gamma-1}\cdot2|y|^{1/\gamma}
=2^{1/\gamma-1}\big(\rho_{\alpha,\beta}(v)+\rho_{\alpha,\beta}(v')\big),
\]
so equality holds and $C_{\triangle}(\gamma)$ cannot be decreased.

Finally, $d_{\alpha,\beta}(z,z')=\rho_{\alpha,\beta}(z-z')$ is symmetric, since $\rho_{\alpha,\beta}(-v)=\rho_{\alpha,\beta}(v)$, and vanishes only for $z=z'$. Writing $z-z''=(z-z')+(z'-z'')$, the case $\gamma\ge1$ above is precisely the triangle inequality for $d_{\alpha,\beta}$. If $\gamma<1$, the points $z=0$, $z'=(0,y)$, $z''=(0,2y)$ with $y\neq0$ satisfy, by the equality case just exhibited, $d_{\alpha,\beta}(z,z'')=2^{1/\gamma-1}\big(d_{\alpha,\beta}(z,z')+d_{\alpha,\beta}(z',z'')\big) >d_{\alpha,\beta}(z,z')+d_{\alpha,\beta}(z',z'')$, so the triangle inequality fails and $d_{\alpha,\beta}$ is not a metric.
\end{proof}

\begin{lemma}[Ball-cylinder equivalence]
\label{lem:ball-cylinder} For every $\gamma>0$, every $\zeta\in\mathbb{R}^n\times\mathbb{R}^m$, and every $r>0$,
\begin{equation*}
B(\zeta,r) \;\subset\; \mathcal{C}(\zeta,r) \;\subset\; B\big(\zeta,\kappa(\gamma)\, r\big),
\qquad \kappa(\gamma) := 2^{1/(2\gamma)},
\end{equation*}
with sharp constant $\kappa(\gamma)$ on the right-hand inclusion. In particular $\kappa(\gamma)$ is independent of $r$ and of the center $\zeta$.
\end{lemma}

\begin{proof}
\emph{Left inclusion.} If $d_{\alpha,\beta}(w,\zeta)<r$ then, writing $w=(x,y)$, $|x-\xi|^{2\gamma}<r^{2\gamma}$ and $|y-\eta|^2<r^{2\gamma}$, whence $|x-\xi|<r$ and $|y-\eta|<r^{\gamma}$; thus $w\in \mathcal{C}(\zeta,r)$.

\emph{Right inclusion.} If $w\in \mathcal{C}(\zeta,r)$ then $|x-\xi|<r$ and $|y-\eta|<r^{\gamma}$, so
\begin{equation*}
d_{\alpha,\beta}(w,\zeta) = \big(|x-\xi|^{2\gamma}+|y-\eta|^2\big)^{1/(2\gamma)}
< \big(r^{2\gamma}+r^{2\gamma}\big)^{1/(2\gamma)} = 2^{1/(2\gamma)}\, r.
\end{equation*}
Sharpness of $\kappa$ follows by letting $|x-\xi|,|y-\eta|^{1/\gamma}\to r$ simultaneously.
\end{proof}

\begin{remark}
$\kappa(\gamma)=2^{1/(2\gamma)}$ is strictly decreasing in $\gamma$, with $\kappa(\gamma)\in(1,2]$ for $\gamma\ge\tfrac12$ and $\kappa(\gamma)\in(1,\sqrt2\,]$ for $\gamma\ge1$; moreover $\kappa(\gamma)\downarrow1$ as $\gamma\to\infty$ and $\kappa(\gamma)\to\infty$ as $\gamma\to0^{+}$.
\end{remark}

\section{The uniform Grushin-Poincar\'e inequality and subcritical embeddings}
\label{sec:poincare} In this section we prove Theorem~\ref{thm:uniform-poincare}, first on a cell of unit size (Lemma~\ref{lem:poincare-cell}) and then on arbitrary cylinders by anisotropic scaling, and deduce Theorem~\ref{thm:Sobolev_subcritical} from it by a chaining and covering argument.\\



\begin{lemma}[Poincar\'e inequality on a unit cell]
\label{lem:poincare-cell} Let $0<\alpha<1$, $\beta\ge0$ and let $\theta,L>0$. Let $U\subset\R^n$ be a bounded measurable set with $|U|\ge\theta$ and $\operatorname{diam}U\le L$, and let $V\subset\R^m$ be a nonempty bounded convex open set with $\operatorname{diam}V\le L$. Write $E:=U\times V$,
\[
[u]_{x,\alpha;E}^2
:=\frac{1}{2}\int_V\int_U\int_U
\frac{|u(x,y)-u(\xi,y)|^2}{|x-\xi|^{n+2\alpha}}\,dx\,d\xi\,dy,
\qquad
\mathcal G(u;E):=\int_E|x|^{2\beta}|\nabla_yu|^2\,dx\,dy .
\]
Then for every  {$u\in\mathcal H_{\alpha,\beta}(E)$, with $\mathcal H_{\alpha,\beta}$ as in \eqref{eq:H-space}},
\begin{equation}
\label{eq:poincare-cell}
\int_E|u-u_E|^2\,dx\,dy
\;\le\; C_\star\big([u]_{x,\alpha;E}^2+\mathcal G(u;E)\big),
\end{equation}
where, with $\varpi_n:=|B_1^n|$ and $\tau_{\theta}:=\min\big\{1,(\theta/2\varpi_n)^{1/n}\big\}$,
\begin{equation}
\label{eq:poincare-cell-constant}
C_\star \;=\; 3\max\left\{\frac{5\,L^{n+2\alpha}}{\theta},\;
\frac{2\,|U|\,L^{2}\,\tau_{\theta}^{-2\beta}}{\pi^2\theta}\right\}.
\end{equation}
In particular $C_\star$ depends only on $n,\alpha,\beta,\theta,L$ and $|U|$, and \emph{not} on the positions of $U$ and $V$.
\end{lemma}
\begin{proof}
Set \( G:=U\setminus B^n(0,\tau_{\theta}). \) Since $|B^n(0,\tau_{\theta})|=\varpi_n\tau_{\theta}^n\le\theta/2$ by the choice of $\tau_{\theta}$, we have
\begin{equation}
\label{eq:good-set}
|G|\;\ge\;|U|-\varpi_n\tau_{\theta}^{n}\;\ge\;\theta-\tfrac\theta2\;=\;\tfrac\theta2,
\qquad\text{and}\qquad
|x|^{2\beta}\ge\tau_{\theta}^{2\beta}\ \text{ on }G,
\end{equation}
the second assertion because $|x|\ge\tau_{\theta}$ on $G$ and $\beta\ge0$.

For $y\in V$ put
\[
a(y):=\fint_U u(x,y)\,dx,\qquad
b(y):=\fint_G u(x,y)\,dx,\qquad
b_V:=\fint_V b(y)\,dy .
\]
Since $u_E$ is the best constant $L^2(E)$-approximation of $u$,
\[
\int_E|u-u_E|^2\,dx\,dy\le\int_E|u-b_V|^2\,dx\,dy,
\]
and decomposing $u-b_V=(u-a(y))+(a(y)-b(y))+(b(y)-b_V)$ together with $|G_1+G_2+G_3|^2\le3(|G_1|^2+|G_2|^2+|G_3|^2)$ gives
\begin{equation}
\label{eq:cell-split}
\int_E|u-u_E|^2\,dx\,dy\le3\,(I_1+I_2+I_3),
\end{equation}
with $I_1:=\int_V\int_U|u(x,y)-a(y)|^2\,dx\,dy$, $I_2:=\int_V\int_U|a(y)-b(y)|^2\,dx\,dy=|U|\int_V|a-b|^2\,dy$ and $I_3:=|U|\int_V|b(y)-b_V|^2\,dy$.

\smallskip\noindent\emph{Estimate of $I_1$.} For a measurable set $W$ of finite positive measure and $g\in L^2(W)$ one has the elementary identity
\begin{equation}
\label{eq:variance-identity}
\int_W|g-g_W|^2\,dx=\frac1{2|W|}\int_W\int_W|g(x)-g(\xi)|^2\,dx\,d\xi ,
\end{equation}
obtained by expanding the square on the right. Applying \eqref{eq:variance-identity} on $W=U$ to $g=u(\cdot,y)$ and using $|x-\xi|\le L$ for $x,\xi\in U$, so that $1\le L^{n+2\alpha}|x-\xi|^{-(n+2\alpha)}$,
\begin{align*}
\int_U|u(x,y)-a(y)|^2\,dx
&=\frac1{2|U|}\int_U\int_U|u(x,y)-u(\xi,y)|^2\,dx\,d\xi \\
& \le\frac{L^{n+2\alpha}}{2\theta}\int_U\int_U
\frac{|u(x,y)-u(\xi,y)|^2}{|x-\xi|^{n+2\alpha}}\,dx\,d\xi .
\end{align*}
Integrating in $y\in V$,
\begin{equation}
\label{eq:I1-cell}
I_1\le\frac{L^{n+2\alpha}}{\theta}\,[u]_{x,\alpha;E}^2 .
\end{equation}

\smallskip\noindent\emph{Estimate of $I_2$.} Since $a(y)-b(y)=\fint_U\fint_G\big(u(x,y)-u(\xi,y)\big)\,d\xi\,dx$, Jensen's inequality and $G\subseteq U$ give
\[
|a(y)-b(y)|^2
\le\fint_U\fint_G|u(x,y)-u(\xi,y)|^2\,d\xi\,dx
\le\frac1{|U|\,|G|}\int_U\int_U|u(x,y)-u(\xi,y)|^2\,dx\,d\xi .
\]
Using $1\le L^{n+2\alpha}|x-\xi|^{-(n+2\alpha)}$ again and then \eqref{eq:good-set},
\begin{align*}
|U|\,|a(y)-b(y)|^2
&\le\frac{L^{n+2\alpha}}{|G|}\int_U\int_U
\frac{|u(x,y)-u(\xi,y)|^2}{|x-\xi|^{n+2\alpha}}\,dx\,d\xi \\
& \le\frac{2L^{n+2\alpha}}{\theta}\int_U\int_U
\frac{|u(x,y)-u(\xi,y)|^2}{|x-\xi|^{n+2\alpha}}\,dx\,d\xi ,
\end{align*}
whence, integrating in $y$,
\begin{equation*}
I_2\le\frac{4L^{n+2\alpha}}{\theta}\,[u]_{x,\alpha;E}^2 .
\end{equation*}

\smallskip\noindent\emph{Estimate of $I_3$.}  {We first check that $b\in W^{1,2}(V)$ with}
\begin{equation}
\label{eq:b-weak-gradient}
 {\nabla_yb(y)=\fint_G\nabla_yu(x,y)\,dx
\qquad\text{for a.e.\ }y\in V .}
\end{equation}
 {On $G$ one has $|x|\ge\tau_{\theta}>0$ by \eqref{eq:good-set}, whence $|\nabla_yu|\le\tau_{\theta}^{-\beta}|x|^{\beta}|\nabla_yu|$ there; since $|x|^\beta\nabla_yu\in L^2(E)$ and $|G\times V|<\infty$, this gives $|\nabla_yu|\in L^2(G\times V)\subset L^1(G\times V)$, and likewise $u\in L^1(G\times V)$ because $u\in L^2(E)$. Hence for every $\varphi\in C_c^\infty(V)$ and every $j\in\{1,\dots,m\}$, Fubini's theorem and the slice-wise weak differentiability of $u$ give}
\[
 {\int_Vb\,\partial_{y_j}\varphi\,dy
=\fint_G\!\int_Vu\,\partial_{y_j}\varphi\,dy\,dx
=-\fint_G\!\int_V\partial_{y_j}u\,\varphi\,dy\,dx
=-\int_V\Big(\fint_G\partial_{y_j}u\,dx\Big)\varphi\,dy ,}
\]
 {which is \eqref{eq:b-weak-gradient}. By Jensen's inequality the right-hand side of \eqref{eq:b-weak-gradient} lies in $L^2(V)$, and $b\in L^2(V)$, so indeed $b\in W^{1,2}(V)$.}

The set $V$ is convex with $\operatorname{diam}V\le L$, so the Payne-Weinberger inequality (see \citep{bebendorf_poincare, payne_weinberger}) gives
\[
\int_V|b-b_V|^2\,dy\le\Big(\frac{\operatorname{diam}V}{\pi}\Big)^2
\int_V|\nabla_yb|^2\,dy\le\frac{L^2}{\pi^2}\int_V|\nabla_yb|^2\,dy,
\]
with a constant depending only on $L$ and in particular independent of the position of $V$.  {Using \eqref{eq:b-weak-gradient}, Jensen's inequality} and then \eqref{eq:good-set},
\begin{align*}
|\nabla_yb(y)|^2
&\le\fint_G|\nabla_yu(x,y)|^2\,dx
\le\frac{\tau_{\theta}^{-2\beta}}{|G|}\int_G|x|^{2\beta}|\nabla_yu(x,y)|^2\,dx\\
&\le\frac{2\tau_{\theta}^{-2\beta}}{\theta}\int_U|x|^{2\beta}|\nabla_yu(x,y)|^2\,dx .
\end{align*}

Hence
\begin{equation}
\label{eq:I3-cell}
I_3\le|U|\,\frac{L^2}{\pi^2}\,\frac{2\tau_{\theta}^{-2\beta}}{\theta}\,
\mathcal G(u;E)
=\frac{2\,|U|\,L^{2}\,\tau_{\theta}^{-2\beta}}{\pi^2\theta}\,\mathcal G(u;E).
\end{equation}
Substituting \eqref{eq:I1-cell}-\eqref{eq:I3-cell} into \eqref{eq:cell-split} yields \eqref{eq:poincare-cell} with the constant \eqref{eq:poincare-cell-constant}.
\end{proof}

\begin{proof}[Proof of Theorem \ref{thm:uniform-poincare}]
Fix $z_0$ and $r>0$ and define the anisotropic affine map
\[
T(X,Y):=\big(rX,\;r^{\gamma}Y+y_0\big),
\qquad v:=u\circ T .
\]
Then $T$ is a bijection from $U\times V$ onto $\mathcal{C}_r(z_0)$, where
\[
U:=B^n(x_0/r,\,1),\qquad V:=B^m(0,1),
\]
because $|X-x_0/r|<1\iff|rX-x_0|<r$ and $|Y|<1\iff|r^\gamma Y+y_0-y_0|<r^\gamma$. Moreover $dx\,dy=r^{Q_{\alpha,\beta}}\,dX\,dY$ with $Q_{\alpha,\beta}=n+\gamma m$ as in \eqref{eq:Q}, and $v_{U\times V}=u_{\mathcal{C}_r(z_0)}$, so
\begin{equation}
\label{eq:unif-lhs}
\int_{\mathcal{C}_r(z_0)}|u-u_{\mathcal{C}_r(z_0)}|^2\,dx\,dy
=r^{Q_{\alpha,\beta}}\int_{U\times V}|v-v_{U\times V}|^2\,dX\,dY .
\end{equation}
Writing $x=rX$, $\xi=r\Xi$, $y=r^{\gamma}Y+y_0$, we get $dx\,d\xi\,dy=r^{2n+\gamma m}\,dX\,d\Xi\,dY$, and $|x-\xi|^{-(n+2\alpha)}=r^{-(n+2\alpha)}|X-\Xi|^{-(n+2\alpha)}$, so that
\begin{equation*}
[u]^2_{x,\alpha;\mathcal{C}_r(z_0)}=r^{\,Q_{\alpha,\beta}-2\alpha}\,[v]^2_{X,\alpha;U\times V}.
\end{equation*}
Similarly $\nabla_Yv(X,Y)=r^{\gamma}\nabla_yu(T(X,Y))$ and $|x|^{2\beta}=r^{2\beta}|X|^{2\beta}$, and since $\gamma-\beta=\alpha$,
\begin{equation}
\label{eq:unif-grad}
\int_{\mathcal{C}_r(z_0)}|x|^{2\beta}|\nabla_yu|^2\,dx\,dy
=r^{\,Q_{\alpha,\beta}-2\alpha}\int_{U\times V}|X|^{2\beta}|\nabla_Yv|^2\,dX\,dY .
\end{equation}
The pair $(U,V)$ satisfies the hypotheses of Lemma~\ref{lem:poincare-cell} with
\[
\theta=|U|=\varpi_n=|B_1^n|,\qquad L=2,
\]
uniformly in $x_0/r$. Applying \eqref{eq:poincare-cell} to $v$, multiplying by $r^{Q_{\alpha,\beta}}$ and using \eqref{eq:unif-lhs}-\eqref{eq:unif-grad} gives \eqref{eq:poincare-uniform} with $C=C_\star(n,\alpha,\beta)$.
\end{proof}

We now upgrade Theorem~\ref{thm:uniform-poincare} to a Sobolev embedding at every subcritical exponent, working on anisotropic cylinders $ \mathcal{C}(z_0,r)$ and writing $\mu$ for Lebesgue measure $dx\,dy$ on $\R^{n+m}$.

We recall the following lemma, see {\cite[Lemma~20.18]{heinonen_nonlinear_potential_book}}.
\begin{lemma}
\label{lem:measure_integrability} Let $(X,\nu)$ be a finite measure space and $u:X\to\R$ measurable. Suppose there exist $q_{0}>1$ and $C_0>0$ such that $\nu(\{x\in X:|u(x)|>t\})\le C_0t^{-q_{0}}$ for all $t>0$. Then for every $0<q<q_{0}$,
\[
\int_X|u|^{q}\,d\nu
\le\frac{q_{0}}{q_{0}-q}\,C_0^{q/q_{0}}\,\nu(X)^{(q_{0}-q)/q_{0}}.
\]
\end{lemma}

Now, to prove the main Theorem~\ref{thm:Sobolev_subcritical}, we require to choose a chain of dyadic points satisfying some specific properties. To this end, fix $z_0,z$ with $z\in \mathcal{C}(z_0,r)$, and let $\rho:=d_{\alpha,\beta}(z_0,z)<\kappa(\gamma)\,r$ with $\kappa(\gamma)=2^{1/(2\gamma)}$ as in Lemma~\ref{lem:ball-cylinder}. For $M>0$, to be chosen later, define the offset point along the anisotropic dilation curve,
\begin{equation*}
z_0^{*} \;:=\; z_0 - \delta_{s}(z-z_0), \qquad s:=\frac{Mr}{\rho}, \qquad
\delta_s(x,y):=(sx,\,s^{\gamma}y),
\end{equation*}
so that, by exact homogeneity of $\delta_s$ under $\rho_{\alpha,\beta}(\cdot)=d_{\alpha,\beta}(\cdot,0)$, we have $d_{\alpha,\beta}(z_0^{*},z_0) = Mr$. Set $d:=d_{\alpha,\beta}(z_0^{*},z)$ and
\begin{equation}
\label{eq:base-cylinder}
\mathcal{C}^{(0)}:= \mathcal{C}\Big(z_0^{*},\,\frac{d}{2\kappa(\gamma)}\Big).
\end{equation}
We show that, with the proper choice of $M(\gamma)$, the base case of the dyadic chain satisfies
\begin{equation}
\label{eq:margin-goal}
 \mathcal{C}^{(0)} \;\subset\;  \mathcal{C}(z_0,r)^{c}.
\end{equation}

\begin{proposition}
\label{prop:M-offset} Let $\gamma=\alpha+\beta>0$, let $\kappa=\kappa(\gamma)=2^{1/(2\gamma)}$ be the constant of Lemma~\ref{lem:ball-cylinder}, and set
\begin{equation*}
\Theta(M):=\max\Big\{1+\frac{\kappa}{M},\ \big(1+\sqrt2\,M^{-\gamma}\big)^{1/\gamma}\Big\}.
\end{equation*}
Suppose that $M>0$ satisfies
\begin{equation}
\label{eq:M-offset-conditions}
\Theta(M)\le\frac32,
\qquad
M\ \ge\ 4\kappa,
\qquad
M\ \ge\ \Big(\frac{\sqrt2}{1-(3/4)^{\gamma}}\Big)^{1/\gamma}.
\end{equation}
Then \eqref{eq:margin-goal} holds, and moreover $Mr\le d<\tfrac32\,Mr$. In particular, the explicit constant
\begin{equation}
\label{eq:M-offset}
M(\gamma):=\max\left\{
2^{\,2+\frac1{2\gamma}},\ \
\Big(\frac{\sqrt2}{(3/2)^{\gamma}-1}\Big)^{1/\gamma},\ \
\Big(\frac{\sqrt2}{1-(3/4)^{\gamma}}\Big)^{1/\gamma}
\right\}
\end{equation}
is admissible for \emph{every} $\gamma>0$.
\end{proposition}

\begin{proof}
Write $z-z_0=(\Delta x,\Delta y)$, $a:=|\Delta x|$, $b:=|\Delta y|$, and recall $ \mathcal{C}(\zeta,t)=\{|x-\xi|<t,\ |y-\eta|<t^{\gamma}\}$, so that $z\in \mathcal{C}(z_0,r)$ means $a<r$ and $b<r^{\gamma}$; we may assume $z\ne z_0$, i.e.\ $\rho:=\rho_{\alpha,\beta}(\Delta x,\Delta y)>0$. Set
\begin{equation}
\label{eq:AB-def}
A:=\frac a\rho\in[0,1],\qquad B:=\frac b{\rho^{\gamma}}\in[0,1],
\qquad\text{so that}\qquad A^{2\gamma}+B^{2}=\frac{a^{2\gamma}+b^{2}}{\rho^{2\gamma}}=1 .
\end{equation}
From $a<r$ and $b<r^{\gamma}$ we also get $\rho^{2\gamma}=a^{2\gamma}+b^{2}<2r^{2\gamma}$, hence
\begin{equation}
\label{eq:rho-bounds}
\rho<\kappa r\qquad\text{and}\qquad\rho^{\gamma}<\kappa^{\gamma}r^{\gamma}=\sqrt2\,r^{\gamma}.
\end{equation}

We first record the exact offset between $z_0$ and $z_0^{*}$. By definition $s\rho=Mr$, hence also $s^{\gamma}\rho^{\gamma}=(Mr)^{\gamma}$; since $z_0^{*}-z_0=-\delta_s(z-z_0)=(-s\Delta x,-s^{\gamma}\Delta y)$,
\begin{equation}
\label{eq:offset-coords}
|\xi_0^{*}-\xi_0|=s\,a=s\rho A=Mr\,A,
\qquad
|\eta_0^{*}-\eta_0|=s^{\gamma}b=s^{\gamma}\rho^{\gamma}B=(Mr)^{\gamma}B .
\end{equation}
Likewise $z_0^{*}-z=-\big((1+s)\Delta x,\ (1+s^{\gamma})\Delta y\big)$, so that, using $(1+s)a=(\rho+Mr)A$ and $(1+s^{\gamma})b=\big(\rho^{\gamma}+(Mr)^{\gamma}\big)B$,
\begin{equation*}
d^{2\gamma}
=(1+s)^{2\gamma}a^{2\gamma}+(1+s^{\gamma})^{2}b^{2}
=A^{2\gamma}\big(\rho+Mr\big)^{2\gamma}+B^{2}\big(\rho^{\gamma}+(Mr)^{\gamma}\big)^{2}.
\end{equation*}
By \eqref{eq:AB-def} this gives at once the lower bound $d\ge Mr\,(A^{2\gamma}+B^{2})^{1/(2\gamma)}=Mr$. In the other direction, \eqref{eq:rho-bounds} yields
\begin{equation*}
\big(\rho+Mr\big)^{2\gamma}<(Mr)^{2\gamma}\Big(1+\frac{\kappa}{M}\Big)^{2\gamma}
\le(Mr)^{2\gamma}\Theta(M)^{2\gamma},
\qquad
\big(\rho^{\gamma}+(Mr)^{\gamma}\big)^{2}<(Mr)^{2\gamma}\Theta(M)^{2\gamma},
\end{equation*}
whence, again by $A^{2\gamma}+B^{2}=1$ and by the first condition in \eqref{eq:M-offset-conditions},
\begin{equation}
\label{eq:d-upper}
Mr\ \le\ d\ <\ \Theta(M)\,Mr\ \le\ \tfrac32\,Mr .
\end{equation}

Now let $w=(X,Y)\in \mathcal{C}^{(0)}$, so that $|X-\xi_0^{*}|<d/(2\kappa)$ and $|Y-\eta_0^{*}|<\big(d/(2\kappa)\big)^{\gamma}$. Combining with \eqref{eq:offset-coords} and the Euclidean triangle inequality in each block of variables separately,
\begin{equation*}
|X-\xi_0|>Mr\,A-\frac{d}{2\kappa}=:\Phi_x,
\qquad
|Y-\eta_0|>(Mr)^{\gamma}B-\Big(\frac{d}{2\kappa}\Big)^{\gamma}=:\Phi_y ;
\end{equation*}
by the definition of $ \mathcal{C}(z_0,r)$, either $\Phi_x\ge r$ or $\Phi_y\ge r^{\gamma}$ already forces $w\notin \mathcal{C}(z_0,r)$. Since $\kappa^{-2\gamma}=\tfrac12$ and $\kappa^{-\gamma}=2^{-1/2}$, the constraint $A^{2\gamma}+B^{2}=1$ shows that the inequalities $A<\kappa^{-1}$ and $B<\kappa^{-\gamma}$ cannot both hold, for they would give $A^{2\gamma}+B^{2}<\tfrac12+\tfrac12=1$. We may therefore distinguish two regimes.

\emph{Case 1: $A\ge\kappa^{-1}$.} By \eqref{eq:d-upper},
\begin{equation*}
\Phi_x\ >\ \frac{Mr}{\kappa}-\frac{3Mr}{4\kappa}\ =\ \frac{Mr}{4\kappa}\ \ge\ r,
\end{equation*}
the last inequality being the second condition in \eqref{eq:M-offset-conditions}.

\emph{Case 2: $B\ge\kappa^{-\gamma}=2^{-1/2}$.} Since $t\mapsto t^{\gamma}$ is increasing, \eqref{eq:d-upper} gives
\begin{equation*}
\Phi_y\ >\ \frac{(Mr)^{\gamma}}{\sqrt2}-\Big(\frac{3Mr}{4\kappa}\Big)^{\gamma}
=\frac{(Mr)^{\gamma}}{\sqrt2}\Big(1-(3/4)^{\gamma}\Big)\ \ge\ r^{\gamma},
\end{equation*}
the last inequality being the third condition in \eqref{eq:M-offset-conditions}. In either case $w\notin \mathcal{C}(z_0,r)$, which proves \eqref{eq:margin-goal}.

It remains to check that \eqref{eq:M-offset} meets \eqref{eq:M-offset-conditions}. The requirement $\Theta(M)\le\tfrac32$ amounts to the two conditions $M\ge2\kappa$ and $M\ge\big(\sqrt2/((3/2)^{\gamma}-1)\big)^{1/\gamma}$, the first of which is implied by $M\ge4\kappa=2^{2+\frac1{2\gamma}}$; hence \eqref{eq:M-offset} is exactly the maximum of the three surviving constraints. For every $\gamma>0$ one has $(3/2)^{\gamma}>1$ and $(3/4)^{\gamma}<1$ strictly, so all three are finite.
\end{proof}

\begin{proof}[Proof of Theorem \ref{thm:Sobolev_subcritical}]
The idea is to join each point of the level set $\{|u|>t\}$ to a cylinder outside the support of $u$, where the average of $u$ vanishes, by a dyadic chain of cylinders. The Poincar\'e inequality along the chain then gives a cylinder around the point on which the energy is large, and the level set is covered by such cylinders with the Vitali lemma.

Throughout we regard $u$ as defined on all of $\R^{n+m}$, extended by zero outside $\mathcal{C}(z_0,r)$; the extension is Lipschitz on $\R^{n+m}$ because $\operatorname{supp}u$ is a compact subset of the open cylinder.

Fix $\varepsilon\in(0,\alpha)$. We first establish the weak-type estimate
\begin{equation}
\label{eq:weaktype}
\big|\{z\in \mathcal{C}(z_0,r):| u(z)|>t\}\big|
\le C(\gamma,\varepsilon)\left(r^{2\varepsilon} \int_{ \mathcal{C}(z_0,r)}|\nabla_{\mathcal L} u|^2\,dz\right)^{\frac{Q}{Q - 2\alpha + 2\varepsilon}}t^{-\frac{2Q}{Q-2\alpha + 2\varepsilon}},
\qquad t>0.
\end{equation}
Fix $t>0$ and $z\in \mathcal{C}(z_0,r)$ be a Lebesgue point with $| u(z)|>t$, and set
\begin{equation*}
z_0^{*} \;:=\; z_0 - \delta_{s}(z-z_0), \qquad s:=\frac{M(\gamma)r}{\rho},
\end{equation*}
with $M(\gamma)$ as in Proposition~\ref{prop:M-offset}; then $\mathcal{C}^{(0)}\subset \mathcal{C}(z_0,r)^{c}$ by \eqref{eq:margin-goal}. For each $k\in\mathbb{N}\cup\{0\}$, define the dyadic chain points lying between $z_0^{*}$ and $z$ as
\begin{equation*}
\zeta_k \;:=\; z + \delta_{2^{-k}}\big(z_0^{*}-z\big), \qquad k\ge0,
\end{equation*}
so that $d_{\alpha,\beta}(\zeta_k,z)=2^{-k}d$ with $d:=d_{\alpha,\beta}(z_0^{*},z)$, and set
\begin{equation}
\label{eq:chain-radii}
r_k:=\frac{2^{-k}d}{2\kappa(\gamma)},
\qquad
\mathcal{C}^{(k)}:=\mathcal{C}\big(\zeta_k,r_k\big),\qquad k\ge0,
\end{equation}
which for $k=0$ is consistent with \eqref{eq:base-cylinder}. Next we claim that for every $k\ge1$,
\begin{equation}\label{eq:dyadic_cylinder_containment}
\mathcal{C}^{(k)} \;\subset\; \lambda(\gamma)\, \mathcal{C}^{(k-1)},
\end{equation}
where
\begin{equation*}
\lambda(\gamma) \;:=\; \frac{C_{\triangle}(\gamma)\,\kappa(\gamma)\big(1+2\vartheta(\gamma)\big)}{2},
\qquad
\vartheta(\gamma):=\max\big\{1,(2^{\gamma}-1)^{1/\gamma}\big\},
\end{equation*}
and $\lambda(\gamma)\mathcal{C}^{(k-1)}:= \mathcal{C}\big(\zeta_{k-1},\lambda(\gamma)\,r_{k-1}\big)$.

To this end, write $v:=z_0^{*}-z=(\Delta x,\Delta y)$ and $p:=2^{-k}$. Since $\zeta_k=z+\delta_p(v)$ and $\zeta_{k-1}=z+\delta_{2p}(v)$,
\begin{equation*}
\zeta_k-\zeta_{k-1} \;=\; \delta_p(v)-\delta_{2p}(v)
\;=\; \big(-p\Delta x,\ -p^{\gamma}(2^{\gamma}-1)\Delta y\big)
\;=\; -\delta_p\big(\Delta x,\ (2^{\gamma}-1)\Delta y\big),
\end{equation*}
so that, by exact homogeneity of $d_{\alpha,\beta}(\cdot,0)$ under $\delta_p$,
\begin{equation*}
d_{\alpha,\beta}(\zeta_k,\zeta_{k-1}) \;=\; p\, d_{\alpha,\beta}\big((\Delta x,(2^{\gamma}-1)\Delta y),\,0\big).
\end{equation*}
Since
\begin{equation*}
d_{\alpha,\beta}\big((\Delta x,(2^{\gamma}-1)\Delta y),0\big)^{2\gamma}
= |\Delta x|^{2\gamma}+(2^{\gamma}-1)^2|\Delta y|^2 \le
\vartheta(\gamma)^{2\gamma}\big(|\Delta x|^{2\gamma}+|\Delta y|^2\big) = \vartheta(\gamma)^{2\gamma}\,d^{2\gamma},
\end{equation*}
the definition \eqref{eq:chain-radii} of $r_k$ yields
\begin{equation*}
d_{\alpha,\beta}(\zeta_k,\zeta_{k-1}) \;\le\; 2^{-k}\vartheta(\gamma)\,d
\;=\; 2\,\kappa(\gamma)\,\vartheta(\gamma)\,r_k.
\end{equation*}
Consequently, for $w\in \mathcal{C}^{(k)}$ we have $d_{\alpha,\beta}(w,\zeta_k)<\kappa(\gamma)r_k$ by Lemma~\ref{lem:ball-cylinder}, whence
\begin{equation*}
d_{\alpha,\beta}(w,\zeta_{k-1}) \le C_{\triangle}(\gamma)\Big(d_{\alpha,\beta}(w,\zeta_k) +
d_{\alpha,\beta}(\zeta_k,\zeta_{k-1})\Big)
< C_{\triangle}(\gamma)\,\kappa(\gamma)\big(1+2\vartheta(\gamma)\big)\,r_k
= \lambda(\gamma)\,r_{k-1},
\end{equation*}
using $r_k=r_{k-1}/2$ in the last step. Thus, by the left inclusion in Lemma~\ref{lem:ball-cylinder}, $w\in\mathcal{C}\big(\zeta_{k-1},\lambda(\gamma)r_{k-1}\big)$, which proves the claim.

Since $ u$ is continuous and $r_k\to0$ with $\zeta_k\to z$, we have $ u_{\mathcal{C}^{(k)}}\to u(z)$, while $ u_{\mathcal{C}^{(0)}}=0$ because $ u$ is supported in $\mathcal{C}(z_0,r)$ and $\mathcal{C}^{(0)}\subset\mathcal{C}(z_0,r)^c$. Hence the triangle inequality together with the Poincar\'e inequality~\eqref{eq:poincare-uniform}, applied along the chain $\{\mathcal{C}^{(k)}\}$ and using \eqref{eq:dyadic_cylinder_containment}, gives
\begin{equation}
\label{eq:telescoping}
\begin{aligned}
t&<| u(z)|=| u(z)- u_{\mathcal{C}^{(0)}}|
\le \sum_{k=1}^\infty |  u_{\mathcal{C}^{(k)}} -  u_{\mathcal{C}^{(k-1)}}|\\
&\le  \sum_{k=1}^\infty \left(|  u_{\mathcal{C}^{(k)}} -  u_{\lambda(\gamma)\mathcal{C}^{(k-1)}}| + |  u_{\lambda(\gamma)\mathcal{C}^{(k-1)}} -  u_{\mathcal{C}^{(k-1)}}|\right) \\
&\le  \sum_{k=1}^\infty \left( \fint_{\mathcal{C}^{(k)}} |  u(z) -  u_{\lambda(\gamma)\mathcal{C}^{(k-1)}}|dz + \fint_{\mathcal{C}^{(k-1)}}|  u(z) -  u_{\lambda(\gamma)\mathcal{C}^{(k-1)}}| dz\right) \\
&\le C(\gamma)\sum_{k=0}^\infty d^{\alpha}\,2^{-k \alpha}
\Big(\fint_{\lambda(\gamma) \mathcal{C}^{(k)}}|\nabla_{\mathcal L} u|^2\,dz\Big)^{1/2},
\end{aligned}
\end{equation}
the factor $\kappa(\gamma)^{-\alpha}$ produced by \eqref{eq:chain-radii} having been absorbed into $C(\gamma)$. Writing $t=C(\varepsilon)\sum_{k\ge0}t\,2^{-k\varepsilon}$ and comparing with \eqref{eq:telescoping}, there is an index $k_z\ge0$ such that
\begin{equation}
\label{eq:critical}
t\,2^{-k_z\varepsilon}\le C(\gamma,\varepsilon)d^{\alpha}\,2^{-k_z \alpha}
\Big(\fint_{\lambda(\gamma) \mathcal{C}^{(k_z)}}|\nabla_{\mathcal L} u|^2\,dz\Big)^{1/2}.
\end{equation}

\noindent Using \eqref{eq:d-upper} in \eqref{eq:critical}, we get
\begin{align*}
t&\le C(\gamma,\varepsilon)d^\varepsilon(d2^{-k_z})^{\alpha-\varepsilon}\mu(\lambda (\gamma) \mathcal{C}^{(k_z)})^{-1/2}\Big(\int_{\lambda (\gamma) \mathcal{C}^{(k_z)}}|\nabla_{\mathcal
L} u|^2\,dz\Big)^{1/2} \\
&\le C(\gamma,\varepsilon)r^\varepsilon \mu(\lambda (\gamma) \mathcal{C}^{(k_z)})^{\frac{\alpha -\varepsilon}{Q}-\frac12} \Big(\int_{\lambda (\gamma) \mathcal{C}^{(k_z)}}|\nabla_{\mathcal L} u|^2\,dz\Big)^{1/2},
\end{align*}
and rearranging,
\begin{equation}
\label{eq:balls-covering}
\mu(\lambda (\gamma) \mathcal{C}^{(k_z)})
\le C(\gamma, \varepsilon) \left(r^{2 \varepsilon}
\int_{\lambda (\gamma) \mathcal{C}^{(k_z)}}|\nabla_{\mathcal L} u|^2\,dz\right)^{\frac{Q}{Q - 2\alpha + 2\varepsilon}}t^{-\frac{2Q}{Q-2\alpha + 2\varepsilon}}.
\end{equation}
Since $d_{\alpha,\beta}(\zeta_{k_z},z)=2^{-k_z}d=2\kappa(\gamma)\,r_{k_z}$ exactly, we have $$z\in \overline{B}\big(\zeta_{k_z},2\kappa(\gamma)r_{k_z}\big)\subset \mathcal{C}\big(\zeta_{k_z},3\kappa(\gamma)r_{k_z}\big)=3\kappa(\gamma)\mathcal{C}^{(k_z)}$$, uniformly in $\gamma$, for every $z$ in the level set $\{| u|>t\}$. Moreover $C_{\triangle}(\gamma)\ge1$ and $\vartheta(\gamma)\ge1$ give $\lambda(\gamma)\ge\tfrac32\kappa(\gamma)$, so that $3\kappa(\gamma)\le2\lambda(\gamma)$ and therefore $z\in 2\lambda(\gamma)\mathcal{C}^{(k_z)}$. Now the Vitali covering theorem developed in \cite[pp.~9-10]{stein_singular_book} can be adapted to the Grushin cylinders $\mathcal{C}^{(j)}$ by noting the fact that if $\mathcal{C}^{(j)} \cap \mathcal{C}^{(j_0)} \neq \emptyset$ and $\operatorname{rad}(\mathcal{C}^{(j)})/2 \leq \operatorname{rad}(\mathcal{C}^{(j_0)})$ then
\begin{equation*}
 \mathcal{C}^{(j)} \subset \max\big\{ 5, (2^{\gamma +1}+1)^{1/\gamma}\big\} \mathcal{C}^{(j_0)}.
\end{equation*}
Thus applying the Vitali covering theorem on $\mathcal{C}(z_0,r)$ yields a countable, pairwise disjoint subcollection $\{2\lambda(\gamma)\mathcal{C}^{(i)}\}_{i\ge1}$ of Grushin cylinders of the type $2\lambda(\gamma) \mathcal{C}^{(k_z)}$ with
\begin{equation*}
\{z\in \mathcal{C}(z_0,r):| u(z)|>t\}\subset\bigcup_i 2 \max\big\{ 5, (2^{\gamma +1}+1)^{1/\gamma}\big\} \lambda(\gamma) \mathcal{C}^{(i)}.
\end{equation*}
Hence, by \eqref{eq:balls-covering} we obtain
\begin{align*}
\mu\big(\{z\in \mathcal{C}(z_0,r):| u|>t\}\big)
&\le\sum_i\mu\big(2 \max\big\{ 5, (2^{\gamma +1}+1)^{1/\gamma}\big\}\lambda(\gamma)\mathcal{C}^{(i)}\big)\le C(\gamma)\sum_i\mu(\lambda(\gamma)\mathcal{C}^{(i)})\\
&\le C(\gamma,\varepsilon) \sum_i \left(r^{2 \varepsilon}
\int_{\lambda(\gamma)\mathcal{C}^{(i)}}|\nabla_{\mathcal L} u|^2\,dz\right)^{\frac{Q}{Q - 2\alpha + 2\varepsilon}}t^{-\frac{2Q}{Q-2\alpha + 2\varepsilon}}\\
&\le C(\gamma,\varepsilon)\left(r^{2\varepsilon} \int_{ \mathcal{C}(z_0,r)}|\nabla_{\mathcal L} u|^2\,dz\right)^{\frac{Q}{Q - 2\alpha + 2\varepsilon}}t^{-\frac{2Q}{Q-2\alpha + 2\varepsilon}},
\end{align*}
where the last step uses the disjointness of the $\lambda(\gamma)\mathcal{C}^{(i)}$'s and  {Corollary~\ref{cor:support-comparison}}. This proves \eqref{eq:weaktype}. Applying Lemma~\ref{lem:measure_integrability} with $X= \mathcal{C}(z_0,r)$ and $q_{0}=2Q/(Q-2\alpha+2\varepsilon)$ we obtain \eqref{eq:subcritical-sobolev}.
\end{proof}

\section{Critical Grushin-Sobolev inequality and local compactness}
\label{sec:riesz}

In this section we prove Theorems~\ref{thm:global} and~\ref{lem:rellich}. Theorem~\ref{thm:Sobolev_subcritical} cannot be used for this, since $2^*_{\alpha,\beta,\varepsilon}<2^*_{\alpha,\beta}$ for every $\varepsilon>0$. Instead we bound $|u-u_{\mathcal C_r}|$ by a Riesz-type potential, prove a weak-type estimate for it, and pass to the strong estimate by truncation. Finally we conclude the section with a Rellich-type local compactness result.

We first state the following  Theorem~\ref{thm:maximal_fn}, namely the standard Hardy-Littlewood maximal theorem on doubling metric measure spaces (see \cite[Ch.~I]{stein_singular_book}). It remains valid in our setting, since the Vitali  covering theorem can be established for the Grushin cylinders $C_r(z)$.
\begin{theorem}[Maximal theorem]
\label{thm:maximal_fn} Let $(X,\mu)$ be a quasi-metric measure space with doubling constant $C_d$, let $\Omega\subseteq X$, and for $v\in L^1_{\mathrm{loc}}(\Omega,\mu)$ define
\[
\mathcal M_\Omega v(x):=\sup_{r>0}\frac1{\mu(B(x,r))}\int_{\Omega\cap
B(x,r)}|v|\,d\mu.
\]
Then, for every $t>0$, $\mu(\{x\in\Omega:\mathcal M_\Omega v(x)>t\})\le\frac Ct\int_\Omega|v|\,d\mu$, with $C=C(C_d)$; and for every $1<p\le\infty$, $\|\mathcal M_\Omega v\|_{L^p(\Omega)}\le C\|v\|_{L^p(\Omega)}$, with $C=C(C_d,p)$.
\end{theorem}

Now, for a cylinder $ \mathcal{C}_r\subset\R^{n+m}=:\R^N$ with $z \in  \mathcal{C}_r$ and $g\ge0$ locally integrable, define
\begin{equation}\label{eq:reisz_potential}
\mathcal J^{ \mathcal{C}_r}_{\alpha,2}g(z)
:=\sum_{2^i\le2r\max\{2,2^{1/(\alpha+\beta)}\}}
2^{i\alpha}\left(\fint_{ \mathcal{C}_{2^i}(z)}|g|^2\,dw\right)^{1/2},
\end{equation}
where the sum runs over dyadic scales $2^i$ and $ \mathcal{C}_{2^i}(z)$ denotes the anisotropic cylinder of radius $2^i$ centered at $z$.

\begin{theorem}[Representation formula]
\label{thm:representation}   {Let $u\in\mathcal H^{\mathrm{loc}}_{\alpha,\beta}(\R^{n+m})$, as in \eqref{eq:H-loc}. Then for every anisotropic Grushin cylinder $\mathcal{C}_r\subseteq\R^{n+m}$,}
\[
|u(z)-u_{ \mathcal{C}_r}|
\le C(\alpha, \beta,n,m)\,\mathcal J^{ \mathcal{C}_r}_{\alpha,2}\big(|\nabla_{\mathcal L}u|\big)(z)
\]
for a.e.\ $z\in \mathcal{C}_r$.
\end{theorem}

\begin{proof}
Let $z\in \mathcal{C}_r$ be a Lebesgue point of $u$, and let $i_0$ be the least integer with $2^{i_0}>r\max\{2,2^{1/(\alpha+\beta)}\}$, so that $\mathcal{C}_r\cap( \mathcal{C}_{2^{i_0}}(z)\setminus \mathcal{C}_{2^{i_0-1}}(z))\ne\emptyset$ and $ \mathcal{C}_r\subseteq \mathcal{C}_{2^{i_0}}(z)$. Write $u_i:=u_{ \mathcal{C}_{2^i}(z)}$. Since $u_i\to u(z)$ as $i\to-\infty$ by the Lebesgue differentiation theorem, H\"older's inequality and the Poincar\'e inequality of Theorem~\ref{thm:uniform-poincare} give

\begin{align*}
|u(z)-u_{ \mathcal{C}_r}|
&\le\big|u_{ \mathcal{C}_r}-u_{i_0}(z)\big|
+\sum_{i=-\infty}^{i_0}\big|u_i(z)-u_{i-1}(z)\big|\\
&\le\fint_{ \mathcal{C}_r}|u(w)-u_{i_0}(z)|\,dw
+\sum_{i=-\infty}^{i_0}\fint_{ \mathcal{C}_{2^{i-1}}(z)}|u(w)-u_i(z)|\,dw\\
&\le\frac{| \mathcal{C}_{2^{i_0}}|}{| \mathcal{C}_r|}
\fint_{ \mathcal{C}_{2^{i_0}}(z)}|u(w)-u_{i_0}(z)|\,dw
+\sum_{i=-\infty}^{i_0}\frac{| \mathcal{C}_{2^i}|}{| \mathcal{C}_{2^{i-1}}|}
\fint_{ \mathcal{C}_{2^i}(z)}|u(w)-u_i(z)|\,dw\\
&\le C(\alpha, \beta,n,m)\sum_{i=-\infty}^{i_0}\fint_{\mathcal{C}_{2^i}(z)}|u(w)-u_i(z)|\,dw \\
&\le C(\alpha, \beta,n,m)\sum_{i=-\infty}^{i_0}\left(\fint_{\mathcal
\mathcal{C}_{2^i}(z)}|u(w)-u_i(z)|^2\,dw\right)^{1/2}\\
&\le C(\alpha, \beta,n,m)\,\sum_{i=-\infty}^{i_0}2^{i\alpha}
\left(\fint_{ \mathcal{C}_{2^i}(z)}|\nabla_{\mathcal L}u|^2\,dw\right)^{1/2}
\le C(\alpha, \beta,n,m)\,\mathcal J^{ \mathcal{C}_r}_{\alpha,2}|\nabla_{\mathcal L}u|(z).
\end{align*}
\end{proof}

\begin{theorem}[Weak-type estimate for the Riesz potential]
\label{thm:weaktype-riesz} {Fix $z_0=(x_0,y_0)\in\R^{n+m}$ and $R>0$, set $\varkappa:=\max\{2,2^{1/(\alpha+\beta)}\}=\max\{2,\kappa(\gamma)^2\}$ with $\kappa(\gamma)$ as in Lemma~\ref{lem:ball-cylinder}, put $R_0:=2\varkappa R$, and define the set}
\begin{equation}
\label{eq:VR}
\begin{aligned}
V_R(z_0)
&:= B^n\big(x_0,R+R_0\big)\times  B^m\big(y_0,R^{\alpha+\beta}+R_0^{\alpha+\beta}\big).
\end{aligned}
\end{equation}
 {Then for every $u\in\mathcal H^{\mathrm{loc}}_{\alpha,\beta}(\R^{n+m})$ one has $\mathcal J^{\mathcal{C}_R}_{\alpha,2}|\nabla_{\mathcal L}u| \in L^{2^*_{\alpha,\beta}}_w(\mathcal{C}_R(z_0))$, and}
\begin{equation}
\label{eq:main_weak_type_riesz}
\Big|\big\{z\in \mathcal{C}_R(z_0):\mathcal J^{ \mathcal{C}_R}_{\alpha,2}|\nabla_{\mathcal L}u|(z)>t\big\}\Big|
\le C\,t^{-2^*_{\alpha,\beta}}\|\nabla_{\mathcal L}u\|_{L^2(V_R)}^{2^*_{\alpha,\beta}},
\end{equation}
where for \(1\le q<\infty\), we write \( L^q_w(\Omega)\), for the weak Lebesgue space, equipped with the quasi-norm
\[
\|f\|_{L^q_w(\Omega)}
:=
\sup_{t>0}t\,|\{z\in\Omega:|f(z)|>t\}|^{1/q}.
\]
\end{theorem}
\begin{proof}
Fix $0<r\le R_0$, $ \mathcal{C}_R =  \mathcal{C}_R(z_0)$ and split $\mathcal J^{ \mathcal{C}_R}_{\alpha,2}|\nabla_{\mathcal L}u|=J_r|\nabla_{\mathcal L}u|+J^r|\nabla_{\mathcal L}u|$, where
\[
J_r|\nabla_{\mathcal L}u|(z):=\sum_{2^i\le r}2^{i\alpha}
\left(\fint_{ \mathcal{C}_{2^i}(z)}|\nabla_{\mathcal L}u|^2\,dw\right)^{1/2},
\] and \[
J^r|\nabla_{\mathcal L}u|(z):=\sum_{r<2^i\le R_0}2^{i\alpha}
\left(\fint_{ \mathcal{C}_{2^i}(z)}|\nabla_{\mathcal L}u|^2\,dw\right)^{1/2}.
\]
For $z\in \mathcal{C}_R$, we have $\mathcal{C}_{R_0}(z) \subset V_R$. Thus summing the geometric series in $J_r$ and comparing with the maximal function of $|\nabla_{\mathcal L}u|^2$ gives
\begin{equation}
\label{eq:Jn_bound}
J_r|\nabla_{\mathcal L}u|(z)
\le\Big(\sum_{2^i\le r}2^{i\alpha}\Big)\big(\mathcal M_{V_R}|\nabla_{\mathcal
L}u|^2(z)\big)^{1/2}
\le C(\alpha)\,r^\alpha\big(\mathcal M_{V_R}|\nabla_{\mathcal
L}u|^2(z)\big)^{1/2}
\end{equation}
where $\mathcal M_{V_R}|\nabla_{\mathcal L}u|^2(z) = \displaystyle \sup_{r > 0} \frac{1}{\mu(\mathcal{C}(z,r))} \int_{V_R \cap \mathcal{C}(z,r)} |\nabla_{\mathcal L}u|^2(w) \, dw$.

Since $\alpha<Q/2$, summing the geometric series in $J^r$ instead gives
\begin{equation}
\label{eq:Jn_upper}
\begin{aligned}
J^r|\nabla_{\mathcal L}u|(z)
&=\sum_{r<2^i\le R_0}2^{i(\alpha-Q/2)}
\left(\int_{ \mathcal{C}_{2^i}(z)}|\nabla_{\mathcal L}u|^2\,dw\right)^{1/2}\\
&\le C(\alpha, \beta,n,m)\,r^{\alpha-Q/2}
\left(\int_{V_R}|\nabla_{\mathcal L}u|^2\,dw\right)^{1/2}.
\end{aligned}
\end{equation}
Combining \eqref{eq:Jn_bound}-\eqref{eq:Jn_upper},
\begin{equation}
\label{eq:J_combined}
\mathcal J^{ \mathcal{C}_R}_{\alpha,2}|\nabla_{\mathcal L}u|(z)
\le C(\alpha, \beta,n,m)\left[r^\alpha\big(\mathcal M_{V_R}|\nabla_{\mathcal
L}u|^2(z)\big)^{1/2}
+r^{\alpha-Q/2}\left(\int_{V_R}|\nabla_{\mathcal
L}u|^2\,dw\right)^{1/2}\right].
\end{equation}

The first term is dominated by the second precisely when
\begin{equation*}
r\le\left[\int_{V_R}|\nabla_{\mathcal L}u|^2\,dw
\Big/\mathcal M_{V_R}\big(|\nabla_{\mathcal L}u|^2(z)\big)\right]^{1/Q}=:r_z.
\end{equation*}
If $r_z\le R_0$, choosing $r=r_z$ in~\eqref{eq:J_combined} gives
\begin{equation}
\label{eq:J_Rbound0}
\mathcal J^{ \mathcal{C}_R}_{\alpha,2}|\nabla_{\mathcal L}u|(z)
\le C(\alpha, \beta,n,m)\left(\int_{V_R}|\nabla_{\mathcal
L}u|^2\,dw\right)^{\alpha/Q}
\big(\mathcal M_{V_R}(|\nabla_{\mathcal L}u|^2(z))\big)^{(Q-2\alpha)/2Q};
\end{equation}
if instead $r_z>R_0$, choosing $r=R_0$ gives
\begin{equation}
\label{eq:J_Rbound}
\mathcal J^{ \mathcal{C}_R}_{\alpha,2}|\nabla_{\mathcal L}u|(z)
\le C(\alpha, \beta,n,m)\,R^{\alpha-Q/2}
\left(\int_{V_R}|\nabla_{\mathcal L}u|^2\,dw\right)^{1/2}.
\end{equation}
Let $E_1,E_2\subset \mathcal{C}_R$ denote the sets on which \eqref{eq:J_Rbound0}, respectively \eqref{eq:J_Rbound}, holds, and let $ \mathcal{C}_{R,t}:=\{z\in \mathcal{C}_R:\mathcal J^{\mathcal{C}_R}_{\alpha,2}|\nabla_{\mathcal L}u|(z)>t\}$, so that $| \mathcal{C}_{R,t}|\le| \mathcal{C}_{R,t}\cap E_1|+| \mathcal{C}_{R,t}\cap E_2|$. By~\eqref{eq:J_Rbound0} and the maximal theorem (Theorem~\ref{thm:maximal_fn}),
\begin{align}
| \mathcal{C}_{R,t} &\cap E_1| \notag \\
&\le\Big|\Big\{z\in \mathcal{C}_R:\mathcal M_{V_R}(|\nabla_{\mathcal L}u|^2(z))
>C(\alpha, \beta,n,m)\,t^{2Q/(Q-2\alpha)}
\Big(\int_{V_R}|\nabla_{\mathcal L}u|^2\,dw\Big)^{-2\alpha/(Q-2\alpha)}\Big\}\Big|\notag\\
&\le C(\alpha, \beta,n,m)\,t^{-2Q/(Q-2\alpha)}
\left(\int_{V_R}|\nabla_{\mathcal L}u|^2\,dw\right)^{2\alpha/(Q-2\alpha)}
\int_{V_R}|\nabla_{\mathcal L}u|^2\,dw\notag\\
&=C(\alpha, \beta,n,m)\,t^{-2^*_{\alpha,\beta}}
\left(\int_{V_R}|\nabla_{\mathcal L}u|^2\,dw\right)^{2^*_{\alpha,\beta}/2}.
\label{eq:E1_bound}
\end{align}
By~\eqref{eq:J_Rbound}, $ \mathcal{C}_{R,t}\cap E_2=\emptyset$ once $t$ exceeds $C(\alpha,\beta)R^{\alpha-Q/2}\big(\int_{V_R}|\nabla_{\mathcal L}u|^2\,dw\big)^{1/2}$; for smaller $t$,
\begin{equation}
\label{eq:E2_bound}
| \mathcal{C}_{R,t}\cap E_2|\le| \mathcal{C}_R|=C(n, m)R^Q
\le C(n,m)\,t^{-2^*_{\alpha,\beta}}
\left(\int_{V_R}|\nabla_{\mathcal L}u|^2\,dw\right)^{2^*_{\alpha,\beta}/2}.
\end{equation}
Adding \eqref{eq:E1_bound} and \eqref{eq:E2_bound} gives \eqref{eq:main_weak_type_riesz}.
\end{proof}

Combining Theorems~\ref{thm:representation} and~\ref{thm:weaktype-riesz} yields the main weak-type estimate for $u$ itself:  {for every $z_0\in\R^{n+m}$, every $R>0$ and every $u\in\mathcal H^{\mathrm{loc}}_{\alpha,\beta}(\R^{n+m})$,}

\begin{equation}
\label{eq:main_weak_type}
\big|\{z\in \mathcal{C}_R:|u-u_{ \mathcal{C}_R}|>t\}\big|
\le C(\alpha, \beta,n,m)\,t^{-2^*_{\alpha,\beta}}
\left(\int_{V_R}|\nabla_{\mathcal L}u|^2\,dw\right)^{2^*_{\alpha,\beta}/2}.
\end{equation}

\medskip Next, to get the critical exponent inequality from this weak inequality we need to work with specific truncations of the functions and their respective estimates. Firstly, we estimate the nonlocal term truncations. To this end, for $k\in\mathbb Z$ set $t_1^{(k)}=2^{k-1}$, $t_2^{(k)}=2^k$, and $w_k:=T^{t_2^{(k)}}_{t_1^{(k)}}[v]=\min\{\max\{0,v-t_1^{(k)}\},t_2^{(k)}-t_1^{(k)}\}$, where $v \geq 0$ is a measurable function. For simplicity, we write sometimes $v_{t_1}^{t_2} := T^{t_2}_{t_1}[v]$.

\begin{lemma}[Telescoping identity]
\label{lem:telescope} For a.e.\ $(x,y)\in \mathcal{C}_r$, $v(x,y)=\sum_{k\in\mathbb Z}w_k(x,y)$.
\end{lemma}

\begin{proof}
For \(s=0\), it holds trivially. Fix $s=v(x,y)>0$ and let $k_0$ be the unique integer with $2^{k_0-1}\le s<2^{k_0}$. Then $w_k(x,y)=2^{k-1}$ for $k<k_0$, $w_{k_0}(x,y)=s-2^{k_0-1}$, and $w_k(x,y)=0$ for $k>k_0$. Summing, $\sum_{k<k_0}2^{k-1}+(s-2^{k_0-1})=s$.
\end{proof}

\begin{lemma}
\label{lem:elementary} If $(c_k)_{k\in\mathbb Z}$ are nonnegative reals with $\sum_kc_k=S<\infty$, then $\sum_{k\in\mathbb Z}c_k^2\le S^2$.
\end{lemma}

Fix $(x,y),(\xi,y)\in \mathcal{C}_r$, put $a:=v(x,y)$, $b:=v(\xi,y)$, and assume without loss of generality $a\le b$. Define the increments $c_k:=w_k(\xi,y)-w_k(x,y)$, $k\in\mathbb Z$.

\begin{lemma}
\label{lem:increments} For every $k\in\mathbb Z$, $c_k\ge0$, and $\sum_{k\in\mathbb Z}c_k=b-a$.
\end{lemma}

\begin{proof}
Each truncation $T:=T^{t_2^{(k)}}_{t_1^{(k)}}$ is non-decreasing, and $a\le b$, so $w_k(x,y)=T(a)\le T(b)=w_k(\xi,y)$, i.e.\ $c_k\ge0$. Applying the telescoping identity of Lemma~\ref{lem:telescope} at $(x,y)$ and $(\xi,y)$ and subtracting (justified since all $c_k\ge0$ and the partial sums are bounded by $b-a$) gives $\sum_kc_k=\sum_k\big(w_k(\xi,y)-w_k(x,y)\big)=v(\xi,y)-v(x,y)=b-a$.
\end{proof}

\begin{proposition}
\label{prop:sharp} For every $(x,y),(\xi,y)\in\R^{n+m}$,
\[
\sum_{k\in\mathbb Z}\big|w_k(x,y)-w_k(\xi,y)\big|^2\le\big|v(x,y)-v(\xi,y)\big|^2 .
\]
\end{proposition}

\begin{proof}
By symmetry assume $a=v(x,y)\le v(\xi,y)=b$. Since $w_k(x,y)\le w_k(\xi,y)$, Lemma~\ref{lem:increments} gives $|w_k(x,y)-w_k(\xi,y)|=c_k\ge0$ with $\sum_kc_k=b-a$, and Lemma~\ref{lem:elementary} with $S=b-a$ yields
\[
\sum_{k\in\mathbb Z}|w_k(x,y)-w_k(\xi,y)|^2
=\sum_{k\in\mathbb Z}c_k^2\le(b-a)^2=|v(x,y)-v(\xi,y)|^2.\qedhere
\]
\end{proof}

Multiplying the inequality of Proposition~\ref{prop:sharp} by $\tfrac12|x-\xi|^{-n-2\alpha}$ and integrating over all $\xi\in\R^n$ gives the same bound for the unrestricted density $D_{x,\alpha}$:

\begin{corollary}
\label{cor:truncation-D} Let $v\ge0$ be measurable on $\R^{n+m}$ with $\nabla_yv\in L^1_{\mathrm{loc}}$ slice-wise, and let $w_k:=T^{t_2^{(k)}}_{t_1^{(k)}}[v]$ be as above. Then for every measurable $\Omega\subseteq\R^{n+m}$,
\[
\sum_{k\in\mathbb Z}\int_\Omega D_{x,\alpha}w_k\,dz
\le\int_\Omega D_{x,\alpha}v\,dz,
\qquad
\sum_{k\in\mathbb Z}\int_\Omega|x|^{2\beta}|\nabla_yw_k|^2\,dz
\le\int_\Omega|x|^{2\beta}|\nabla_yv|^2\,dz ,
\]
and consequently
\[
\sum_{k\in\mathbb Z}\big\|\nabla_{\mathcal L}w_k\big\|_{L^2(\Omega)}^2
\le\big\|\nabla_{\mathcal L}v\big\|_{L^2(\Omega)}^2 .
\]
The same holds for any reindexing of the dyadic ladder, in particular for the family $\big\{v_{2^{k-2}}^{2^{k-1}}\big\}_{k\in\mathbb Z}$.
\end{corollary}
\begin{proof}
All integrands are nonnegative, so the order of summation and integration is immaterial by Tonelli's theorem. Multiplying the inequality of Proposition~\ref{prop:sharp} by $\tfrac12|x-\xi|^{-(n+2\alpha)}$ and integrating in $\xi\in\R^n$ and then in $(x,y)\in\Omega$ gives the first inequality. For the second, $\nabla_yw_k=\nabla_yv\, \chi_{\{t_1^{(k)}<v\le t_2^{(k)}\}}$ a.e., and the sets $\{t_1^{(k)}<v\le t_2^{(k)}\}$, $k\in\mathbb Z$, are pairwise disjoint. Adding the two, weighted by $1$ and $(1+\beta)^2$, and using \eqref{eq:gradient-on-omega} as an identity gives the last assertion.
\end{proof}

\begin{lemma}
\label{lem:measure_type_lemma} Let $\nu$ be a finite measure on a set $Y$ and let $w\ge0$ be $\nu$-measurable with $\nu(\{w=0\})\ge\nu(Y)/2$. Then for every $t>0$,
\[
\nu(\{w>t\})\le2\inf_{c\in\R}\nu\big(\{|w-c|>t/2\}\big).
\]
\end{lemma}
\begin{proof}
Fix $c\in\mathbb{R}$. If $c\le t/2$, then \( \{w>t\}\subset \{|w-c|>t/2\}, \) and hence \( \nu(\{w>t\}) \le \nu(\{|w-c|>t/2\}). \) On the other hand, if $c>t/2$, then \( \{w=0\}\subset \{|w-c|>t/2\}. \) Therefore,
\[
\nu(\{|w-c|>t/2\})
\ge \nu(\{w=0\})
\ge \frac{\nu(Y)}{2}
\ge \frac{\nu(\{w>t\})}{2}.
\]
Thus, in either case,
\[
\nu(\{w>t\})
\le 2\nu(\{|w-c|>t/2\}).
\]
Taking the infimum over $c\in\mathbb{R}$, which completes the proof.
\end{proof}

\begin{theorem}[Fractional Grushin-Sobolev inequality on anisotropic cylinders]
\label{thm:sobolev-cylinders} Let $0<\alpha<1$, $\beta\ge0$, $Q=n+m(\alpha+\beta)$, and $2^*_{\alpha,\beta}=2Q/(Q-2\alpha)$. Let $ \mathcal{C}_R=B_R^n\times B_{R^{\alpha+\beta}}^m \subset\R^n\times\R^m$ and let  {$u\in\mathcal H^{\mathrm{loc}}_{\alpha,\beta}(\R^{n+m})$, the set $V_R$ being as in \eqref{eq:VR}}. Then
\begin{equation}
\label{eq:Sobolev_Grushin_cylinder}
\inf_{c\in\R}\left(\int_{ \mathcal{C}_R}|u-c|^{2^*_{\alpha,\beta}}\,dz\right)^{1/2^*_{\alpha,\beta}}
\le C_2\left(\int_{V_R}|\nabla_{\mathcal L}u|^2\,dz\right)^{1/2},
\end{equation}
with $C_2=8(4C_1)^{1/2^*_{\alpha,\beta}}$, where $C_1=C(\alpha, \beta,n,m)$ is the constant in~\eqref{eq:main_weak_type}.
\end{theorem}

\begin{proof}
Choose $b\in\R$ (a median of $u$ on $ \mathcal{C}_R$) with $|\{u\ge b\}|\ge| \mathcal{C}_R|/2$ and $|\{u\le b\}|\ge| \mathcal{C}_R|/2$. Set $v_+:=\max\{u-b,0\}$, $v_-:=-\min\{u-b,0\}$, and let $v$ denote either $v_+$ or $v_-$ below. {Lemma~\ref{lem:lip-stable} gives $v_{t_1}^{t_2}\in\mathcal H^{\mathrm{loc}}_{\alpha,\beta}(\R^{n+m})$ with $|\nabla_{\mathcal L}v_{t_1}^{t_2}|\le|\nabla_{\mathcal L}u|$ a.e.; moreover $\{v_{t_1}^{t_2}=0\}\supseteq\{v=0\}$, so $|\{v_{t_1}^{t_2}=0\}|\ge|\mathcal{C}_R|/2$}. Applying Lemma~\ref{lem:measure_type_lemma} together with the weak-type estimate~\eqref{eq:main_weak_type}, we obtain
\begin{align*}
\sup_{t\ge0}\big|\{v_{t_1}^{t_2}>t\}\big|\,t^{2^*_{\alpha,\beta}}
&\le2^{2^*_{\alpha,\beta}+1}\inf_{c\in\R}\sup_{t\ge0}
\Big|\big\{|v_{t_1}^{t_2}-c|>\tfrac t2\big\}\Big|\Big(\frac t2\Big)^{2^*_{\alpha,\beta}}\\
&\le2^{2^*_{\alpha,\beta}+1}C_1
\big\|\nabla_{\mathcal L}v_{t_1}^{t_2}\big\|_{L^2(V_R)}^{2^*_{\alpha,\beta}}.
\end{align*}
By a standard dyadic decomposition, Proposition~\ref{prop:sharp} and above, we obtain

\begin{align*}
\int_{ \mathcal{C}_R}v^{2^*_{\alpha,\beta}}\,dz
&\le\sum_{k=-\infty}^\infty2^{k\cdot2^*_{\alpha,\beta}}
\big|\{2^{k-1}<v\le2^k\}\big|
\leq \sum_{k=-\infty}^\infty2^{k\cdot2^*_{\alpha,\beta}}\big|\{v\ge2^{k-1}\}\big|\\
&=\sum_{k=-\infty}^\infty2^{k\cdot2^*_{\alpha,\beta}}
\big|\{v_{2^{k-2}}^{2^{k-1}}\ge2^{k-2}\}\big|
\le2^{3\cdot2^*_{\alpha,\beta}+1}C_1
\sum_{k=-\infty}^\infty\Big(\int_{V_R}|\nabla_{\mathcal
L}v_{2^{k-2}}^{2^{k-1}}|^2\,dz\Big)^{2^*_{\alpha,\beta}/2}\\
&\le2^{3\cdot2^*_{\alpha,\beta}+1}C_1
\Big(\sum_{k=-\infty}^\infty\int_{V_R}|\nabla_{\mathcal
L}v_{2^{k-2}}^{2^{k-1}}|^2\,dz\Big)^{2^*_{\alpha,\beta}/2}\\
&\le2^{3\cdot2^*_{\alpha,\beta}+1}C_1
\big\|\nabla_{\mathcal L}v\big\|_{L^2(V_R)}^{2^*_{\alpha,\beta}}
\le2^{3\cdot2^*_{\alpha,\beta}+1}C_1
\big\|\nabla_{\mathcal L}u\big\|_{L^2(V_R)}^{2^*_{\alpha,\beta}},
\end{align*}
where the third inequality uses $2^*_{\alpha,\beta}/2\ge1$ (recall $Q>2\alpha$) together with the superadditivity of $s\mapsto s^{2^*_{\alpha,\beta}/2}$ on nonnegative summable families, the fourth is Corollary~\ref{cor:truncation-D} applied on $\Omega=V_R$ to the reindexed dyadic ladder, and the last holds  {by \eqref{eq:lip-stable-density}}.

Taking $v=v_+$ and $v=v_-$ separately and summing,
\[
\int_{ \mathcal{C}_R}|u-b|^{2^*_{\alpha,\beta}}\,dz
=\int_{ \mathcal{C}_R}v_+^{2^*_{\alpha,\beta}}\,dz
+\int_{ \mathcal{C}_R}v_-^{2^*_{\alpha,\beta}}\,dz
\le2^{3\cdot2^*_{\alpha,\beta}+2}C_1\big\|\nabla_{\mathcal
L}u\big\|_{L^2(V_R)}^{2^*_{\alpha,\beta}},
\]
and taking $1/2^*_{\alpha,\beta}$ powers gives~\eqref{eq:Sobolev_Grushin_cylinder} with $C_2=8(4C_1)^{1/2^*_{\alpha,\beta}}$.
\end{proof}

\begin{lemma}
\label{lem:exhaustion} For $u\in C_c^\infty(\R^{n+m})$, the map $R\mapsto\int_{\mathcal{C}_R}|\nabla_{\mathcal L}u|^2\,dz$ is nondecreasing on $(0,\infty)$, and
\begin{equation*}
\lim_{R\to\infty}\int_{\mathcal{C}_R}|\nabla_{\mathcal L}u|^2\,dz
=\int_{\R^{n+m}}|\nabla_{\mathcal L}u|^2\,dz\ = \,\mathcal
E_{\alpha,\beta}(u)<+\infty.
\end{equation*}
\end{lemma}


\begin{lemma}[Vanishing of the optimal centering constant]
\label{lem:vanishing-constant} Let $u\in C_c^\infty(\R^{n+m})$ and let $R_0>0$ be such that $\mathrm{supp}(u)\subset\mathcal{C}_{R_0}$. For $R\ge R_0$ define
\[
g_R(c):=\int_{\mathcal{C}_R}|u-c|^{2^*_{\alpha,\beta}}\,dz,\qquad c\in\R.
\]
Then:
\begin{enumerate}[label=(\roman*)]
\item for every $R\ge R_0$, $g_R$ attains its infimum over $\R$ at some point $c_R\in\R$;
\item $g_R(c_R)\le M:=\int_{\R^{n+m}}|u|^{2^*_{\alpha,\beta}}\,dz$ for every $R\ge R_0$;
\item $c_R\to0$ as $R\to\infty$;
\item $\displaystyle\lim_{R\to\infty}\int_{\mathcal{C}_{R_0}}|u-c_R|^{2^*_{\alpha,\beta}}\,dz=M$.
\end{enumerate}
\end{lemma}

\begin{proof}
If $u\equiv0$ the statement is trivial (with $c_R\equiv0$, $M=0$), so assume $u\not\equiv0$, and in particular $M\in(0,\infty)$ (finite because $u$ is bounded with compact support).

\smallskip\noindent Fix $R\ge R_0$. The function $c\mapsto|u(z)-c|^{2^*_{\alpha,\beta}}$ is convex for each fixed $z$ (as $2^*_{\alpha,\beta}>1$), so $g_R$ is convex, hence continuous, on $\R$. Moreover $g_R$ is coercive: by the reverse triangle inequality in $L^{2^*_{\alpha,\beta}}(\mathcal{C}_R)$,
\begin{align*}
g_R(c)^{1/2^*_{\alpha,\beta}}& =\|u-c\|_{L^{2^*_{\alpha,\beta}}(\mathcal{C}_R)} \\
& \ge\big|\,\|c\|_{L^{2^*_{\alpha,\beta}}(\mathcal{C}_R)}-\|u\|_{L^{2^*_{\alpha,\beta}}(\mathcal{C}_R)}\,\big| \\
& =\big|\,|c|\,|\mathcal{C}_R|^{1/2^*_{\alpha,\beta}}-M^{1/2^*_{\alpha,\beta}}\,\big|\xrightarrow[|c|\to\infty]{}\infty,
\end{align*}
using $\|u\|_{L^{2^*_{\alpha,\beta}}(\mathcal{C}_R)}=\|u\|_{L^{2^*_{\alpha,\beta}}(\R^{n+m})}=M^{1/2^*_{\alpha,\beta}}$ (valid since $u\equiv0$ outside $\mathcal{C}_{R_0}\subset\mathcal{C}_R$) and $|\mathcal{C}_R|<\infty$. A continuous coercive function on $\R$ attains its infimum, proving (i).

\smallskip\noindent\emph{(ii)} Since $c=0$ is an admissible competitor and $u\equiv0$ on $\mathcal{C}_R\setminus\mathcal{C}_{R_0}$,
\[
g_R(c_R)=\min_{c\in\R}g_R(c)\le g_R(0)=\int_{\mathcal{C}_R}|u|^{2^*_{\alpha,\beta}}\,dz
=\int_{\mathcal{C}_{R_0}}|u|^{2^*_{\alpha,\beta}}\,dz=\int_{\R^{n+m}}|u|^{2^*_{\alpha,\beta}}\,dz=M.
\]

\smallskip\noindent\emph{(iii)} Since $u\equiv0$ on $\mathcal{C}_R\setminus\mathcal{C}_{R_0}$ for $R\ge R_0$,
\[
g_R(c_R)=\int_{\mathcal{C}_{R_0}}|u-c_R|^{2^*_{\alpha,\beta}}\,dz+|c_R|^{2^*_{\alpha,\beta}}\big(|\mathcal{C}_R|-|\mathcal{C}_{R_0}|\big)
\ \ge\ |c_R|^{2^*_{\alpha,\beta}}\big(|\mathcal{C}_R|-|\mathcal{C}_{R_0}|\big).
\]
Combined with (ii), this gives
\begin{equation}
\label{eq:cR-bound}
|c_R|^{2^*_{\alpha,\beta}}\le\frac{M}{|\mathcal{C}_R|-|\mathcal{C}_{R_0}|}.
\end{equation}
Since, $|\mathcal{C}_R|-|\mathcal{C}_{R_0}|=C(n,m)(R^Q-R_0^Q)\to\infty$ as $R\to\infty$, so the right-hand side of \eqref{eq:cR-bound} tends to $0$, whence $c_R\to0$.

\smallskip\noindent\emph{(iv)} Consider the auxiliary function
\[
h(c):=\int_{\mathcal{C}_{R_0}}|u-c|^{2^*_{\alpha,\beta}}\,dz,\qquad c\in\R.
\]
Since $u$ is bounded on $\mathcal{C}_{R_0}$ (say $|u|\le M_0$ there) and $|\mathcal{C}_{R_0}|<\infty$, dominated convergence shows $h$ is continuous at $c=0$: for $|c|\le1$, $|u-c|^{2^*_{\alpha,\beta}}\le(M_0+1)^{2^*_{\alpha,\beta}}\mathbf1_{\mathcal{C}_{R_0}}\in L^1(\mathcal{C}_{R_0})$, and $|u(z)-c|^{2^*_{\alpha,\beta}}\to|u(z)|^{2^*_{\alpha,\beta}}$ pointwise as $c\to0$, so by dominated convergence $h(c)\to h(0)=\int_{\mathcal{C}_{R_0}}|u|^{2^*_{\alpha,\beta}}\,dz=M$ as $c\to0$. Since $c_R\to0$ by (iii), continuity of $h$ at $0$ gives
\[
\lim_{R\to\infty}\int_{\mathcal{C}_{R_0}}|u-c_R|^{2^*_{\alpha,\beta}}\,dz=\lim_{R\to\infty}h(c_R)=h(0)=M.\qedhere
\]
\end{proof}

\begin{proof}[Proof of Theorem \ref{thm:global}]
Fix $u\in C_c^\infty(\R^{n+m})$; if $u\equiv0$ there is nothing to prove, so assume $u\not\equiv0$. Let $R_0>0$ be such that $\mathrm{supp}(u)\subset\mathcal{C}_{R_0}$, and let $M:=\int_{\R^{n+m}}|u|^{2^*_{\alpha,\beta}}\,dz\in(0,\infty)$ as in Lemma~\ref{lem:vanishing-constant}. With the notation $g_R(c)=\int_{\mathcal{C}_R}|u-c|^{2^*_{\alpha,\beta}}\,dz$ of Lemma~\ref{lem:vanishing-constant} and $c_R$ a minimizer of $g_R$, inequality~\eqref{eq:Sobolev_Grushin_cylinder} reads
\begin{equation}
\label{eq:step1}
g_R(c_R)^{1/2^*_{\alpha,\beta}}=\inf_{c\in\R}\left(\int_{\mathcal{C}_R}|u-c|^{2^*_{\alpha,\beta}}\,dz\right)^{1/2^*_{\alpha,\beta}}
\le C_2\left(\int_{V_R}|\nabla_{\mathcal L}u|^2\,dz\right)^{1/2},
\qquad R>0.
\end{equation}
For every $R\ge R_0$, since $u\equiv0$ on $\mathcal{C}_R\setminus\mathcal{C}_{R_0}$,
\[
g_R(c_R)=\int_{\mathcal{C}_{R_0}}|u-c_R|^{2^*_{\alpha,\beta}}\,dz+|c_R|^{2^*_{\alpha,\beta}}\big(|\mathcal{C}_R|-|\mathcal{C}_{R_0}|\big)
\ \ge\ \int_{\mathcal{C}_{R_0}}|u-c_R|^{2^*_{\alpha,\beta}}\,dz,
\]
and by part (iv) of Lemma~\ref{lem:vanishing-constant}, the right-hand side converges to $\int_{\R^{n+m}}|u|^{2^*_{\alpha,\beta}}\,dz$ as $R\to\infty$. Hence by Lemma~\ref{lem:exhaustion} and equation \eqref{eq:step1}, as $R\to\infty$, we obtain
\[
\left( \int_{\R^{n+m}}|u|^{2^*_{\alpha,\beta}}\,dz\ \right)^{1/2_{\alpha,\beta}^*} \le\ C_2 \left(\int_{\rnn}|\nabla_{\mathcal L}u|^2\,dz\right)^{1/2}.
\]
By a standard extension argument, we can extend the above inequality to $\mathcal D^{\alpha,2}_\beta(\R^{n+m})$. This completes the proof.
\end{proof}

We end this section by proving a Rellich-Kondrachov type compact embedding in our setup for the space $\mathcal D^{\alpha,2}_\beta(\R^{n+m})$.

\begin{proof}[Proof of Theorem \ref{lem:rellich}]
Write $Q:=Q_{\alpha,\beta}$ and $\Lambda:=\sup_k\mathcal E_{\alpha,\beta}(u_k)^{1/2}<\infty$. To prove the theorem, we use the Riesz-Kolmogorov criterion given in Theorem $4.26$ of \cite{brezis_book}.

We begin with the uniform bound on the set $\mathcal A$. Theorem~\ref{thm:global} give $\|u_k\|_{L^{2^*_{\alpha,\beta}}(\R^{n+m})}\le C\Lambda$ for every $k$, so that, by H\"older's inequality and $\tfrac12-\tfrac1{2^*_{\alpha,\beta}}=\tfrac{\alpha}{Q}$,
\begin{equation}
\label{eq:rk-L2bound}
\|u_k\|_{L^2(\mathcal A)}\le\|u_k\|_{L^{2^*_{\alpha,\beta}}(\R^{n+m})}\,|\mathcal A|^{\alpha/Q}
\le C(\mathcal A,\Lambda).
\end{equation}
Now, note that any general translation on $\R^{n+m}$ can be factored as $\tau_{(h,\zeta)}=\tau^x_h\circ\tau^y_\zeta$ where $\tau^x_h$ and $\tau^y_\zeta$ are translations in $\R^n$ and $\R^m$ respectively. So, it enough to prove uniform convergence of the translations separately in $x$ and $y$ variable. Moreover, since the space $C_c^\infty(\R^{n+m})$ is dense in $\mathcal D^{\alpha,2}_\beta(\R^{n+m})$ with respect to $\mathcal E_{\alpha,\beta}(\cdot)^{1/2}$, it is enough to prove the translation estimates for $C_c^\infty(\R^{n+m})$.

Let $v\in C_c^\infty(\R^{n+m})$. First we prove the continuity of translations in the $x$-variable. To this end, write $\tau^x_hv(x,y):=v(x-h,y)$ and take the Fourier transform in $x$ and applying Plancherel for a.e.\ fixed $y$, together with the identification of Gagliardo seminorm in Fourier terms, we get
\begin{equation*}
\begin{aligned}
\|\tau^x_hv-v\|_{L^2(\R^{n+m})}^2
& =\int_{\R^m}\int_{\R^n}\big|e^{-i\xi\cdot h}-1\big|^2|\widehat v(\xi,y)|^2
\,d\xi\,dy \\
& \leq  \int_{\R^m}\int_{\R^n}\min\{4,|\xi|^2 |h|^2\}|\widehat v(\xi,y)|^2
\,d\xi\,dy \\
& \leq  \int_{\R^m}\int_{\R^n} 4^{1-\alpha}|h|^{2\alpha}|\xi|^{2\alpha}|\widehat v(\xi,y)|^2 \,d\xi\,dy \\
& \leq C(n,\alpha)\,|h|^{2\alpha} [v]_{x,\alpha}^2.
\end{aligned}
\end{equation*}
To estimate the translations in the $y$-variable we split at a level $\sigma\in(0,1]$. The set $\{|x|\ge\sigma\}$ is invariant under $y$-translations and carries $|x|^{-\beta}\le\sigma^{-\beta}$, so with $\tau^y_\zeta v(x,y):=v(x,y-\zeta)$, we obtain
\[
\|\tau^y_\zeta v-v\|_{L^2(\{|x|\ge\sigma\})}
\le|\zeta|\,\|\nabla_yv\|_{L^2(\{|x|\ge\sigma\})}
\le|\zeta|\,\sigma^{-\beta}\,\big\||x|^{\beta}\nabla_yv\big\|_{L^2(\R^{n+m})} .
\]
Let $\mathcal A_1:=\{z:\operatorname{dist}(z,\mathcal A)\le1\}$; for $|\zeta|\le1$ both $v$ and $\tau^y_\zeta v$ are supported, over $\mathcal A$, in $E_\sigma:=\{|x|<\sigma\}\cap\mathcal A_1$, whose measure is at most $C(\mathcal A)\sigma^n$, so that
\[
\|\tau^y_\zeta v-v\|_{L^2(\{|x|<\sigma\}\cap\mathcal A)}
\le2\|v\|_{L^{2^*_{\alpha,\beta}}}|E_\sigma|^{\frac12-\frac1{2^*_{\alpha,\beta}}}
\le C(\mathcal A)\,\|v\|_{L^{2^*_{\alpha,\beta}}}\,\sigma^{\frac{\alpha n}{Q}} .
\]
Adding the two contributions and invoking Theorem~\ref{thm:global},
\begin{equation*}
\|\tau^y_\zeta v-v\|_{L^2(\mathcal A)}
\le C\Big(|\zeta|\,\sigma^{-\beta}+\sigma^{\frac{\alpha n}{Q}}\Big)
\mathcal E_{\alpha,\beta}(v)^{1/2},
\qquad 0<\sigma\le1,\ |\zeta|\le1,
\end{equation*}
with $C=C(n,m,\alpha,\beta,\mathcal A)$. Given $\varepsilon>0$ one now fixes $\sigma$ with $C\sigma^{\alpha n/Q}\Lambda<\varepsilon/2$ and then $|\zeta|$ with $C|\zeta|\sigma^{-\beta}\Lambda<\varepsilon/2$; the resulting threshold depends on $\varepsilon$ and $\Lambda$, not on $k$. Thus, translations in $y$ are uniformly continuous.

Therefore, by Theorem $4.26$ of \cite{brezis_book} the family $\{u_k\}$ has a subsequence converges in $L^2(\mathcal A)$, say to $u$.

For $\delta\in(2,2^*_{\alpha,\beta})$ write $\tfrac1\delta=\tfrac{1-\theta}2+\tfrac{\theta}{2^*_{\alpha,\beta}}$ with $\theta\in(0,1)$; then
\[
\|u_k-u_j\|_{L^\delta(\mathcal A)}
\le\|u_k-u_j\|_{L^2(\mathcal A)}^{1-\theta}
\|u_k-u_j\|_{L^{2^*_{\alpha,\beta}}(\R^{n+m})}^{\theta}
\le  C \Lambda^{\theta}
\|u_k-u_j\|_{L^2(\mathcal A)}^{1-\theta}\to0,
\]
so the same subsequence is Cauchy in $L^\delta(\mathcal A)$, while for $1\le\delta<2$ the claim follows from the case $\delta=2$ and H\"older's inequality on the bounded set $\mathcal A$.
\end{proof}

\begin{corollary}[Local compactness]
Let $\{u_k\}$ be bounded in $\mathcal D^{\alpha,2}_\beta(\R^{n+m})$. Then there are a subsequence (not relabelled) and $u\in\mathcal D^{\alpha,2}_\beta(\R^{n+m})$ with $u_k\rightharpoonup u$ weakly in $\mathcal D^{\alpha,2}_\beta(\R^{n+m})$ and $u_k\to u$ a.e. in $\R^{n+m}$.
\end{corollary}
\begin{proof}
Weak subsequential convergence is immediate, $\mathcal D^{\alpha,2}_\beta(\R^{n+m})$ being a Hilbert space. Applying Theorem~\ref{lem:rellich} with $\delta=2$ on the cylinder $\mathcal C_j$ produces an $L^2(\mathcal C_j)$-convergent subsequence, hence, by \cite[Theorem 4.9]{brezis_book}, one converging a.e.\ on $\mathcal C_j$; a diagonal extraction over $j=1,2,\dots$ together with $\mathcal C_j\nearrow\R^{n+m}$ yields a single subsequence converging a.e.\ on $\R^{n+m}$, whose pointwise limit agrees with $u$ by uniqueness of distributional limits.
\end{proof}


\section{Concentration-Compactness Principle and Sobolev extremals}
\label{sec:optimal}
The embedding $\mathcal D^{\alpha,2}_\beta(\R^{n+m})\hookrightarrow L^{2^*_{\alpha,\beta}}(\R^{n+m})$ is continuous but not compact, since the Sobolev quotient is invariant under the noncompact group \eqref{scaling} below. In this section we prove Theorems~\ref{lem:ccp-mixed-fractional} and~\ref{thm:existence-minimizer-alphabeta}. \\

We recall $\mathcal D^{\alpha,2}_\beta(\R^{n+m})$ as the completion of $C_c^\infty(\R^{n+m})$ with respect to the norm
\[
\|u\|_{\mathcal D^{\alpha,2}_\beta}
:=\Big([u]_{x,\alpha}^2
+(1+\beta)^2\int_{\R^{n+m}}|x|^{2\beta}|\nabla_yu|^2\,dx\,dy\Big)^{1/2}.
\]
For brevity we write
\[
\mathcal E_{\alpha,\beta}(u):=[u]_{x,\alpha}^2
+(1+\beta)^2\int_{\R^{n+m}}|x|^{2\beta}|\nabla_yu|^2\,dx\,dy.
\]
By Theorem~\ref{thm:global}, there is a constant $C>0$ such that
\begin{equation}
\label{optimal}
\left(\int_{\R^{n+m}}|u|^{2^*_{\alpha,\beta}}\,dx\,dy\right)^{2/2^*_{\alpha,\beta}}
\le C\,\mathcal E_{\alpha,\beta}(u),
\qquad u\in C_c^\infty(\R^{n+m}).
\end{equation}
The optimal Sobolev constant in~\eqref{optimal} is
\[
S_{\alpha,\beta}(n,m):=\inf_{\substack{u\in\mathcal
D^{\alpha,2}_\beta(\R^{n+m})\\u\ne0}}
\frac{\mathcal E_{\alpha,\beta}(u)}
{\big(\int_{\R^{n+m}}|u|^{2^*_{\alpha,\beta}}\,dx\,dy\big)^{2/2^*_{\alpha,\beta}}}
=\inf_{\substack{u\in\mathcal D^{\alpha,2}_\beta(\R^{n+m})\\
\|u\|_{L^{2^*_{\alpha,\beta}}}=1}}\mathcal E_{\alpha,\beta}(u).
\]
We show that, for $\beta>0$ and all $n,m\ge1$, this infimum is attained.

The quotient defining $S_{\alpha,\beta}(n,m)$ is invariant under the translation-dilation group
\begin{equation}
\label{scaling}
u_{\rho,\eta}(x,y):=\rho^{\frac{Q_{\alpha,\beta}-2\alpha}2}u(\rho
x,\rho^{\alpha+\beta}y+\eta),
\qquad\rho>0,\ \eta\in\R^m,
\end{equation}
in the sense that $\|u_{\rho,\eta}\|_{L^{2^*_{\alpha,\beta}}(\R^{n+m})} =\|u\|_{L^{2^*_{\alpha,\beta}}(\R^{n+m})}$ and $\mathcal E_{\alpha,\beta}(u_{\rho,\eta})=\mathcal E_{\alpha,\beta}(u)$. Notice that the quotient is \emph{not} invariant under translations in the $x$-variable for $\beta > 0$, because of the presence of the weight $|x|^{2\beta}$.

We first recall some standard notation from measure theory. Let \(\Omega\subset\mathbb{R}^{n+m}\) be an open set. We denote by \(C_0(\Omega)\) the space of continuous functions \(\varphi:\Omega\to\mathbb{R}\) that vanish at infinity, endowed with the supremum norm
\[
\|\varphi\|_{\infty}:=\sup_{z\in\Omega}|\varphi(z)|.
\]
 We denote by
\[
\mathcal{M}(\Omega):=C_0(\Omega)^*
\]
the space of finite signed measures on \(\Omega\). For \(\mu\in\mathcal{M}(\Omega)\), its total variation norm is defined by
\[
\|\mu\|_{\mathcal{M}(\Omega)}
:=|\mu|(\Omega)
=
\sup\left\{
\left|\int_{\Omega}\varphi\,d\mu\right|:
\varphi\in C_0(\Omega),\
\|\varphi\|_{\infty}\leq1
\right\}.
\]
In particular, if \(\mu\) is nonnegative, then \( \|\mu\|_{\mathcal{M}(\Omega)}=\mu(\Omega). \) For a sequence \(\{\mu_k\}\subset\mathcal{M}(\Omega)\), we write \( \mu_k\stackrel{*}{\rightharpoonup}\mu ~~\text{in }\mathcal{M}(\Omega) \) if
\[
\int_{\Omega}\varphi\,d\mu_k
\longrightarrow
\int_{\Omega}\varphi\,d\mu
\qquad
\text{for every }\varphi\in C_0(\Omega).
\]
For measurable $w:\R^{n+m}\to\R$, we also recall the fractional $x$-energy density
\[
D_{x,\alpha}w(x,y):=\frac{1}2\int_{\R^n}
\frac{|w(x,y)-w(\xi,y)|^2}{|x-\xi|^{n+2\alpha}}\,d\xi,
\quad\text{so that}\quad
[w]_{x,\alpha}^2=\int_{\R^{n+m}}D_{x,\alpha}w(x,y)\,dx\,dy.
\]

We begin with a lemma describing how the energy of a bounded sequence in $\mathcal D^{\alpha,2}_\beta(\R^{n+m})$ changes when the sequence is multiplied by a cutoff function.

\begin{lemma}[Cut-off estimate]
\label{lem:ccp-cutoff} Let $\{v_k\}\subset\mathcal D^{\alpha,2}_\beta(\R^{n+m})$ be bounded with $v_k\to0$ a.e.\ in $\R^{n+m}$, and set
\[
d\mu_k:=\Big[D_{x,\alpha}v_k+(1+\beta)^2|x|^{2\beta}|\nabla_yv_k|^2\Big]dx\,dy .
\]
Let $\psi\in C^\infty(\R^{n+m})$ satisfy $0\le\psi\le1$, let $\nabla\psi$ be bounded and supported in a compact set $\mathcal K\subset\{|x|\le R_{x}\}\times\{|y|\le R_{y}\}$ for some $R_{x}, R_{y} >0$, and assume that $\psi(\cdot,y)$ is constant on $\{|x|>R_{x}\}$ for every $y\in\R^m$. Then for every $\varepsilon>0$ there is $C_\varepsilon=C_\varepsilon(\beta)>0$, independent of $k$, such that
\begin{equation}
\label{eq:ccp-cutoff}
\mathcal E_{\alpha,\beta}(\psi v_k)
\le(1+\varepsilon)\int_{\R^{n+m}}\psi^2\,d\mu_k+C_\varepsilon\,\mathcal{R}_k(\psi),
\qquad
\lim_{k\to\infty}\mathcal{R}_k(\psi)=0,
\end{equation}
where, with $K_\psi( \xi,y):=\int_{\R^n}\frac{|\psi(x,y)-\psi(\xi,y)|^2}{|x-\xi|^{n+2\alpha}}\,dx$,
\[
\mathcal{R}_k(\psi):=\int_{\R^{n+m}}|x|^{2\beta}|\nabla_y\psi|^2|v_k|^2\,dx\,dy
+\int_{\R^{n+m}}|v_k(\xi, y)|^2K_\psi(\xi,y)\,d\xi\,dy .
\]
\end{lemma}

\begin{proof}
By Young's inequality, we obtain pointwise $$|\nabla_y(\psi v_k)|^2\le(1+\varepsilon)\psi^2|\nabla_yv_k|^2+C_\varepsilon|v_k|^2|\nabla_y\psi|^2.$$ Likewise, splitting
\[
\psi(x,y)v_k(x,y)-\psi(\xi,y)v_k(\xi,y)
=\psi(x,y)\big(v_k(x,y)-v_k(\xi,y)\big)+v_k(\xi,y)\big(\psi(x,y)-\psi(\xi,y)\big)
\]
and applying Young's inequality before integrating in $(x,\xi,y)$ gives
\[
[\psi v_k]_{x,\alpha}^2
\le(1+\varepsilon)\int_{\R^{n+m}}\psi^2D_{x,\alpha}v_k\,dx\,dy
+C_\varepsilon\int_{\R^{n+m}}|v_k|^2K_\psi\,d\xi\,dy .
\]
Adding the two estimates yields first term in \eqref{eq:ccp-cutoff}.

\smallskip

Now, we show that $\mathcal{R}_k(\psi) \to 0$ as $k \to \infty$. For the first term we note that, as $\beta\ge0$, the function $|x|^{2\beta}|\nabla_y\psi|^2$ is bounded and supported in the compact set $\mathcal K$. Thus we obtain $$\int_{\R^{n+m}}|x|^{2\beta}|\nabla_y\psi|^2|v_k|^2\le C\int_{\mathcal K}|v_k|^2\to0,$$ where Theorem \ref{lem:rellich} gives $v_k\to0$ in $L^2_{\mathrm{loc}}(\R^{n+m})$. For the second term, we first estimate $K_\psi$. Since $\psi$ is smooth with bounded gradient and $0\le\psi\le1$, the mean value theorem gives
\begin{equation}\label{eq:ccp-Kdecay0}
K_\psi(\xi,y)\le\|\nabla_x\psi\|_\infty^2\int_{|h|<1}|h|^{2-n-2\alpha}dh
+\int_{|h|\ge1}|h|^{-n-2\alpha}dh=:\Gamma_0<\infty ,
\end{equation}
both integrals converging because $0<\alpha<1$. If $|y|>R_{y}$ then $\nabla\psi(\cdot,y)\equiv0$, so $\psi(\cdot,y)$ is constant and $K_\psi(\cdot,y)\equiv0$. Finally, if $|\xi|>2R_{x}$ then $\psi(x,y)-\psi(\xi,y)$ vanishes for $|x|>R_{x}$ by hypothesis, while for $|x|<R_{x}$ we have $|x-\xi|\ge|\xi|/2$ there, so
\begin{equation}
\label{eq:ccp-Kdecay}
K_\psi(\xi,y)\le\int_{|x|\le R_{x}}\frac{dx}{|x-\xi|^{n+2\alpha}}\le\frac{C(n)\,R_{x}^n}{|\xi|^{n+2\alpha}} .
\end{equation}
Now, fix $\Lambda>2R_{x}$ and split
\[
\int_{\R^{n+m}}|v_k|^2K_\psi
=\int_{\{|\xi|\le\Lambda,\,|y|\le R_{y}\}}|v_k|^2K_\psi
+\int_{\{|\xi|>\Lambda,\,|y|\le R_{y}\}}|v_k|^2K_\psi=:\mathrm{I}_k+\mathrm{II}_k .
\]
By \eqref{eq:ccp-Kdecay0} and $v_k\to0$ in $L^2_{\mathrm{loc}}(\R^{n+m})$, we deduce $$\mathrm{I}_k\le\Gamma_0\int_{\{|\xi|\le\Lambda,|y|\le R_{y}\}}|v_k|^2\to0$$ as $k\to\infty$, for each fixed $\Lambda$. \\
For $\mathrm{II}_k$ use \eqref{eq:ccp-Kdecay} and H\"older's inequality with exponents $2^*_{\alpha,\beta}/2$ and $2^*_{\alpha,\beta}/(2^*_{\alpha,\beta}-2)$; since
\[
q:=(n+2\alpha)\frac{2^*_{\alpha,\beta}}{2^*_{\alpha,\beta}-2}=(n+2\alpha)\frac{Q_{\alpha,\beta}}{2\alpha}>n,
\]
we obtain, with $C$ independent of $k$ and $\Lambda$,
\[
\mathrm{II}_k\le CR_{x}^n\|v_k\|_{L^{2^*_{\alpha,\beta}}}^2
\bigg(\int_{\{|\xi|>\Lambda,\,|y|\le R_{y}\}}|\xi|^{-q}\,d\xi\,dy\bigg)^{\frac{2^*_{\alpha,\beta}-2}{2^*_{\alpha,\beta}}}
\le C(R_{x},R_{y})\,\Lambda^{-\theta},
\]
where $\theta:=n+2\alpha-\frac{2\alpha n}{Q_{\alpha,\beta}}$ which is positive since $Q_{\alpha,\beta}>2\alpha$. Hence $\limsup_{k\to\infty}\int|v_k|^2K_\psi\le C(R_{x},R_{y})\Lambda^{-\theta}$ for every $\Lambda>2R_{x}$, and letting $\Lambda\to\infty$ gives the claim.
\end{proof}

\begin{lemma}
\label{lem:vitali-upgrade} Let $1\le r<q<\infty$ and let $\{g_k\}$ be bounded in $L^{q}(\R^{n+m})$ with $g_k\to0$ a.e. Then $g_k\to0$ in $L^{r}(\mathcal A)$ for every bounded measurable $\mathcal A\subset\R^{n+m}$.
\end{lemma}

\begin{proof}
For measurable $E\subset\mathcal A$, H\"older's inequality gives $\int_E|g_k|^{r}\le\|g_k\|_{L^{q}}^{r}|E|^{1-r/q}\le C|E|^{1-r/q}$ with $C$ independent of $k$, and $1-r/q>0$; so $\{|g_k|^{r}\}$ is uniformly integrable on $\mathcal A$, which has finite measure. Since $g_k\to0$ a.e., Vitali's convergence theorem (see e.g.\ \cite{folland_book}) applies.
\end{proof}

\begin{proof}[Proof of Theorem \ref{lem:ccp-mixed-fractional}]
Write  $d\mu_k:=\big[D_{x,\alpha}w_k+(1+\beta)^2|x|^{2\beta}|\nabla_yw_k|^2\big]dx\,dy$, so that $\mu_k(\R^{n+m})=\mathcal E_{\alpha,\beta}(w_k)$ and $\mu_k\rightharpoonup^*\mu$. The sequence $\{w_k\}$ is bounded in $\mathcal D^{\alpha,2}_\beta$ and $w_k\to0$ a.e., so Lemma~\ref{lem:ccp-cutoff} applies to it.

Let $\varphi\in C_c^\infty(\R^{n+m})$; replacing $\varphi$ by $|\varphi|$ we may assume $0\le\varphi\le1$ after normalization, and $\varphi$ satisfies the hypotheses of Lemma~\ref{lem:ccp-cutoff} with $A,B$ determined by $\operatorname{supp}\varphi$. By the definition of $S_{\alpha,\beta}(n,m)$ and \eqref{eq:ccp-cutoff},
\[
S_{\alpha,\beta}(n,m)\bigg(\int_{\R^{n+m}}|\varphi w_k|^{2^*_{\alpha,\beta}}\,dx\,dy\bigg)^{2/{2^*_{\alpha,\beta}}}
\le\mathcal E_{\alpha,\beta}(\varphi w_k)
\le(1+\varepsilon)\int_{\R^{n+m}}\varphi^2\,d\mu_k+o_k(1).
\]
Since $\varphi^2\in C_c(\R^{n+m})$ and $|\varphi|^{2^*_{\alpha,\beta}}\in C_c(\R^{n+m})$, the weak-$*$ convergences give $\int_{\R^{n+m}}\varphi^2d\mu_k\to\int_{\R^{n+m}}\varphi^2d\mu$ and $\int_{\R^{n+m}}|\varphi w_k|^{2^*_{\alpha,\beta}}\,dx\,dy\to\int_{\R^{n+m}}|\varphi|^{2^*_{\alpha,\beta}} d\nu$. Letting $k\to\infty$ and then $\varepsilon\downarrow0$,
\begin{equation}
\label{eqmeasure}
S_{\alpha,\beta}(n,m)\bigg(\int_{\R^{n+m}}|\varphi|^{2^*_{\alpha,\beta}}\,d\nu\bigg)^{2/{2^*_{\alpha,\beta}}}
\le\int_{\R^{n+m}}\varphi^2\,d\mu
\qquad\text{for all }\varphi\in C_c^\infty(\R^{n+m}).
\end{equation}
By density, \eqref{eqmeasure} extends to all $\varphi\in C_c(\R^{n+m})$.

Choose $\varphi_j\in C_c^\infty(\R^{n+m})$ with $0\le\varphi_j\le1$, $\varphi_j\equiv1$ on $B_j$ and $\varphi_j\uparrow1$ pointwise. Putting $\varphi_j$ in \eqref{eqmeasure} and letting $j\to\infty$, monotone convergence gives
\begin{equation}
\label{eqfm}
\|\nu\|^{2/{2^*_{\alpha,\beta}}}\le S_{\alpha,\beta}(n,m)^{-1}\|\mu\| .
\end{equation}

Fix $\eta\in C^\infty(\R^{n+m})$ with $0\le\eta\le1$, $\eta\equiv0$ on $B_1$ and $\eta\equiv1$ off $B_2$, and set $\eta_R:=\eta\circ\delta_{1/R}$, so that $\eta_R\equiv0$ on $B_R$ and $\eta_R\equiv1$ off $B_{2R}$. Then $\nabla\eta_R$ is supported in the compact set $\overline{B_{2R}}\setminus B_R\subset \{|x|\le2R\}\times\{|y|\le(2R)^{\alpha+\beta}\}$, and for $|y|>(2R)^{\alpha+\beta}$ one has $\eta_R(\cdot,y)\equiv1$, while for $|x|>2R$ one has $\eta_R(x,y)=1$; hence Lemma~\ref{lem:ccp-cutoff} applies to $\eta_R$ with $R_x=2R$, $R_y=(2R)^{\alpha+\beta}$. Consequently, for each fixed $R$, we have
\[
S_{\alpha,\beta}(n,m)\bigg(\int_{\R^{n+m}}|\eta_Rw_k|^{2^*_{\alpha,\beta}}\,dx\,dy\bigg)^{2/{2^*_{\alpha,\beta}}}
\le\mathcal E_{\alpha,\beta}(\eta_Rw_k)
\le(1+\varepsilon)\int_{\R^{n+m}}\eta_R^2\,d\mu_k+o_k(1).
\]
Using $\mathbf 1_{B_{2R}^c}\le\eta_R\le\mathbf 1_{B_R^c}$ on both sides and taking $\limsup_{k\to\infty}$,
\[
S_{\alpha,\beta}(n,m)\Big(\limsup_{k\to\infty}\int_{B_{2R}^c}|w_k|^{2^*_{\alpha,\beta}}\,dx\,dy\Big)^{2/{2^*_{\alpha,\beta}}}
\le(1+\varepsilon)\limsup_{k\to\infty}\mu_k\big(B_R^c\big).
\]
Letting $R\to\infty$ and then $\varepsilon\downarrow0$ yields
\begin{equation*}
\nu_\infty^{2/{2^*_{\alpha,\beta}}}\le S_{\alpha,\beta}(n,m)^{-1}\mu_\infty ,
\end{equation*}
which together with \eqref{eqfm} proves \eqref{minq}.

Since $\mathcal E_{\alpha,\beta}$ is the quadratic form of the bilinear form $\mathcal B_{\alpha,\beta}$ and $u_k=u+w_k$,
\[
\mathcal E_{\alpha,\beta}(u_k)
=\mathcal E_{\alpha,\beta}(u)+\mathcal E_{\alpha,\beta}(w_k)
+2\,\mathcal B_{\alpha,\beta}(u,w_k),
\qquad
\mathcal B_{\alpha,\beta}(u,w_k)\to0,
\]
the last limit because $w_k\rightharpoonup0$ in $\mathcal D^{\alpha,2}_\beta (\R^{n+m})$ and $\mathcal B_{\alpha,\beta}(u,\cdot)$ is a bounded linear functional there. For every $R>0$ with $\mu(\partial B_R)=0$ one has $\mu_k(B_R)\to\mu(B_R)$, whence
\[
\limsup_{k\to\infty}\mathcal E_{\alpha,\beta}(w_k)
=\limsup_{k\to\infty}\big[\mu_k(B_R)+\mu_k(B_R^c)\big]
=\mu(B_R)+\limsup_{k\to\infty}\mu_k(B_R^c).
\]
The left-hand side does not depend on $R$; letting $R\to\infty$ along such radii (all but countably many) gives $\limsup_k\mathcal E_{\alpha,\beta}(w_k)=\|\mu\|+\mu_\infty$, and \eqref{eq:E-balance} follows. Identity \eqref{eq:Lp-balance} is obtained in the same way, starting from the Brezis-Lieb lemma
\[
\int_{\R^{n+m}}|u_k|^{2^*_{\alpha,\beta}}\,dx\,dy
=\int_{\R^{n+m}}|u|^{2^*_{\alpha,\beta}}\,dx\,dy+\int_{\R^{n+m}}|w_k|^{2^*_{\alpha,\beta}}\,dx\,dy+o_k(1),
\]
which applies since $u_k\to u$ a.e.\ and $\{u_k\}$ is bounded in $L^{2^*_{\alpha,\beta}}$, and then splitting $\int_{\R^{n+m}}|w_k|^{2^*_{\alpha,\beta}}$ over $B_R$ and $B_R^c$ as above.

Assume $u=0$, $\nu\ne0$ and $S_{\alpha,\beta}(n,m)\|\nu\|^{2/{2^*_{\alpha,\beta}}}=\|\mu\|$. Taking square roots in \eqref{eqmeasure} and applying H\"older's inequality with respect to $\mu$ (exponents ${2^*_{\alpha,\beta}}/2$ and ${2^*_{\alpha,\beta}}/({2^*_{\alpha,\beta}}-2)$),
\begin{equation}
\label{eq:localized-square-root}
\begin{aligned}
S_{\alpha,\beta}(n,m)^{1/2} &\bigg(\int_{\R^{n+m}}|\varphi|^{2^*_{\alpha,\beta}}\,d\nu\bigg)^{1/{2^*_{\alpha,\beta}}}
\le\bigg(\int_{\R^{n+m}}\varphi^2\,d\mu\bigg)^{1/2} \\
&\le\|\mu\|^{\frac12-\frac{1}{2^*_{\alpha,\beta}}}\bigg(\int_{\R^{n+m}}|\varphi|^{2^*_{\alpha,\beta}}\,d\mu\bigg)^{1/{2^*_{\alpha,\beta}}},
\end{aligned}
\end{equation}
so that, raising to the power ${2^*_{\alpha,\beta}}$ and using that $\varphi\in C_c$ is arbitrary, we get $$\nu\le S_{\alpha,\beta}(n,m)^{-{2^*_{\alpha,\beta}}/2}\|\mu\|^{{2^*_{\alpha,\beta}}/2-1}\mu$$ as measures. The right-hand measure has total mass $S_{\alpha,\beta}(n,m)^{-{2^*_{\alpha,\beta}}/2}\|\mu\|^{{2^*_{\alpha,\beta}}/2}=\|\nu\|$ by the equality assumption, so the nonnegative measure $S_{\alpha,\beta}(n,m)^{-{2^*_{\alpha,\beta}}/2}\|\mu\|^{{2^*_{\alpha,\beta}}/2-1}\mu-\nu$ has zero mass and hence vanishes:
\[
\mu=S_{\alpha,\beta}(n,m)\|\nu\|^{\frac{2}{2^*_{\alpha,\beta}}-1}\nu .
\]
Substituting this into the first inequality in \eqref{eq:localized-square-root} gives
\begin{equation}
\label{eq:nu-localized-final}
\|\nu\|^{\frac12-\frac{1}{2^*_{\alpha,\beta}}}\bigg(\int_{\R^{n+m}}|\varphi|^{2^*_{\alpha,\beta}}\,d\nu\bigg)^{1/{2^*_{\alpha,\beta}}}
\le\bigg(\int_{\R^{n+m}}\varphi^2\,d\nu\bigg)^{1/2},
\qquad \varphi\in C_c(\R^{n+m}).
\end{equation}
Let $\Omega\subset\R^{n+m}$ be open and pick $\varphi_j\in C_c(\Omega)$ with $0\le\varphi_j\le1$ and $\varphi_j\uparrow\mathbf 1_\Omega$ pointwise; monotone convergence in \eqref{eq:nu-localized-final} gives $\|\nu\|^{\frac12-\frac{1}{2^*_{\alpha,\beta}}}\nu(\Omega)^{1/{2^*_{\alpha,\beta}}}\le\nu(\Omega)^{1/2}$. If $\nu(\Omega)>0$ this reads $\|\nu\|^{\frac12-\frac{1}{2^*_{\alpha,\beta}}}\le\nu(\Omega)^{\frac12-\frac{1}{2^*_{\alpha,\beta}}}$, and since $2^*_{\alpha,\beta}>2$ and $\nu(\Omega)\le\|\nu\|$ we conclude $\nu(\Omega)=\|\nu\|$. Thus
\begin{equation}
\label{eq:zero-one-law}
\nu(\Omega)\in\{0,\|\nu\|\}\qquad\text{for every open }\Omega .
\end{equation}
Finally, $\operatorname{supp}\nu$ is a single point: if $z_1\ne z_2$ both belonged to it, we could choose disjoint open sets $\Omega_1\ni z_1$, $\Omega_2\ni z_2$, and \eqref{eq:zero-one-law} would force $\nu(\Omega_1)=\nu(\Omega_2)=\|\nu\|>0$, whence $\nu(\Omega_1\cup\Omega_2)=2\|\nu\|>\|\nu\|$, a contradiction. Writing $\operatorname{supp}\nu=\{z_0\}$ we get $\nu=\nu_0\delta_{z_0}$ with $\nu_0=\|\nu\|>0$.


Finally for the last part of the proof. Fix $z_0=(x_0,y_0)$ with $x_0\ne0$.  Put $c_0:=|x_0|/2>0$ and $\mathcal A:=\{|x-x_0|<c_0\}\times\{|y-y_0|<1\}$, so that
\begin{equation}
\label{eq:frozen-weight}
c_0\le|x|\le3c_0\qquad\text{for every }(x,y)\in\mathcal A .
\end{equation}
Fix $\chi\in C_c^\infty(\mathcal A)$ with $0\le\chi\le1$ and $\chi\equiv1$ on a smaller neighbourhood $\mathcal A'\Subset\mathcal A$ of $z_0$, and set $\varphi_k:=\chi w_k$. Then $\{\varphi_k\}$ is bounded in the \emph{non-degenerate} space $\mathcal D^{\alpha,2}_0(\R^{n+m})$, whose energy is that of the same family of operators with the parameter $\beta$ replaced by $0$,
\[
\mathcal E_{\alpha,0}(v)=[v]_{x,\alpha}^2
+\int_{\R^{n+m}}|\nabla_yv|^2\,dx\,dy .
\]
Indeed, by \eqref{eq:frozen-weight},
\[
\|\nabla_y\varphi_k\|_{L^2}
\le\|\chi\nabla_yw_k\|_{L^2}+\|w_k\nabla_y\chi\|_{L^2}
\le c_0^{-\beta}\big\||x|^{\beta}\nabla_yw_k\big\|_{L^2}
+\|\nabla_y\chi\|_{\infty}\|w_k\|_{L^2(\mathcal A)},
\]
which is bounded uniformly in $k$ by \eqref{eq:rk-L2bound}, while $[\chi w_k]_{x,\alpha}$ is bounded by the localization inequality established in the first part of the proof of Lemma~\ref{lem:ccp-cutoff}, together with $\|K_\chi\|_\infty<\infty$ and the boundedness of $\{w_k\}$ in $L^2(\mathcal A)$.

However, the homogeneous dimension attached to $(\alpha,0)$ is $Q_{\alpha,0}=n+\alpha m$, and $Q_{\alpha,0}>2\alpha$ always holds. Hence Theorem~\ref{thm:global} with $\beta=0$ give
\begin{equation}
\label{eq:frozen-sobolev}
\|v\|_{L^{2^*_{\alpha,0}}(\R^{n+m})}^2
\le S_{\alpha,0}(n,m)^{-1}\,\mathcal E_{\alpha,0}(v),
\qquad
2^*_{\alpha,0}=\frac{2Q_{\alpha,0}}{Q_{\alpha,0}-2\alpha}.
\end{equation}
Since $\beta>0$ and $m\ge1$ force $Q_{\alpha,0}=n+\alpha m<n+(\alpha+\beta)m=Q$, and $t\mapsto2t/(t-2\alpha)$ is strictly decreasing on $(2\alpha,\infty)$, we obtain
\begin{equation}
\label{eq:exponent-gap}
2^*_{\alpha,\beta}<2^*_{\alpha,0} .
\end{equation}
Consequently $\{\varphi_k\}$ is bounded in $L^{2^*_{\alpha,0}}(\R^{n+m})$ by \eqref{eq:frozen-sobolev}, and $\varphi_k\to0$ a.e.\ because $w_k\to0$ a.e. Since \eqref{eq:exponent-gap} places $2^*_{\alpha,\beta}$ strictly below $2^*_{\alpha,0}$, Lemma~\ref{lem:vitali-upgrade} applies with $r=2^*_{\alpha,\beta}$ and $q=2^*_{\alpha,0}$ and gives $\int_{\mathcal A}|\chi w_k|^{2^*}\,dz\to0$. Because $\chi\equiv1$ on $\mathcal A'$ this yields $\int_{\mathcal A'}|w_k|^{2^*}\,dz\to0$, and testing the weak-$*$ convergence $|w_k|^{2^*}dz\stackrel{*}{\rightharpoonup}\nu$ against $\psi\in C_c(\mathcal A')$ with $\psi\uparrow\mathbf1_{\mathcal A'}$ gives $\nu(\mathcal A')=0$. Covering $\R^{n+m}\setminus\Sigma$ by countably many such $\mathcal A'$ completes the proof.
\end{proof}

\begin{proof}[Proof of Theorem \ref{thm:existence-minimizer-alphabeta}]
Following as in \cite[Theorem~3.4]{alves_BN_grushin}, by using the Lemma \ref{lem:ccp-mixed-fractional} with the key fact that the measure $\nu$ concentrates only on the degenerate hyperplane $\Sigma$, we have our result.
\end{proof}

\section*{Declarations}
			\noindent\textbf{Ethical Approval.} Not applicable.

			\medskip
			\noindent\textbf{Competing interests.} The authors declare that they have no competing interests.

			\medskip
			\noindent\textbf{Authors contributions.} The authors contributed equally to this work.

			\medskip
			\noindent\textbf{Availability of data and materials.} Data sharing is not applicable to this article as no new data were created or analyzed in this study.

			\section*{Acknowledgements}
			\noindent The author, S. Malhotra (PMRF ID-1403266), acknowledges the support provided by the Ministry of Education, Government of India, under the Prime Minister's Research Fellows scheme.

\bibliographystyle{abbrv}
\bibliography{references1}

@article{alves_BN_grushin,
  author  = {Alves, C. O. and Gandal, S. and Loiudice, A. and Tyagi, J.},
  title   = {A {Br\'ezis--Nirenberg} type problem for a class of degenerate elliptic problems involving the {Grushin} operator},
  journal = {J. Geom. Anal.},
  volume  = {34},
  number  = {2},
  year    = {2024},
  pages   = {Paper No. 52, 41 pp.},
  doi     = {10.1007/s12220-023-01507-3}
}

@techreport{aronszajn_frac,
  author      = {Aronszajn, N.},
  title       = {Boundary values of functions with finite {Dirichlet} integral},
  institution = {University of Kansas},
  type        = {Technical Report},
  number      = {14},
  year        = {1955},
  pages       = {77--94},
  note        = {Studies in Eigenvalue Problems}
}

@article{balhara2019,
  author  = {Balhara, R.},
  title   = {{Hardy}'s inequality for the fractional powers of the {Grushin} operator},
  journal = {Proc. Indian Acad. Sci. Math. Sci.},
  volume  = {129},
  number  = {3},
  year    = {2019},
  pages   = {Paper No. 33, 25 pp.},
  doi     = {10.1007/s12044-019-0471-2}
}

@article{baouendi_phd_thesis,
  author  = {Baouendi, M. S.},
  title   = {Sur une classe d'op\'erateurs elliptiques d\'eg\'en\'er\'es},
  journal = {Bull. Soc. Math. France},
  volume  = {95},
  year    = {1967},
  pages   = {45--87}
}

@article{bebendorf_poincare,
  author  = {Bebendorf, M.},
  title   = {A note on the {Poincar\'e} inequality for convex domains},
  journal = {Z. Anal. Anwendungen},
  volume  = {22},
  number  = {4},
  year    = {2003},
  pages   = {751--756},
  doi     = {10.4171/ZAA/1170}
}

@article{bogdan_jakubowski,
  author  = {Bogdan, K. and Jakubowski, T.},
  title   = {Estimates of heat kernel of fractional {Laplacian} perturbed by gradient operators},
  journal = {Comm. Math. Phys.},
  volume  = {271},
  number  = {1},
  year    = {2007},
  pages   = {179--198}
}

@article{bogdan_sztonyk_knopova,
  author  = {Bogdan, K. and Sztonyk, P. and Knopova, V.},
  title   = {Heat kernel of anisotropic nonlocal operators},
  journal = {Doc. Math.},
  volume  = {25},
  year    = {2020},
  pages   = {1--54}
}

@article{bourgain_brezis_mironescu,
  author  = {Bourgain, J. and Brezis, H. and Mironescu, P.},
  title   = {Limiting embedding theorems for {$W^{s,p}$} when {$s\uparrow1$} and applications},
  journal = {J. Anal. Math.},
  volume  = {87},
  year    = {2002},
  pages   = {77--101},
  note    = {Dedicated to the memory of Thomas H. Wolff},
  doi     = {10.1007/BF02868470}
}

@book{brezis_book,
  author    = {Brezis, H.},
  title     = {Functional analysis, {Sobolev} spaces and partial differential equations},
  series    = {Universitext},
  publisher = {Springer},
  address   = {New York},
  year      = {2011},
  pages     = {xiv+599},
  isbn      = {978-0-387-70913-0}
}

@article{caffarelli_silvestre,
  author  = {Caffarelli, L. and Silvestre, L.},
  title   = {An extension problem related to the fractional {Laplacian}},
  journal = {Comm. Partial Differential Equations},
  volume  = {32},
  number  = {7--9},
  year    = {2007},
  pages   = {1245--1260},
  doi     = {10.1080/03605300600987306}
}

@article{cotsiolis_tavoularis,
  author  = {Cotsiolis, A. and Tavoularis, N. K.},
  title   = {Best constants for {Sobolev} inequalities for higher order fractional derivatives},
  journal = {J. Math. Anal. Appl.},
  volume  = {295},
  number  = {1},
  year    = {2004},
  pages   = {225--236},
  doi     = {10.1016/j.jmaa.2004.03.034}
}

@article{dambrosio,
  author  = {D'Ambrosio, L.},
  title   = {{Hardy} inequalities related to {Grushin} type operators},
  journal = {Proc. Amer. Math. Soc.},
  volume  = {132},
  number  = {3},
  year    = {2004},
  pages   = {725--734}
}

@article{di_nezza_palatucci_valdinoci,
  author  = {Di Nezza, E. and Palatucci, G. and Valdinoci, E.},
  title   = {Hitchhiker's guide to the fractional {Sobolev} spaces},
  journal = {Bull. Sci. Math.},
  volume  = {136},
  number  = {5},
  year    = {2012},
  pages   = {521--573}
}

@book{folland_book,
  author    = {Folland, G. B.},
  title     = {Real analysis: Modern techniques and their applications},
  series    = {Pure and Applied Mathematics (New York)},
  edition   = {2},
  publisher = {John Wiley \& Sons},
  address   = {New York},
  year      = {1999},
  pages     = {xvi+386},
  isbn      = {0-471-31716-0}
}

@article{franchi_gutierrez_wheeden,
  author  = {Franchi, B. and Guti\'errez, C. E. and Wheeden, R. L.},
  title   = {Weighted {Sobolev--Poincar\'e} inequalities for {Grushin} type operators},
  journal = {Comm. Partial Differential Equations},
  volume  = {19},
  number  = {3--4},
  year    = {1994},
  pages   = {523--604},
  doi     = {10.1080/03605309408821025}
}

@article{franchi_lanconelli_lambda_op_holder_cty,
  author  = {Franchi, B. and Lanconelli, E.},
  title   = {{H\"older} regularity theorem for a class of linear nonuniformly elliptic operators with measurable coefficients},
  journal = {Ann. Scuola Norm. Sup. Pisa Cl. Sci. (4)},
  volume  = {10},
  number  = {4},
  year    = {1983},
  pages   = {523--541}
}

@article{franchi_lanconelli_origin_lambda_operators,
  author  = {Franchi, B. and Lanconelli, E.},
  title   = {Une m\'etrique associ\'ee \`a une classe d'op\'erateurs elliptiques d\'eg\'en\'er\'es},
  journal = {Rend. Sem. Mat. Univ. Politec. Torino},
  year    = {1983},
  pages   = {105--114},
  note    = {Conference on linear partial and pseudodifferential operators (Torino, 1982)}
}

@article{franchi_lanconelli_lambda_op_harnack,
  author  = {Franchi, B. and Lanconelli, E.},
  title   = {An embedding theorem for {Sobolev} spaces related to nonsmooth vector fields and {Harnack} inequality},
  journal = {Comm. Partial Differential Equations},
  volume  = {9},
  number  = {13},
  year    = {1984},
  pages   = {1237--1264},
  doi     = {10.1080/03605308408820362}
}

@article{franchi_lu_wheeden,
  author  = {Franchi, B. and Lu, G. and Wheeden, R. L.},
  title   = {Representation formulas and weighted {Poincar\'e} inequalities for {H\"ormander} vector fields},
  journal = {Ann. Inst. Fourier (Grenoble)},
  volume  = {45},
  number  = {2},
  year    = {1995},
  pages   = {577--604},
  doi     = {10.5802/aif.1466}
}

@article{gagliardo_frac_space,
  author  = {Gagliardo, E.},
  title   = {Propriet\`a di alcune classi di funzioni in pi\`u variabili},
  journal = {Ricerche Mat.},
  volume  = {7},
  year    = {1958},
  pages   = {102--137}
}

@article{garofalo_nhieu,
  author  = {Garofalo, N. and Nhieu, D.-M.},
  title   = {Isoperimetric and {Sobolev} inequalities for {Carnot--Carath\'eodory} spaces and the existence of minimal surfaces},
  journal = {Comm. Pure Appl. Math.},
  volume  = {49},
  number  = {10},
  year    = {1996},
  pages   = {1081--1144}
}

@article{grushin_origin,
  author  = {Grushin, V. V.},
  title   = {A certain class of hypoelliptic operators},
  journal = {Mat. Sb. (N.S.)},
  volume  = {83(125)},
  year    = {1970},
  pages   = {456--473}
}

@article{hajlasz_koskela,
  author  = {Haj{\l}asz, P. and Koskela, P.},
  title   = {{Sobolev} met {Poincar\'e}},
  journal = {Mem. Amer. Math. Soc.},
  volume  = {145},
  number  = {688},
  year    = {2000},
  pages   = {x+101},
  doi     = {10.1090/memo/0688}
}

@book{heinonen_nonlinear_potential_book,
  author    = {Heinonen, J. and Kilpel\"ainen, T. and Martio, O.},
  title     = {Nonlinear potential theory of degenerate elliptic equations},
  series    = {Oxford Mathematical Monographs},
  publisher = {Clarendon Press, Oxford University Press},
  address   = {New York},
  year      = {1993},
  pages     = {vi+363},
  note      = {Oxford Science Publications},
  isbn      = {0-19-853669-0}
}

@article{jleliSamet2017,
  author  = {Jleli, M. and Samet, B.},
  title   = {Liouville-type theorems for an elliptic system involving fractional {Laplacian} operators with mixed order},
  journal = {Electron. J. Differential Equations},
  volume  = {2017},
  number  = {105},
  year    = {2017},
  pages   = {1--11}
}

@article{lieb_sharp,
  author  = {Lieb, E. H.},
  title   = {Sharp constants in the {Hardy--Littlewood--Sobolev} and related inequalities},
  journal = {Ann. of Math. (2)},
  volume  = {118},
  number  = {2},
  year    = {1983},
  pages   = {349--374},
  doi     = {10.2307/2007032}
}

@article{lions1985,
  author  = {Lions, P.-L.},
  title   = {The concentration-compactness principle in the calculus of variations. The limit case, part 1},
  journal = {Rev. Mat. Iberoamericana},
  volume  = {1},
  number  = {1},
  year    = {1985},
  pages   = {145--201}
}

@article{loiudice,
  author  = {Loiudice, A.},
  title   = {Semilinear subelliptic problems with critical growth on {Carnot} groups},
  journal = {Manuscripta Math.},
  volume  = {124},
  number  = {2},
  year    = {2007},
  pages   = {247--259}
}

@article{monti_sobolev_grushin,
  author  = {Monti, R.},
  title   = {{Sobolev} inequalities for weighted gradients},
  journal = {Comm. Partial Differential Equations},
  volume  = {31},
  number  = {10--12},
  year    = {2006},
  pages   = {1479--1504},
  doi     = {10.1080/03605300500361594}
}

@article{monti_morbidelli,
  author  = {Monti, R. and Morbidelli, D.},
  title   = {{Kelvin} transform for {Grushin} operators and critical semilinear equations},
  journal = {Duke Math. J.},
  volume  = {131},
  number  = {1},
  year    = {2006},
  pages   = {167--202}
}

@article{Pala-Pis-Cal,
  author  = {Palatucci, G. and Pisante, A.},
  title   = {Improved {Sobolev} embeddings, profile decomposition, and concentration-compactness for fractional {Sobolev} spaces},
  journal = {Calc. Var. Partial Differential Equations},
  volume  = {50},
  number  = {3--4},
  year    = {2014},
  pages   = {799--829},
  doi     = {10.1007/s00526-013-0656-y}
}

@article{payne_weinberger,
  author  = {Payne, L. E. and Weinberger, H. F.},
  title   = {An optimal {Poincar\'e} inequality for convex domains},
  journal = {Arch. Rational Mech. Anal.},
  volume  = {5},
  year    = {1960},
  pages   = {286--292},
  doi     = {10.1007/BF00252910}
}

@article{ponce_poincare,
  author  = {Ponce, A. C.},
  title   = {An estimate in the spirit of {Poincar\'e}'s inequality},
  journal = {J. Eur. Math. Soc. (JEMS)},
  volume  = {6},
  number  = {1},
  year    = {2004},
  pages   = {1--15}
}

@article{servadei_valdinoci,
  author  = {Servadei, R. and Valdinoci, E.},
  title   = {The {Brezis--Nirenberg} result for the fractional {Laplacian}},
  journal = {Trans. Amer. Math. Soc.},
  volume  = {367},
  number  = {1},
  year    = {2015},
  pages   = {67--102}
}

@article{slobodeckii1958generalized,
  author  = {Slobodecki{\i}, L. N.},
  title   = {Generalized {Sobolev} spaces and their application to boundary problems for partial differential equations},
  journal = {Leningrad. Gos. Ped. Inst. Ucen. Zap.},
  volume  = {197},
  year    = {1958},
  pages   = {54--112}
}

@book{stein_singular_book,
  author    = {Stein, E. M.},
  title     = {Singular integrals and differentiability properties of functions},
  series    = {Princeton Mathematical Series},
  volume    = {30},
  publisher = {Princeton University Press},
  address   = {Princeton, NJ},
  year      = {1970},
  pages     = {xiv+290}
}

@article{tricomi_grushin,
  author  = {Tricomi, F. G.},
  title   = {Sulle equazioni lineari alle derivate parziali di {$2^{o}$} ordine di tipo misto},
  journal = {Atti Accad. Naz. Lincei Mem. Cl. Sci. Fis. Mat. Natur. (5)},
  volume  = {14},
  year    = {1923},
  pages   = {133--247}
}

@article{marcus_mizel,
  author  = {Marcus, Moshe and Mizel, Victor J.},
  title   = {Complete characterization of functions which act, via
              superposition, on Sobolev spaces},
  journal = {Trans. Amer. Math. Soc.},
  volume  = {251},
  year    = {1979},
  pages   = {187--218}
}

\end{document}